\documentclass[12pt,final]{article}

\usepackage{amsmath}
\usepackage{amsthm}
\usepackage{amssymb}
\usepackage{amsfonts}
\usepackage{latexsym}               
\usepackage{amsgen}
\usepackage[mathscr]{eucal}
\usepackage{mathrsfs}
\usepackage{bm}
\usepackage{enumerate}
\usepackage{cite}
\usepackage{here}

\usepackage{mathptmx} 
\usepackage{color}
\usepackage{graphicx} 

\usepackage{showkeys}

\newtheorem{lem}{Lemma}[section]
\newtheorem{thm}[lem]{Theorem}

\newtheorem{cor}[lem]{Corollary}
\newtheorem{defn}[lem]{Definition}
\newtheorem{rem}[lem]{Remark}

\newcommand{\bqn}{\begin{equation}}
\newcommand{\eqn}{\end{equation}}
\newcommand{\beqx}{\begin{equation*}}
\newcommand{\eeqx}{\end{equation*}}
\newcommand{\barr}{\begin{array}}
\newcommand{\earr}{\end{array}}
\newcommand{\beqn}{\begin{eqnarray}}
\newcommand{\eeqn}{\end{eqnarray}}
\newcommand{\beqnx}{\begin{eqnarray*}}
\newcommand{\eeqnx}{\end{eqnarray*}}
\newcommand{\bmt}{\begin{multline}}
\newcommand{\emt}{\end{multline}}

\numberwithin{equation}{section}

\newcommand{\D}{\partial}

\newcommand{\cF}{{\mathcal F}}

\newcommand{\ve}{\varepsilon}

\newcommand{\vphi}{{\varphi}}

\newcommand{\er}{\eqref}
\newcommand{\lb}{\label}
\newcommand{\qu}{\quad}

\newcommand{{\R}}{\mbox{$\mathbb R$}}

\begin{document}


\title{Traveling Waves of  the Vlasov--Poisson System}

\author{Masahiro Suzuki${}^1$ and Masahiro Takayama${}^2$}

\author{%
{\large\sc Masahiro Suzuki${}^1$},
\\
{\large\sc Masahiro Takayama${}^2$},
\\
and
{\large\sc Katherine Zhiyuan Zhang${}^3$}
}

\date{ \today \\
\bigskip
\normalsize
${}^1$%
Department of Computer Science and Engineering, 
Nagoya Institute of Technology,
\\
Gokiso-cho, Showa-ku, Nagoya, 466-8555, Japan
\\ [7pt]
${}^2$%
Department of Mathematics, 
Keio University, 
\\ 
Hiyoshi, Kohoku-ku, Yokohama, 223-8522, Japan
\\ [7pt]
${}^3$%
Courant Institute of Mathematical Sciences, New York University, \\
New York, NY 10012, USA
}

\maketitle

\begin{abstract}
We consider the Vlasov--Poisson system describing a two-species plasma with spatial dimension $1$ and the velocity variable in $\mathbb{R}^n$. 
We find the necessary and sufficient conditions for the existence of 
solitary waves, shock waves, and wave trains of the system, respectively.
To this end, we need to investigate the distribution of ions trapped by the electrostatic potential.
Furthermore, we classify completely in all possible cases whether or not the traveling wave is unique.
The uniqueness varies according to each traveling wave
when we exclude the variant caused by translation.
For the solitary wave, there are both cases that it is unique and nonunique.
The shock wave is always unique. No wave train is unique.
\end{abstract}

\begin{description}
\item[{\it Keywords:}]
Vlasov--Poisson system, Solitary Waves, Shock Waves,  Wave Trains

\item[{\it 2020 Mathematics Subject Classification:}]
35A01, 
35Q35, 
76X05 

\end{description}


\newpage

\section{Introduction}

The traveling waves of the Vlasov--Poisson--Amp\'ere system have been intensively studied in plasma physics \cite{AM1,BGK1,BG1,BD1,V.K.1,Goldman2,GIBFFS1,HRK1,MB1,PA1,SLASS1}.
The system is equivalent to the Vlasov--Poisson system:
\begin{gather}
\partial_{t} f_{\pm} + \xi_{1} \partial_{x} f_{\pm} \pm q_{\pm} \partial_{x} \phi  \partial_{\xi_{1}} f_{\pm} =0, \ \ t>0 , \ x \in \mathbb R, \ \xi \in \mathbb R^{n},
\label{eq1}
\\
\partial_{xx} \phi 
= e_{+}\int_{\mathbb R^{n}} f_{+} d\xi - e_{-}\int_{\mathbb R^{n}} f_{-}  d\xi, \ \ t>0, \ x \in \mathbb R,
\label{eq2}
\end{gather}
where $t>0$, $x \in \mathbb R$, and $\xi=(\xi_{1},\xi_{2},\ldots,\xi_{n})=(\xi_{1},\xi') \in \mathbb R \times \mathbb R^{n-1}=\mathbb R^{n}$ for $n \in \mathbb N$
are the time variable, the space variable, and velocity, respectively.
The unknown functions $f_{+} = f_{+}(t,x,\xi)$ and $f_{-} = f_{-}(t,x,\xi)$ stand for 
the velocity distribution of positive ions and electrons, respectively,
and the unknown function $-\phi=-\phi(t,x)$ stands for the electrostatic potential. 
Furthermore, $q_{\pm}$ and $e_{\pm}$ are positive constants.

We mention briefly the history of studies of the traveling waves in plasma physics.
Kampen \cite{V.K.1} completed a linear theory of the waves. 
For nonlinear traveling waves, Bohm and Gross \cite{BG1} pointed out that the difficulties of the analysis of the waves are associated with the fact that some particles are trapped by the electrostatic potential of the waves.
In 1957, Bernstein, Greene, and Kruskal \cite{BGK1} analyzed the nonlinear traveling waves in a physical regime that the effect of the trapped particles is negligible.
Nowadays the waves with no trapped particles is called as BGK waves, BGK modes, or BGK equilibria. 
Formulation and dynamics of BGK waves are studied extensively in the physics community, see, for instance, \cite{AM1,BD1,Goldman2,GIBFFS1,HRK1,MB1,PA1,SLASS1}.
Furthermore, the formation and propagation of solitary waves had been observed experimentally by Ikezi, Taylor, and Baker \cite{ITB1}.

We also review mathematical results with rigorous proofs. 
The references \cite{GS1,GS2,Z.L.1} established the existence of the BGK waves of 
the Vlasov--Poisson system and the Vlasov--Poisson--Amp\'ere system.
Furthermore, the stability and instability of the BGK waves were investigated in \cite{GL1, GS1, GS2, Z.L.1, Z.L.2, LZ1}.
For the traveling waves of the Euler--Poisson system which cannot describe trapping,
we refer the reader to the works \cite{BK1,CDMS2,HS1}.
It is of greater interest to study the traveling waves beyond the BGK waves. 

In this paper, we find the necessary and sufficient conditions of the existence of 
solitary waves, shock waves, and wave trains of the Vlasov--Poisson system. 
To this end, we need to investigate the distribution of particles trapped by the electrostatic potential.
For solitary waves and shock waves, we impose the following conditions at $x=\pm\infty$:
\begin{gather}
\lim_{x \to -\infty} f_{\pm} (t,x,\xi) =  F_{\pm}^{l}(\xi), \quad
\lim_{x \to +\infty} f_{\pm} (t,x,\xi) =  F_{\pm}^{r}(\xi), 
\quad t>0, \ \xi \in \mathbb R^{n},
\label{bc1} \\
\lim_{x \to -\infty} \phi (t,x) =  \Phi^{l}, \quad
\lim_{x \to +\infty} \phi (t,x) =  0, \quad t>0,
\label{bc2} 
\end{gather}
where we have taken a reference point of the potential $\phi$ at $x=+\infty$.
For wave trains, we impose a periodic boundary condition in the $x$-direction as usual.
Furthermore, we classify completely in all possible cases whether or not the traveling wave is unique.
The uniqueness varies according to each traveling wave,
when we exclude the variant caused by translation.
For the solitary wave, there are both cases that the solution is unique and nonunique.
The shock wave is always unique. Wave trains are never unique.

In Sections \ref{S2}, \ref{S3}, and \ref{S4}, 
we study the solitary waves, shock waves, and wave trains, respectively.
Before closing this section, we give our notation used throughout this paper.

\medskip

\noindent
{\bf Notation.} 
For $ 1 \leq p \leq +\infty$, $L^p(\Omega)$ is the Lebesgue space
equipped with the norm $\Vert\cdot\Vert_{L^{p}(\Omega)}$.
Let us denote by $(\cdot,\cdot)_{L^{2}(\Omega)}$ the inner product of $L^{2}(\Omega)$.
The notation $H^1(\Omega)$ means the Sobolev space in the $L^2$ sense.
For $1<q<+\infty$, $-\infty<a<b<+\infty$, and interval $I$, 
the function spaces $L^{q}(a,b; L^{1}(\mathbb R^{n-1}))$ and
$L^{q}_{loc}(I; L^{1}(\mathbb R^{n-1}))$ are defined by
\begin{align*}
L^{q}(a,b; L^{1}(\mathbb R^{n-1})):=&
\left\{
\begin{array}{ll}
\bigl\{ f \in L^{1}_{loc}((a,b)) \, | \, \|f\|_{L^{q}((a,b))}<+\infty  \bigr\},
& n= 1,
\\[5pt]
\bigl\{ f \in L^{1}_{loc}((a,b) \times \mathbb R^{n-1}) \, | \, \|f\|_{L^{q}(a,b;L^{1}(\mathbb R^{n-1}))}<+\infty  \bigr\}, & n\geq 2,
\end{array}
\right. 
\\
L^{q}_{loc}(I; L^{1}(\mathbb R^{n-1})):=&
\left\{
\begin{array}{ll}
\bigl\{ f \in L^{1}_{loc} (\mathbb R) \, | \, \|f\|_{L^{q}((a,b))}<+\infty \ \text{for $a,b \in I$, $a<b$} \bigr\}, & n=1,
\\[5pt]
\bigl\{ f \in L^{1}_{loc} (\mathbb R^{n}) \, | \,  \|f\|_{L^{q}(a,b; L^{1}(\mathbb R^{n-1}))}<+\infty \ \text{for $a,b \in I$, $a<b$} \bigr\}, & n\geq2,
\end{array}
\right.
\end{align*}
where
\begin{gather*}
\|f\|_{L^{q}(a,b;L^{1}(\mathbb R^{n-1}))}:=\left\{\int_{a}^{b} \left(\int_{\mathbb R^{n-1}} |f(\xi_{1},\xi')| d\xi' \right)^{q} d\xi_{1} \right\}^{\frac{1}{q}}, \quad n\geq2.
\end{gather*}
Furthermore, 
$\mathbb R_{+}^{n}:=\{\xi \in \mathbb R^{n} ; \xi_{1}>0 \}$ and $\mathbb R_{-}^{n}:=\{\xi \in \mathbb R^{n} ; \xi_{1}<0 \}$ 
stand for $n$-dimensional upper and lower half spaces, respectively.
We also use the one-dimensional indicator function $\chi(s)$ of the set $\{s>0\}$.

\section{Solitary Waves}\label{S2}

The main purpose of this section is to investigate {\it solitary waves} which
are special solutions of the problem \eqref{eq1}--\eqref{bc2} with
\begin{gather*}
F_{\pm}^{l}=F_{\pm}^{r}=F_{\pm}^{\infty}, \quad \Phi^{l}=0 . 
\end{gather*}
We look for solutions of the form 
\begin{gather*}
f_{\pm}(t,x,\xi)=F_{\pm}(x-\alpha t, \xi), \quad \phi(t,x)=\Phi(x-\alpha t) \quad \text{for some $\alpha \in \mathbb R$,}
\end{gather*}
and the potential $\Phi$ has a unique global maximum point.
It is seen by direct computations that $(F_{\pm},\Phi)=(F_{\pm}(X,\xi),\Phi(X))$ solves the following problem:
\begin{subequations}\label{VP2}
\begin{gather}
(\xi_{1}-\alpha) \partial_{X} F_{\pm} \pm q_{\pm}\partial_{X} \Phi  \partial_{\xi_{1}} F_{\pm} =0, \ \ X \in \mathbb R, \ \xi \in \mathbb R^{n},
\label{eq3}
\\
\partial_{XX} \Phi 
= e_{+}\int_{\mathbb R^{n}} F_{+} d\xi - e_{-}\int_{\mathbb R^{n}} F_{-}  d\xi , \ \ X \in \mathbb R,
\label{eq4}\\
\lim_{|X| \to+\infty} F_{\pm} (X,\xi) =  F_{\pm}^{\infty}(\xi), 
\quad \xi \in \mathbb R^{n},
\label{bc3} \\
\lim_{|X| \to+\infty} \Phi (X) =  0.
\label{bc4} 
\end{gather}
\end{subequations}
It is worth pointing out that for the case that the potential $\Phi$ has a unique global minimum point, 
we can reduce it to the case that $\Phi$ has a maximum point by replacing 
$(F_{\pm},\Phi,e_{\pm},q_{\pm})$ by $(F_{\mp},-\Phi,e_{\mp},q_{\mp})$ in \eqref{VP2}.
Therefore, we can focus on the study of the case that $\Phi$ has a unique global maximum point.
Let us give a definition of solutions of \eqref{VP2}.

\begin{defn}\label{DefS1}
We say that $(F_{\pm},\Phi)$ is a solution of the problem \eqref{VP2} if it satisfies the following:
\begin{enumerate}[(i)]
\item $F_{\pm} \in L^{1}_{loc}({\mathbb R}\times \mathbb R^{n}) \cap C({\mathbb R};L^{1}(\mathbb R^{n}))$
and $\Phi \in C^{2}({\mathbb R})$;
\item $F_{\pm}(X,\xi)\geq 0$, $\partial_{X}\Phi(0)=0$, and $\partial_{X}\Phi(X)\gtrless 0$ for $0\gtrless X$;
%
%
\item $F_{\pm}$ solve
\begin{subequations}\label{weak0}
\begin{gather}
(F_{\pm},(\xi_1-\alpha)\D_X\psi \pm q_{\pm} \D_X{\Phi}\D_{\xi_1}\psi)_{L^2({\mathbb R}\times {\mathbb R}^n)}
=0 \quad \hbox{for $\forall \psi \in C_0^1(\mathbb{R}\times \mathbb{R}^n)$},
\label{weak1}\\
\lim_{|X|\to +\infty} \| F_{\pm}(X,\cdot)-F_{\pm}^{\infty} \|_{L^{1}(\mathbb R^{n})}=0 \text{\it ;}
\label{weak2}
\end{gather}
\end{subequations}
\item $\Phi$ solves \eqref{eq4} with \eqref{bc4} in the classical sense.
\end{enumerate}
\end{defn}

The condition (ii) does not allow the variant of the solution caused by translation.
The equation \eqref{weak1} is a standard weak form of the equations \eqref{eq3}.
It is possible to replace {\it the classical sense} in the condition (iv) by {\it the weak sense}.
Indeed, a weak solution $\Phi$ of the problem of \eqref{eq4} with \eqref{bc4}
is a classical solution due to $F_{\pm} \in C({\mathbb R};L^{1}(\mathbb R^{n}))$.
Any solution $(F_{\pm},\Phi)$ satisfies the neutral condition
\begin{gather}\label{netrual2}
\int_{-\infty}^{+\infty} \left(\int_{\mathbb R^{n}} e_{+}F_{+} - e_{-}F_{-} d\xi \right) dX = 0.
\end{gather}
This equality will be shown in Lemma \ref{need1} (see also Theorem \ref{existence1}).
To solve the Poisson equation \eqref{eq4} with \eqref{bc4},
we must require the quasi-neutral condition
\begin{gather}\label{netrual1}
e_{+}\int_{\mathbb R^{n}} F_{+}^{\infty}(\xi) d\xi = e_{-}\int_{\mathbb R^{n}} F_{-}^{\infty}(\xi) d\xi.
\end{gather}

It is not easy to find a necessary and sufficient condition for the solvability of the problem \eqref{VP2}, since it has a nonlinear hyperbolic equation \eqref{eq3}. Moreover, some characteristics curves of \eqref{eq3} are closed loops (see Figure \ref{fig+} below). 
This means that some ions are trapped in bounded area by the electrostatics potential.
Therefore, the value of $F_{+}$ cannot be determined only by the end state $F_{+}^{\infty}$, 
and we need to consider degree of freedom of the distribution of the trapped ions.
Our approach is as follows. 
First we write the expected distributions $F_{\pm}$ by using the characteristics method 
with taking the freedom of the distribution of the trapped ions into account.
Then we substitute these into the Poisson equation \eqref{eq2}, 
and apply some techniques in \cite{M.S.1,MM2} to rewrite the resultant equation to a first order ordinary differential equation.
Eventually, we can reduce the problem \eqref{VP2} to the following problem of an ordinary differential equation only for $\Phi$ (for more details, see subsection \ref{S2.1}):
\begin{gather}\label{phieq2}
(\partial_{x}\Phi)^{2}=2V(\Phi;\beta,G), \quad \lim_{|X| \to+\infty} \Phi (X) =  0.
\end{gather}
The Sagdeev potential $V$ is defined 
for the maximum $\beta>0$ of the potential $\Phi$ and the distribution $G=G(\xi)\geq 0$ on the plane $\{x=0\}$ of the trapped ions   by
\begin{gather}
V=V(\Phi;\beta,G):=\int_{0}^{\Phi} \left(e_{+}\rho_{+}(\varphi;\beta,G) - e_{-}\rho_{-}(\varphi) \right)  d\vphi, \quad \Phi \in [0,\beta],
\label{V0}
\end{gather}
where the expected densities $\rho_{\pm}$ are defined by
\begin{align}
\rho_{+}(\Phi;\beta,G)
&:= \rho_{+}^{\infty}(\Phi) + \rho_{+}^{0}(\Phi;\beta,G)
\notag\\
&:=\int_{{\mathbb R}^n}F_{+}^\infty(\xi)\frac{|\xi_1-\alpha|}{\sqrt{(\xi_1-\alpha)^2+2q_{+}\Phi}}\,d\xi
\notag\\
&\quad +2\int_{\mathbb R^{n-1}}\int_{\sqrt{2q_{+}\beta-2q_{+}\Phi}+\alpha}^{\sqrt{2q_{+}\beta}+\alpha}
G(\xi)\frac{\xi_1-\alpha}{\sqrt{(\xi_1-\alpha)^2+2q_{+}\Phi-2q_{+}\beta}}\,d\xi_1\,d\xi',
\label{rho++}
\\
\rho_{-}(\Phi)
&:=\int_{{\mathbb R}^n}F_{-}^\infty(\xi)\frac{|\xi_1-\alpha|}{\sqrt{(\xi_1-\alpha)^2-2q_{-}\Phi}} \chi((\xi_1-\alpha)^2-2q_{-}{\Phi}) \,d\xi,
\label{rho--}
\end{align}
where we ignore simply the integral with respect to $\xi'$ in $\rho_{+}^{0}$ for the case $n=1$.
We will see in \eqref{rho+'} and \eqref{rho-'} below that $\rho_{\pm}$ really express the densities of the ion and electron.
The expected densities $\rho_{\pm}$ are well-defined for $F_{\pm}^{\infty} \in L^{1}_{loc}(\mathbb R^{n})$ with $F_{\pm}^{\infty}\geq 0$
and $G \in L^{1}_{loc}((\alpha,+\infty)\times \mathbb R^{n-1})$ with $G \geq 0$, since all the integrands are nonnegative. 
The properties of $\rho_{\pm}$ are summarized in the following lemma.
The proof is postponed until subsection \ref{S2.3}.
\begin{lem}\label{rhopm}
Let $\alpha \in \mathbb R$, $\beta >0$, $F_{+}^{\infty} \in L^1(\R^n)$,
$F_{-}^{\infty} \in L^1(\R^n) \cap L_{loc}^{p}({\mathbb R};L^{1}(\mathbb R^{n-1}))$, and
$G \in L^1_{loc}((\alpha,$ $+\infty) \times \R^n) \cap L_{loc}^{p}([\alpha,+\infty);L^{1}(\mathbb R^{n}))$ 
for some $p>2$. Then for any $M>0$, the following estimates hold:
\begin{gather}
 |\rho_{+}^{\infty}(\Phi)|  \leq \|F_{+}^{\infty}\|_{L^{1}(\mathbb R^{n})}, \quad \Phi \in [0,+\infty),
\label{rho+}\\
|\rho_{+}^{0}(\Phi;\beta,G)|  \leq  C\|G\|_{L^{p}(\alpha,\sqrt{2q_{+}\beta}+\alpha;L^{1}(\mathbb R^{n-1}))}, \quad \Phi \in [0,\beta],
\label{rho+1}\\
 |\rho_{-}(\Phi)|   \leq  \sqrt{2} \|F_{-}^{\infty}\|_{L^{1}(\mathbb R^{n})} +C(M)\|F_{-}^{\infty}\|_{L^{p}(-2\sqrt{q_{-}M}+\alpha,2\sqrt{q_{-}M}+\alpha;L^{1}(\mathbb R^{n-1}))},
 \quad \Phi \in [0,M],
\label{rho-}
\end{gather}
where $C$ is a positive constant depending only on $\alpha$, $\beta$, $p$, and $q_{\pm}$, and
$C(M)$ is a positive constant depending also on $M$.
Furthermore, $\rho_{+}^{\infty}(\cdot) \in C([0,+\infty))$,
$\rho_{+}^{0}(\cdot;\beta,G) \in C([0,\beta])$, and $\rho_{-}(\cdot) \in C([0,+\infty))$ hold.
\end{lem}

We are now in a position to state our main theorem.

\begin{thm} \label{existence1}
Let $p>2$, $\alpha \in \mathbb R$, $F_{+}^{\infty} \in L^1(\R^n)$,
$F_{-}^{\infty} \in L^1(\R^n) \cap L_{loc}^{p}({\mathbb R};L^{1}(\mathbb R^{n-1}))$,
and $F_{\pm}^{\infty} \geq 0$.
Assume that the necessary condition \eqref{netrual1} hold.
Suppose that $\beta >0$, $G \in L^1_{loc}((\alpha,+\infty) \times \R^{n-1}) \cap L_{loc}^{p}([\alpha,+\infty);L^{1}(\mathbb R^{n-1}))$, and $G\geq 0$. 
Then the problem \eqref{VP2} has a solution $(F_{\pm},\Phi)$ 
if and only if there exists a pair $(\beta,G)$ with
\begin{subequations}\label{G-beta}
\begin{gather}
F_{-}^\infty(\xi_1+\alpha,\xi')=F_{-}^\infty(-\xi_1+\alpha,\xi'), \quad 
(\xi_{1},\xi')\in(0,\sqrt{2q_{-}\beta})\times\R^{n-1},
\label{G-beta1}\\
V(\Phi;\beta,G)>0 \ \ \text{for $\Phi \in (0,\beta)$}, \quad 
V(\beta;\beta,G)=0,
\label{G-beta2}\\
\int_{0}^{\beta/2} \frac{d\Phi}{\sqrt{V(\Phi;\beta,G)}} =+ \infty, \quad 
\int_{\beta/2}^{\beta} \frac{d\Phi}{\sqrt{V(\Phi;\beta,G)}} <+ \infty.
\label{G-beta3}
\end{gather}
\end{subequations}
The solution satisfies the neutral condition \eqref{netrual2} and $\Phi(0)=\beta$.
Moreover, $F_{\pm}$ can be written by 
\begin{gather}\label{fform+}
\begin{aligned}
&{F}_{+}(X,\xi)
\\
& =F_{+}^\infty(-\sqrt{(\xi_1-\alpha)^2-2q_{+}{\Phi}(X)}+\alpha,\xi')
\chi((\xi_{1}-\alpha)^{2}-2q_{+}\Phi(X))\chi(-(\xi_{1}-\alpha)) \\
& \quad
+G(\sqrt{(\xi_1-\alpha)^2-2q_{+}\Phi(X)+2q_{+}\beta}+\alpha,\xi')
\chi(-(\xi_1-\alpha)^2+2q_{+}\Phi(X)) \\
& \quad
+F_{+}^\infty((\sqrt{(\xi_1-\alpha)^2-2q_{+}\Phi(X)}+\alpha,\xi')
\chi((\xi_{1}-\alpha)^{2}-2q_{+}\Phi(X))\chi(\xi_{1}-\alpha)
\end{aligned}
\end{gather}
and
\begin{gather}\label{fform-}
\begin{aligned}
F_{-}(X,\xi)
&=F_{-}^\infty(-\sqrt{(\xi_1-\alpha)^2+2q_{-}\Phi(X)}+\alpha,\xi')\chi(-(\xi_1-\alpha))
\\
&\quad +F_{-}^\infty(\sqrt{(\xi_1-\alpha)^2+2q_{-}\Phi(X)}+\alpha,\xi')\chi(\xi_1-\alpha). 
\end{aligned}
\end{gather}
\end{thm}

\begin{rem} {\rm
We do not use the information of solution to write the necessary and sufficient condition \eqref{G-beta}.
Indeed, it depends only on $e_{\pm}$, $q_{\pm}$, $F_{\pm}^{\infty}$, $\alpha$, $\beta$, and $G$. 
In subsection \ref{S2.5}, we will provide an example satisfying all the conditions in Theorem \ref{existence1}, and have solitary waves with trapped ions. 
Furthermore, the necessary and sufficient condition \eqref{G-beta} is easy to check by computers as follows. 
The condition \eqref{G-beta1} requires just symmetry with respect to the plane $\{\xi_{1}=\alpha\}$.
If $V(\Phi;\beta,G)$ is a $C^{2}$-function around $\Phi=0,\beta$,  
the condition \eqref{G-beta3} is equivalent to that $\frac{d}{d\Phi}V(0;\beta,G)=0$ and $\frac{d}{d\Phi}V(0;\beta,G)<0$ thanks to the Taylor theorem. 
Therefore, \eqref{G-beta2} and \eqref{G-beta3} are verified if the graph of $V(\Phi;\beta,G)$ is drawn as in Figure \ref{figV} below.
}
\end{rem}

\begin{figure}[H]
\begin{center}
   \includegraphics[width=8cm, bb=0 0 1720 1181]{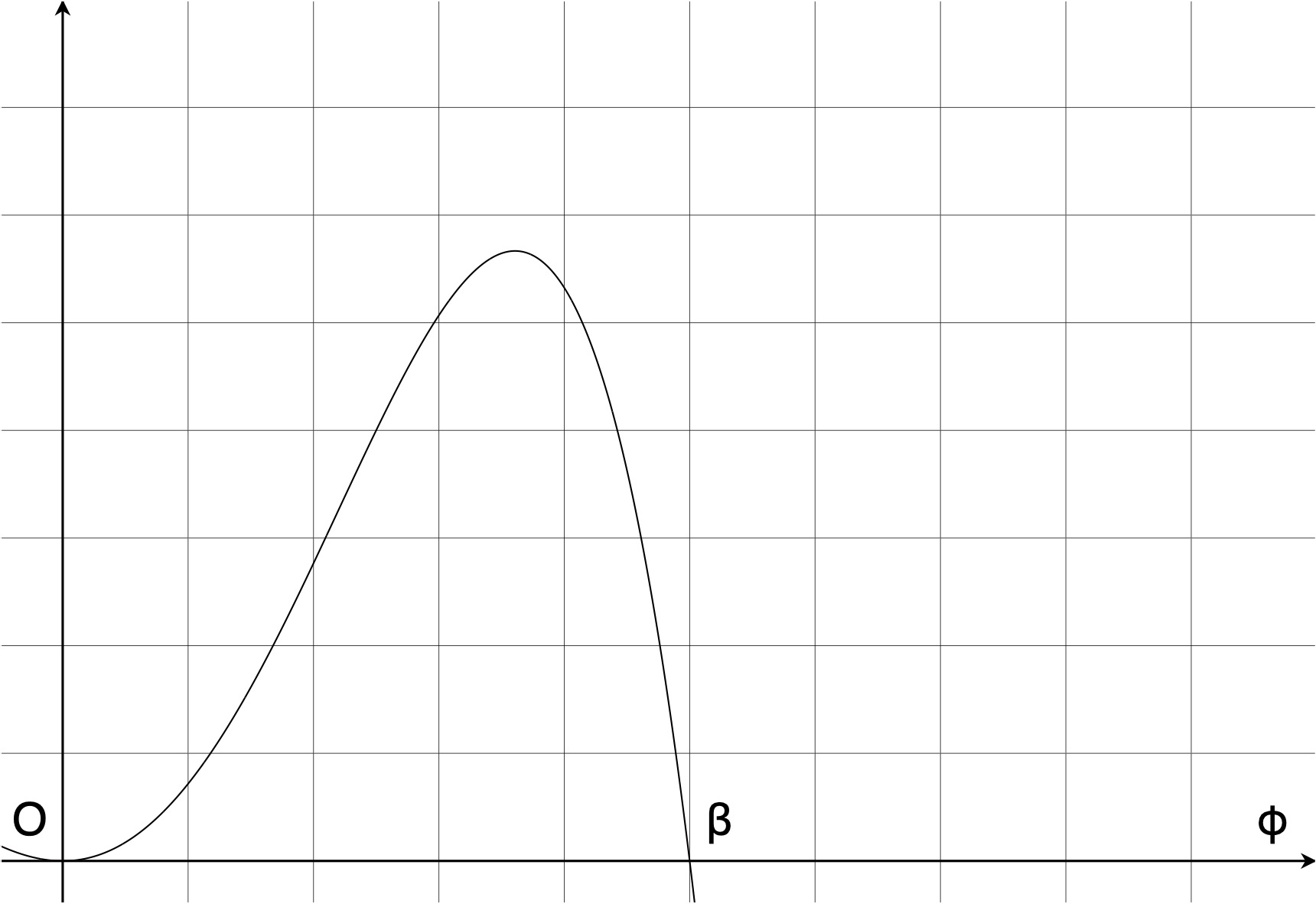}
\end{center}
  \caption{the graph of $V$}  
  \label{figV}
\end{figure}

There is the freedom of the distribution of the trapped ions.
The fact causes the nonuniqueness of solutions of the problem \eqref{VP2}.
The thing to note here is that the uniqueness also hold rarely. 
To discuss the uniqueness, we use the following notation:
\begin{gather}
\beta_{*}\!:=\!\frac{1}{2q_{-}} \!
\left(\sup\{\delta>0 \, | \, F_{-}^\infty(\xi_1\!+\!\alpha,\xi')\!=\!F_{-}^\infty(-\xi_1\!+\!\alpha,\xi'), \
(\xi_{1},\xi')\in(0,\delta)\!\times\!\R^{n-1}\} \right)^{2},
\label{beta*} \\
V^{\infty}=V^{\infty}(\Phi):=\int_{0}^{\Phi} \left(e_{+}\rho_{+}^{\infty}(\varphi) - e_{-}\rho_{-}(\varphi) \right)  d\vphi, \quad \Phi \in [0,+\infty),
\label{Vinfty}
\end{gather}
where $\rho_{+}^{\infty}$ and $\rho_{-}$ are defined in \eqref{rho++} and \eqref{rho--}, respectively.
The constant $\beta_{*}$ is the maximal kinetic energy for $F_{-}^{\infty}$ to be symmetric, and 
the function $V^{\infty}$ is the Sagdeev potential excluding the contribution of trapped ions.
The uniqueness holds only in the case that there is no trapped ion, i.e. $G\equiv0$, with the condition that {\it either} $\beta_{*}$ matches the maximum $\beta$ of the electrostatic potential {\it or} that $V^{\infty}$ is nonnegative. 

We state our main theorem on the uniqueness and nonuniqueness in Theorem \ref{uniqueness0} below.

\begin{thm} \label{uniqueness0}
Let $p>2$, $\alpha \in \mathbb R$, $F_{+}^{\infty} \in L^1(\R^n)$,
$F_{-}^{\infty} \in L^1(\R^n) \cap L_{loc}^{p}({\mathbb R};L^{1}(\mathbb R^{n-1}))$, and $F_{\pm}^{\infty}\geq 0$.
Suppose that the necessary condition \eqref{netrual1} hold, 
and there exists a pair $(\beta,G)$ with \eqref{G-beta}, 
$\beta >0$, $G \in L^1_{loc}((\alpha,+\infty) \times \R^{n-1}) \cap L_{loc}^{p}([\alpha,+\infty);L^{1}(\mathbb R^{n-1}))$, and $G\geq 0$.
Then the solution $(F_{\pm},\Phi)$ of the problem \eqref{VP2} is unique 
if and only if either the following condition (i) or (ii) holds:
\begin{enumerate}[(i)]
\item $G\equiv 0$ on $(\alpha,\sqrt{2q_{+}\beta}+\alpha)\times \mathbb R^{n-1}$ and $\beta=\beta_{*}$;
\item $G\equiv 0$ on $(\alpha,\sqrt{2q_{+}\beta}+\alpha)\times \mathbb R^{n-1}$ and ${V}^{\infty}(\Phi)\geq 0$ for \footnote{We note that $\beta \leq \beta_{*}$ holds due to \eqref{G-beta1}, and also suppose $\beta < \beta_{*}$ without loss of generality.}$\Phi \in (\beta,\beta_{*})$. 
\end{enumerate}
If neither condition (i) nor (ii) holds, there exist infinite many solutions of the problem \eqref{VP2}.
\end{thm}

In plasma physics, it is often assumed that the electron density obeys the Boltzmann relation $\rho_{-}:=\int_{\mathbb R^{n}} f_{-} d\xi=\rho e^{-\kappa\phi}$ in \eqref{eq2} to reduce the transport equations for $f_{-}$, where $\rho$ and $\kappa$ are some positive constants.
The corresponding solitary waves solve
\begin{subequations}\label{VP3}
\begin{gather}
(\xi_{1}-\alpha) \partial_{X} F_{+} + q_{+}\partial_{X} \Phi  \partial_{\xi_{1}} F_{+} =0, \ \ X \in \mathbb R, \ \xi \in \mathbb R^{n},
\label{eq5}
\\
\partial_{XX} \Phi 
= e_{+}\int_{\mathbb R^{n}} F_{+} d\xi - e_{-}\rho e^{-\kappa \Phi}, \ \ X \in \mathbb R,
\label{eq6}\\
\lim_{|X| \to+\infty} F_{+} (X,\xi) =  F_{+}^{\infty}(\xi), 
\quad \xi \in \mathbb R^{n},
\label{bc5} \\
\lim_{|X| \to+\infty} \Phi (X) =  0.
\label{bc6} 
\end{gather}
\end{subequations}
Theorems \ref{existence1} and \ref{uniqueness0} are also applicable to the problem \eqref{VP3} by suitably choosing $F_{-}^{\infty}$ in the problem \eqref{VP2}. 
Indeed, the following corollary holds. 

\begin{cor}\label{cor1}
Suppose that $(F_{\pm},\Phi)$ is a solution of the problem \eqref{VP2} with 
\begin{gather*}
F_{-}^{\infty}(\xi)=\frac{\rho}{\displaystyle \int_{\mathbb R^{n}} e^{\frac{-\kappa}{2q_{-}}|\xi|^2} d\xi}
e^{\frac{-\kappa}{2q_{-}}\{(\xi_{1}-\alpha)^{2}+|\xi'|^{2}\}}.
\end{gather*}
Then $(F_{+},\Phi)$ is a solution of the problem \eqref{VP3}.
\end{cor}
\begin{proof}
Let $(F_{\pm},\Phi)$ be a solution of the problem \eqref{VP2} with $F_{-}^{\infty}(\xi)$ in the corollary.
From Theorem~\ref{existence1}, we see that $F_{-}$ is written as \eqref{fform-}. 
It is seen by direct computation that
\begin{gather*}
\int_{\mathbb R^{n}} F_{-}(X,\xi) d\xi
= \frac{\rho}{\displaystyle \int_{\mathbb R^{n}} e^{\frac{-\kappa}{2q_{-}}|\xi|^2} d\xi}
 \int_{\mathbb R^{n}} e^{\frac{-\kappa}{2q_{-}}\{(\xi_1-\alpha)^2+|\xi'|^{2}+2q_{-}\Phi(X)\}} d\xi
=\rho e^{-\kappa \Phi(X)}.
\end{gather*}
This means that $(F_{+},\Phi)$ also solves \eqref{VP3}.
The proof is complete.
\end{proof}

This section is organized as follows. 
In subsection \ref{S2.1}, we investigate necessary conditions for the solvability of the problem \eqref{VP2}.
Subsection \ref{S2.2} gives the proof of the solvability stated in Theorem \ref{existence1}.
We show Lemma \ref{rhopm} in subsection \ref{S2.3}.
Subsection \ref{S2.4} deals with the proof of the uniqueness and nonuniqueness stated in Theorem \ref{uniqueness0}.
Subsection \ref{S2.5} provides an example with all the conditions in Theorem \ref{existence1}.

\subsection{Necessary conditions of solvability}\label{S2.1}

In this section, we investigate necessary conditions for the solvability of the problem \eqref{VP2}.

\begin{lem}\label{need1}
Let $\alpha \in \mathbb R$, $F_{\pm}^{\infty} \in L^{1}(\mathbb R^{n})$, and $F_{\pm}^{\infty}\geq 0$.
Suppose that  the problem \eqref{VP2} has a solution $(F_{\pm},\Phi)$.
Then the conditions \eqref{netrual2} and \eqref{netrual1} hold;
$F_{-}^{\infty}$ satisfies the condition \eqref{G-beta1} with $\beta=\Phi(0)$; 
$F_{\pm}$ are written as \eqref{fform+} and \eqref{fform-} with $\beta=\Phi(0)$ and $G(\xi)=F_{+}(0,\xi)$;
$\rho_{\pm} \in C([0,\Phi(0)])$; 
$\Phi$ solves the problem \eqref{phieq2} with $\beta=\Phi(0)$ and $G(\xi)=F_{+}(0,\xi)$;
$V$ satisfies \eqref{G-beta2} and \eqref{G-beta3} with $\beta=\Phi(0)$ and  $G(\xi)=F_{+}(0,\xi)$.
\end{lem}

\begin{rem} \rm
The regularity of $F_{\pm}^{\infty}$ in Lemma \ref{need1} are weaker than that of Theorem \ref{existence1}.
\end{rem}

\begin{proof}
We first show \eqref{netrual2}.
Owing to \eqref{weak2} and $F_{\pm} \in C({\mathbb R}; L^{1}(\mathbb R^{n}))$, 
it follows from \eqref{eq4} that $\D_{XX} \Phi$ is bounded and 
therefore $\D_{X} \Phi$ is uniformly continuous on ${\mathbb R}$.
Then this fact together with \eqref{bc4} and the condition (ii) in Definition \ref{DefS1} ensures that
\begin{gather}\label{lim1}
\lim_{|X| \to +\infty}\D_{X}\Phi(X)=0.
\end{gather}
Integrating \eqref{eq4} from $-\infty$ to $+\infty$ and using \eqref{lim1}, we arrive at \eqref{netrual2}.
Let us also show \eqref{netrual1}. It is seen from \eqref{netrual2} 
that there exists a sequence $\{X_{k}\}_{k=1}^{+\infty}$ such that $X_{k} \to +\infty$
and $\lim_{k \to +\infty}(e_{+}\int_{\mathbb R^{n}}F_{+} d\xi-e_{-}\int_{\mathbb R^{n}}F_{-} d\xi)(X_{k})=0$. 
On the other hand, it follows from \eqref{weak2} that
$\lim_{X \to +\infty}(e_{+}\int_{\mathbb R^{n}}F_{+} d\xi-e_{-}\int_{\mathbb R^{n}}F_{-} d\xi)=e_{+}\int_{\mathbb R^{n}}F_{+}^{\infty} d\xi - e_{-}\int_{\mathbb R^{n}}F_{-}^{\infty} d\xi$.
Therefore, \eqref{netrual1} must hold.

Next we show that $F_{-}^{\infty}$ satisfies the condition \eqref{G-beta1} with $\beta=\Phi(0)$,
and $F_{\pm}$ are written as \eqref{fform+} and \eqref{fform-} with $\beta=\Phi(0)$ and  $G(\xi)=F_{+}(0,\xi)$.
Regarding $\Phi$ as a given function and then applying the characteristics method to \eqref{weak1}, 
we see that the values of $F_{+}$ must be the same on the following characteristics curve:
\begin{gather*}
\frac{1}{2}(\xi_{1}-\alpha)^{2}-q_{+}\Phi(X)=c,
\end{gather*}
where $c$ is some constant. We draw the illustration of characteristics for $\alpha=0$ in Figure \ref{fig+} below.
\begin{figure}[H]
\begin{center}
    \includegraphics[width=9.5cm, bb=0 0 1720 1104]{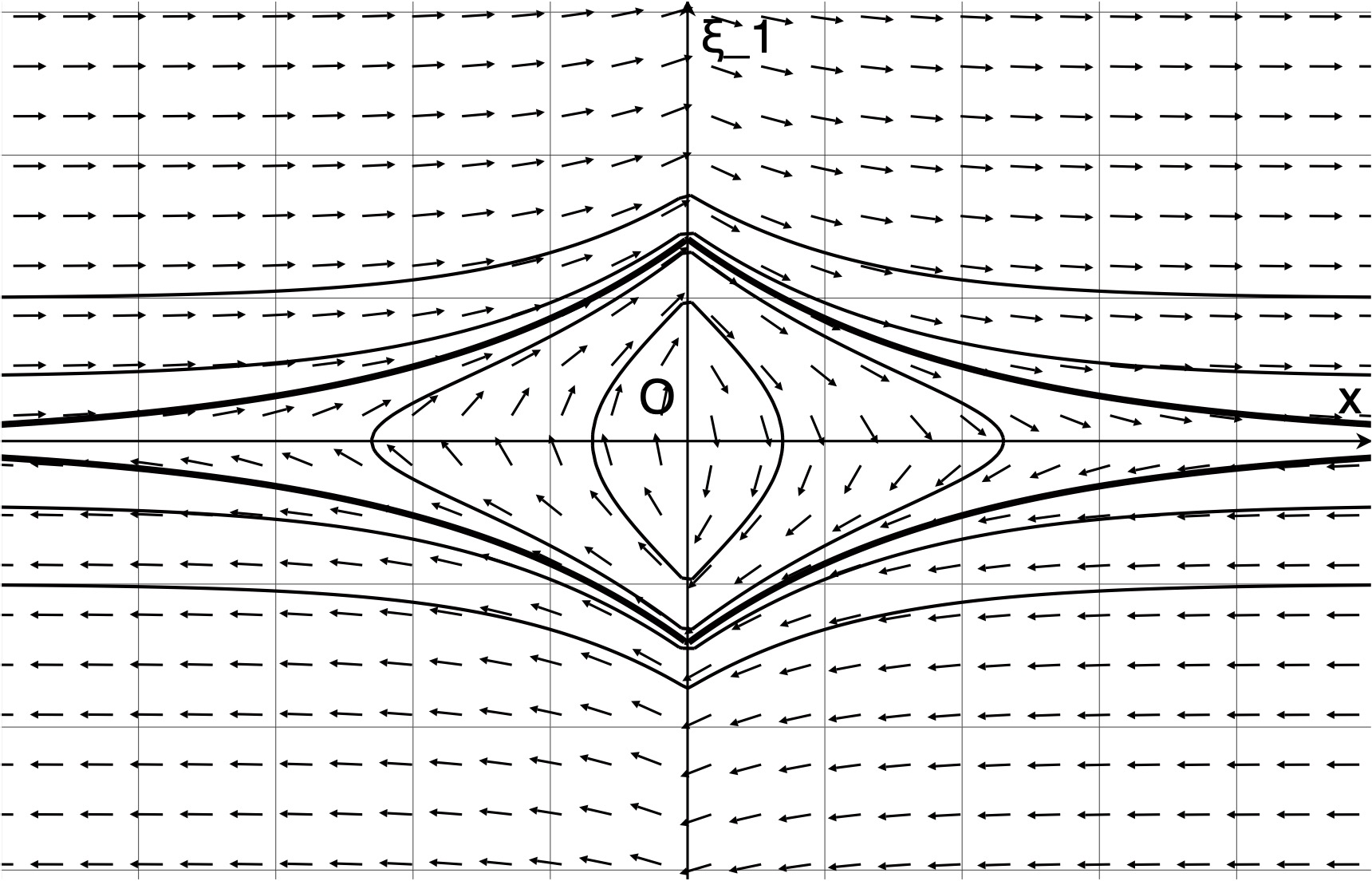}
\end{center}
  \caption{the characteristics of $F_{+}$}  
  \label{fig+}
\end{figure}
It tells us that
\begin{align*}
F_{+}(y,-\sqrt{\eta_{1}^{2}+2q_{+}\Phi(y)}+\alpha,\eta')=F_{+}^{\infty}(\eta_{1}+\alpha,\eta'), &\quad (y,\eta_{1},\eta') \in {\cal X}_{+}^{1},
 \\
F_{+}(y,\pm\sqrt{\eta_1^2-2q_{+}{\Phi}(0)+2q_{+}{\Phi}(y)}+\alpha,\eta')=F_{+}(0,\eta_{1}+\alpha,\eta'), &\quad (y,\eta_{1},\eta') \in {\cal X}_{+}^{2},
 \\
F_{+}(y,\sqrt{\eta_{1}^{2}+2q_{+}\Phi(y)}+\alpha,\eta')=F_{+}^{\infty}(\eta_{1}+\alpha,\eta'), &\quad  (y,\eta_{1},\eta') \in {\cal X}_{+}^{3},
\end{align*}
where $\eta=(\eta_{1},\eta_{2},\ldots,\eta_{n})=(\eta_{1},\eta')$ and
\begin{align*}
{\cal X}_{+}^{1}&:={\R}\times\R_-^n, 
\\
{\cal X}_{+}^{2}&:=\{(y,\eta_1,\eta')\in {\R}\times\R^{n}\;|\; 2q_{+}{\Phi}(0)-2q_{+}{\Phi}(y)<\eta_1^2< 2q_{+}\Phi(0)\}, 
\\
{\cal X}_{+}^{3}&:={\R}\times\R_+^n. 
\end{align*}
Furthermore, we conclude from these three equalities that $F_{+}$ must be written as \eqref{fform+} with $\beta=\Phi(0)$ and  $G(\xi)=F_{+}(0,\xi)$, i.e.
\begin{align*}
&{F}_{+}(X,\xi)
\\
& =F_{+}^\infty(-\sqrt{(\xi_1-\alpha)^2-2q_{+}{\Phi}(X)}+\alpha,\xi')
\chi((\xi_{1}-\alpha)^{2}-2q_{+}\Phi(X))\chi(-(\xi_{1}-\alpha)) \\
& \quad
+F_{+}(0,\sqrt{(\xi_1-\alpha)^2-2q_{+}\Phi(X)+2q_{+}\Phi(0)}+\alpha,\xi')
\chi(-(\xi_1-\alpha)^2+2q_{+}\Phi(X)) \\
& \quad
+F_{+}^\infty(\sqrt{(\xi_1-\alpha)^2-2q_{+}\Phi(X)}+\alpha,\xi')
\chi((\xi_{1}-\alpha)^{2}-2q_{+}\Phi(X))\chi(\xi_{1}-\alpha),
\end{align*}
where $\chi(s)$ is the one-dimensional indicator function of the set $\{s>0\}$.

Next we treat $F_{-}$. The values of $F_{-}$ must be the same on the following characteristics curve:
\begin{gather*}
\frac{1}{2}(\xi_{1}-\alpha)^{2}+q_{-}\Phi(X)=c,
\end{gather*}
where $c$ is some constant. We draw the illustration of characteristics for $\alpha=0$ in Figure \ref{fig-} below.
\begin{figure}[H]
\begin{center}
    \includegraphics[width=9.5cm, bb=0 0 1720 1337]{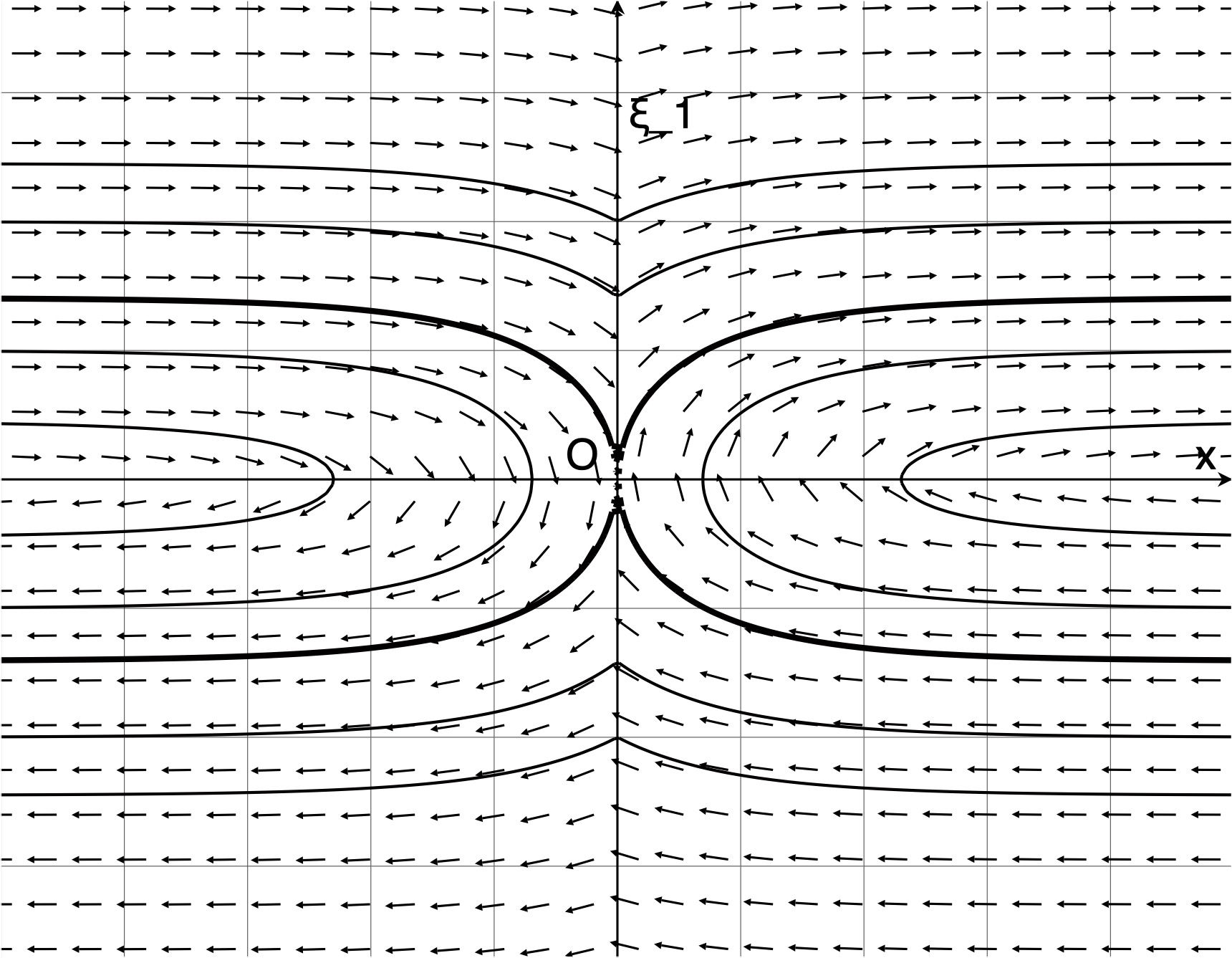}
\end{center}
  \caption{the characteristics of $F_{-}$}  
  \label{fig-}
\end{figure}
It tells us that
\begin{align*}
F_{-}(y,-\sqrt{\eta_{1}^{2}-2q_{-}\Phi(y)}+\alpha,\eta')=F_{-}^{\infty}(\eta_{1}+\alpha,\eta'), &\quad (y,\eta_{1},\eta') \in {\cal X}_{-}^{1},
 \\
F_{-}(y,\pm\sqrt{\eta_1^2-2q_{-}{\Phi}(y)}+\alpha,\eta')=F_{-}^{\infty}(\eta_{1}+\alpha,\eta'), &\quad (y,\eta_{1},\eta') \in {\cal X}_{-}^{2},
 \\
F_{-}(y,\sqrt{\eta_1^2-2q_{-}{\Phi}(y)}+\alpha,\eta')=F_{-}^{\infty}(\eta_{1}+\alpha,\eta'), & \quad  (y,\eta_{1},\eta') \in {\cal X}_{-}^{3},
\end{align*}
where 
\begin{align*}
{\cal X}_{-}^{1}&:={\R}\times(-\infty,-\sqrt{2q_{-}\Phi(0)})\times\R^{n-1},
\\
{\cal X}_{-}^{2}&:=\{(y,\eta_1,\eta')\in\R\times\R^n\;|\;2q_{-}{\Phi}(y)<\eta_1^2<2q_{-}\Phi(0)\}, \quad
\\
{\cal X}_{-}^{3}&:={\R}\times(\sqrt{2q_{-}\Phi(0)},+\infty)\times\R^{n-1}. 
\end{align*}
Due to the second equality, the following must hold:
\begin{align*}
F_{-}^\infty(\eta_1+\alpha,\xi')=F_{-}^\infty(-\eta_1+\alpha,\xi'), \quad 
(\eta_{1},\xi')\in(-\sqrt{2q_{-}\Phi(0)},\sqrt{2q_{-}\Phi(0)})\times\R^{n-1},
\end{align*}
which is the same equality as \eqref{G-beta1} with $\beta=\Phi(0)$.
Furthermore, we conclude from these three equalities that $f$ must be written as \eqref{fform-}, i.e.
\begin{align*}
F_{-}(X,\xi)
&=F_{-}^\infty(-\sqrt{(\xi_1-\alpha)^2+2q_{-}\Phi(X)}+\alpha,\xi')\chi(-(\xi_1-\alpha))
\\
&\quad +F_{-}^\infty(\sqrt{(\xi_1-\alpha)^2+2q_{-}\Phi(X)}+\alpha,\xi')\chi(\xi_1-\alpha). 
\end{align*}

Now we can reduce the problem \eqref{VP2} to a problem of an ordinary differential equation only for $\Phi$.
Integrating \eqref{fform+} over $\mathbb R^{n}$ and using the change of variables 
$\sqrt{(\xi_{1}-\alpha)^{2}-2q_{+}\Phi}=-(\zeta_{1}-\alpha)$,
$\sqrt{(\xi_{1}-\alpha)^{2}-2q_{+}\Phi+2q_{+}\Phi(0)}=(\zeta_{1}-\alpha)$,
and $\sqrt{(\xi_{1}-\alpha)^{2}-2q_{+}\Phi}=\zeta_{1}-\alpha$ 
for the first, second, and third terms on the right hand side, respectively, we see that
\begin{align}
&\int_{\mathbb R^{n}} F_{+}(X,\xi) \,d\xi  
\notag\\
&= \int_{{\mathbb R}^n}F_{+}^\infty(\xi)\frac{|\xi_1-\alpha|}{\sqrt{(\xi_1-\alpha)^2+2q_{+}\Phi(X)}}\,d\xi
\notag\\
&\quad +2\int_{\mathbb R^{n-1}}\int_{\sqrt{2q_{+}\Phi(0)-2q_{+}\Phi(X)}+\alpha}^{\sqrt{2q_{+}\Phi(0)}+\alpha}
F_{+}(0,\xi)\frac{\xi_1-\alpha}{\sqrt{(\xi_1-\alpha)^2+2q_{+}\Phi(X)-2q_{+}\Phi(0)}}\,d\xi_1\,d\xi'
\notag \\
&=  \rho_{+}(\Phi(X);\Phi(0),F_{+}(0,\cdot)),
\label{rho+'}
\end{align}
where $\rho_{+}$ is the same function defined in \eqref{rho++}.
On the other hand, integrating \eqref{fform-} over $\mathbb R^{n}$ and using the change of variables 
$\sqrt{(\xi_{1}-\alpha)^{2}+2q_{-}\Phi}=-(\zeta_{1}-\alpha)$ and 
$\sqrt{(\xi_{1}-\alpha)^{2}+2q_{-}\Phi}=\zeta_{1}-\alpha$ for the first and second terms on the right hand side, respectively,
we see that
\begin{align}
\int_{\mathbb R^{n}} F_{-}(X,\xi) \,d\xi  
&=\int_{{\mathbb R}^n}F_{-}^\infty(\xi)\frac{|\xi_1-\alpha|}{\sqrt{(\xi_1-\alpha)^2-2q_{-}\Phi(X)}} \chi((\xi_1-\alpha)^2-2q_{-}\Phi(X)) \,d\xi
\notag \\
&=  \rho_{-}(\Phi(X)),
\label{rho-'}
\end{align}
where $\rho_{-}$ is the same function defined in \eqref{rho--}.
Substituting \eqref{rho+'} and \eqref{rho-'} into \er{eq4}, we arrive at
\begin{equation}\lb{phieq1} 
\partial_{XX} \Phi 
= e_{+}\rho_{+}(\Phi;\Phi(0),F_{+}(0,\cdot)) - e_{-}\rho_{-}(\Phi), \quad X\in \mathbb R.
\end{equation}
Now we claim that $\rho_{\pm} \in C([0,\Phi(0)])$.
Indeed, owing to $ F_{\pm} \in C({\R};L^{1}(\R^{n}))$ in Definition \ref{DefS1} and \eqref{weak2}, there hold that
\begin{align*}
\rho_{\pm}(\Phi(X))= \| F_{\pm}(X) \|_{L^{1}} \in C({\R}),
\quad
\lim_{|X|\to+\infty}\| F_{\pm}(X) \|_{L^{1}} = \| F_{\pm}^{\infty} \|_{L^{1}}.
\end{align*}
These with $\rho_{\pm}(0)=\int_{\mathbb R^{n}} F_{\pm}^{\infty} d\xi$
imply $\rho_{\pm} \in C([0,\Phi(0)])$ with the aid of \eqref{bc4},
$ \Phi \in C(\R)$, and $\partial_{X}\Phi(X)\gtrless 0$ for $0\gtrless X$ in Definition \ref{DefS1}.
Thus the claim is valid. 
Multiply \er{phieq1} by $\D_{X}\Phi$, integrate it over $(X,+\infty)$, and use \eqref{bc4} and \eqref{lim1} to obtain the first equation in \eqref{phieq2} with $\beta=\Phi(0)$ and $G(\xi)=F_{+}(0,\xi)$, i.e.
\begin{gather*}
(\D_{X} \Phi)^2=2V(\Phi;\Phi(0),F_{+}(0,\cdot)), 
\\
V(\Phi;\Phi(0),F_{+}(0,\cdot))=\int_{0}^{\Phi} \left(e_{+}\rho_{+}(\varphi;\Phi(0),F_{+}(0,\cdot)) - e_{-}\rho_{-}(\varphi) \right)  d\vphi,
\notag
\end{gather*}
where $V$ is the same function defined in \eqref{V0}.
Thus $\Phi$ must solve the problem \eqref{phieq2}.

We prove that $V(\cdot;\Phi(0),F_{+}(0,\cdot))$ satisfies \eqref{G-beta2} and \eqref{G-beta3} with $\beta=\Phi(0)$ and  $G(\xi)=F_{+}(0,\xi)$.
From the first equation in \eqref{phieq2} and the condition (ii) in Definition \ref{DefS1}, the following holds:
\begin{gather*}
V(\Phi;\Phi(0),F_{+}(0,\cdot))>0 \ \ \text{for $\Phi \in (0,\Phi(0))$}, \quad 
V(\Phi(0);\Phi(0),F_{+}(0,\cdot))=0,
\end{gather*}
which are the same conditions as in \eqref{G-beta2} with $\beta=\Phi(0)$ and  $G(\xi)=F_{+}(0,\xi)$.
Thus \eqref{G-beta2} with $\beta=\Phi(0)$ and  $G(\xi)=F_{+}(0,\xi)$ must hold.
Let us also show that 
\begin{gather*}
\int_{0}^{\Phi(0)/2} \frac{d\Phi}{\sqrt{V(\Phi;\Phi(0),F_{+}(0,\cdot))}} =+ \infty, \quad 
\int_{\Phi(0)/2}^{\Phi(0)} \frac{d\Phi}{\sqrt{V(\Phi;\Phi(0),F_{+}(0,\cdot))}} <+ \infty,
\end{gather*}
which are the same conditions as in \eqref{G-beta3} with $\beta=\Phi(0)$ and  $G(\xi)=F_{+}(0,\xi)$.
We note that there exists a unique positive number $X^{*}$ so that $\Phi(X^{*})=\Phi(0)/2$,
and then see from the first equation in \eqref{phieq2} that the first integral is unbounded as follows:
\begin{align*}
\int_{0}^{\Phi(0)/2}\!\! \frac{d\Phi}{\sqrt{2V(\Phi;\Phi(0),F_{+}(0,\cdot))}}
=\int_{X_{*}}^{+\infty} \!\! \frac{-\partial_{X} \Phi (X)}{\sqrt{2V(\Phi(X);\Phi(0),F_{+}(0,\cdot))} } dX
=\int_{X_{*}}^{+\infty} 1dX
=+\infty.
\end{align*}
Similarly, it is seen that the second integral is bounded. 
Thus \eqref{G-beta3} with $\beta=\Phi(0)$ and  $G(\xi)=F_{+}(0,\xi)$ must hold.
The proof is complete.
\end{proof}

\subsection{Solvability}\label{S2.2}

In this subsection, we prove Theorem \ref{existence1}.

\begin{proof}[Proof of Theorem \ref{existence1}.]
Due to Lemma \ref{need1}, it suffices to show the solvability of the problem \eqref{VP2} provided that there exists a pair $(\beta,G)$ with \eqref{G-beta}. 
We first construct $\Phi$ by solving the problem \eqref{phieq2}.
To this end, we solve the following problem for $Y>0$:
\begin{gather}
\D_{Y}{\Psi}(Y)=-\sqrt{2V(\Psi(Y);\beta,G)}, \quad \Psi(0)=\frac{1}{2}\beta.
\label{psieq1}
\end{gather}
Lemma \ref{rhopm} ensures $V(\cdot;\beta,G) \in C^{1}([0,\beta])$. 
From this and \eqref{G-beta2}, we see that $\sqrt{V(\cdot;\beta,G)}$ is Lipschitz continuous on $[\ve,\beta-\ve]$ for any suitably small $\ve>0$.
Therefore, the problem \eqref{psieq1} has a solution $\Psi$ which is strictly decreasing unless $\Psi$ attains zero.
Suppose that $\Psi$ attains zero at a point $Y=Y_{*}<+\infty$. We observe that 
\begin{align*}
Y_{*}&=\int_{0}^{Y_{*}} 1dY
=\int_{0}^{Y_{*}} \frac{-\partial_{Y} \Psi (Y)}{\sqrt{2V(\Psi(Y);\beta,G)} } dY 
= \int_{0}^{\beta/2} \frac{d\Phi}{\sqrt{2V(\Phi;\beta,G)}}=+\infty,
\end{align*}
where we have used \eqref{G-beta3} in deriving the last equality. 
This is a contradiction.
Thus $\Psi$ cannot attain zero at some finite point.
Now it is clear that the problem \eqref{psieq1} is solvable for any $Y>0$ and $\lim_{Y \to +\infty} \Psi(Y)=0$.

Next we solve the problem \eqref{psieq1} for $Y<0$. Similarly as above, 
the problem \eqref{psieq1} has a solution $\Psi$ which is strictly decreasing.
For this case, it is seen by a similar method as above with the second condition in \eqref{G-beta3} 
that the solution $\Psi$ attains $\beta$ at some point $Y=Y_{\#}<0$.
Therefore, setting $X=Y-Y_{\#}$, we conclude that $\Psi \in C^{1}([0,+\infty))$ solves
\begin{gather}
\D_{X}{\Psi}(X)=-\sqrt{2V(\Psi(X);\beta,G)}, \quad \Psi(0)=\beta, \quad \lim_{X \to +\infty} \Psi(X)=0.
\label{psieq2}
\end{gather}

We now set $\Phi$ as
\begin{gather}
\Phi(X)=\left\{
\begin{array}{ll}
\Psi(X) & \text{for $X\geq 0$},
\\
\Psi(-X) & \text{for $X<0$}.
\end{array}
\right.
\end{gather}
It is straightforward to see that $\Phi \in C^{2}(\R)$ satisfies \eqref{phieq1}, \eqref{bc4}, 
and the condition (ii) in Definition \ref{DefS1}.

Now by using $\Phi$ and $\beta=\Phi(0)$, we define $F_{+}$  and  $F_{-}$ as \eqref{fform+} and \eqref{fform-}, respectively, i.e. 
\begin{align*}
{F}_{+}(X,\xi)
& =F_{+}^\infty(-\sqrt{(\xi_1-\alpha)^2-2q_{+}{\Phi}(X)}+\alpha,\xi')
\chi((\xi_{1}-\alpha)^{2}-2q_{+}\Phi(X))\chi(-(\xi_{1}-\alpha)) 
\\
& \quad
+G(\sqrt{(\xi_1-\alpha)^2-2q_{+}\Phi(X)+2q_{+}\Phi(0)}+\alpha,\xi')
\chi(-(\xi_1-\alpha)^2+2q_{+}\Phi(X)) 
\\
& \quad
+F_{+}^\infty((\sqrt{(\xi_1-\alpha)^2-2q_{+}\Phi(X)}+\alpha,\xi')
\chi((\xi_{1}-\alpha)^{2}-2q_{+}\Phi(X))\chi(\xi_{1}-\alpha)
\\
&=:F_{+}^{1}+F_{+}^{2}+F_{+}^{3},
\\
F_{-}(X,\xi)
&=F_{-}^\infty(-\sqrt{(\xi_1-\alpha)^2+2q_{-}\Phi(X)}+\alpha,\xi')\chi(-(\xi_1-\alpha))
\\
&\quad +F_{-}^\infty(\sqrt{(\xi_1-\alpha)^2+2q_{-}\Phi(X)}+\alpha,\xi')\chi(\xi_1-\alpha)
\\
& =:F_{-}^{1}+F_{-}^{2}, 
\end{align*}
and prove that $F_{\pm}$ satisfy the conditions (i)--(iii) in Definition \ref{DefS1}.
Owing to $F_{+}^{\infty} \in L^1(\R^n)$, 
$F_{-}^{\infty} \in L^1(\R^n) \cap L_{loc}^{p}({\mathbb R};L^{1}(\mathbb R^{n-1}))$,
and $G \in L_{loc}^{p}([\alpha,+\infty);L^{1}(\mathbb R^{n-1}))$,
there exist sequences $\{F_{\pm}^{\infty k}\}$, $\{G^{k}\}$ $\subset C^{\infty}_{0}(\mathbb R^{n})$ 
such that $F_{+}^{\infty k}(\xi)=0$ holds if $|\xi_{1}-\alpha |<1/k$;
$F_{-}^{\infty k}(\xi)=0$ holds if $|\xi_{1}-\alpha-\sqrt{2q_{-}\Phi(0)}|<1/k$ or $|\xi_{1}-\alpha+\sqrt{2q_{-}\Phi(0)}|<1/k$;
$F_{-}^{\infty k}$ satisfies \eqref{G-beta1} with $\beta=\Phi(0)$;
$G^{k}(\xi)=0$ holds if $|\xi_{1}-\alpha-\sqrt{2q_{+}\Phi(0)}|<1/k$;
$F_{+}^{\infty k} \to F_{+}^{\infty}$  in $L^{1}(\mathbb R^{n})$ as $k \to +\infty$;
$F_{-}^{\infty k} \to F_{-}^{\infty}$  in $L^{1}(\mathbb R^{n}) \cap L^{p}(-2\sqrt{q_{-}\Phi(0)}+\alpha,2\sqrt{q_{-}\Phi(0)}+\alpha;L^{1}(\mathbb R^{n-1}))$ as $k \to +\infty$;
$G^{k} \to G$  in $L^{p}(\alpha,2\sqrt{q_{+}\Phi(0)}+\alpha;L^{1}(\mathbb R^{n-1}))$ as $k \to +\infty$.

Let us show $F_{+} \in C({\mathbb R};L^{1}(\mathbb R^{n}))\cap L^{1}_{loc}(\mathbb R \times\mathbb R^{n})$ in the condition (i). 
First it is seen from \eqref{rho+}, \eqref{rho+1}, and \eqref{rho+'} that $F_{+}(X,\cdot), F_{+}^{i}(X,\cdot) \in L^{1}(\mathbb R^{n})$.
To investigate the continuity of $F_{+}^{1}$,
we set $J(X,\xi)=\chi((\xi_{1}-\alpha)^{2}-2q_{+}\Phi(X))\chi(-(\xi_{1}-\alpha)) $ and observe that for $X,X_{0} \in \mathbb R$,
\begin{align*}
&\|F_{+}^{1}(X)-F_{+}^{1}(X_{0})\|_{L^{1}(\mathbb R^{n})}
\\
&=\int_{\mathbb R^{n}}\left| F_{+}^\infty(-\sqrt{(\xi_1-\alpha)^2-2q_{+}{\Phi}(X)}+\alpha,\xi')J(X,\xi) \right.
\\
& \mspace{70mu} \left. - F_{+}^\infty(-\sqrt{(\xi_1-\alpha)^2-2q_{+}{\Phi}(X_{0})}+\alpha,\xi')J(X_{0},\xi) \right| d\xi
\\
&\leq \int_{\mathbb R^{n}}\left| F_{+}^\infty(-\sqrt{(\xi_1-\alpha)^2-2q_{+}{\Phi}(X)}+\alpha,\xi')J(X,\xi) \right.
\\
& \mspace{70mu} \left. - F_{+}^{\infty k}(-\sqrt{(\xi_1-\alpha)^2-2q_{+}{\Phi}(X)}+\alpha,\xi')J(X,\xi) \right| d\xi
\\
&\quad + \int_{\mathbb R^{n}}\left| F_{+}^{\infty k}(-\sqrt{(\xi_1-\alpha)^2-2q_{+}{\Phi}(X)}+\alpha,\xi')J(X,\xi) \right.
\\
& \mspace{70mu} \left. - F_{+}^{\infty k}(-\sqrt{(\xi_1-\alpha)^2-2q_{+}{\Phi}(X_{0})}+\alpha,\xi')J(X_{0},\xi) \right| d\xi
\\
&\quad + \int_{\mathbb R^{n}}\left| F_{+}^{\infty k}(-\sqrt{(\xi_1-\alpha)^2-2q_{+}{\Phi}(X_{0})}+\alpha,\xi')J(X_{0},\xi) \right.
\\
& \mspace{70mu} \left. - F_{+}^{\infty}(-\sqrt{(\xi_1-\alpha)^2-2q_{+}{\Phi}(X_{0})}+\alpha,\xi')J(X_{0},\xi) \right| d\xi
\\
&=:K_{1}+K_{2}+K_{3}.
\end{align*}
Using the change of variable $\sqrt{(\xi_{1}-\alpha)^{2}-2q_{+}\Phi}=-(\zeta_{1}-\alpha)$ and the fact $\Phi(X)>0$, we can estimate $K_{1}$ as
\begin{align*}
K_{1}
\leq  \int_{\mathbb R^{n}} | F_{+}^{\infty}(\xi) - F_{+}^{\infty k}(\xi) |  \frac{|\xi_1-\alpha|}{\sqrt{(\xi_1-\alpha)^2+2q_{+}\Phi(X)}} \,d\xi
\leq \|F_{+}^{\infty} - F_{+}^{\infty k} \|_{L^{1}(\mathbb R^{n})}   \to 0 \quad \text{as $k\to +\infty$}.
\end{align*}
Similarly, $K_{3} \to 0$ as $k \to +\infty$. 
Thus $K_{1}$ and $K_{3}$ can be arbitrarily small for suitably large $k$. 
For the fixed $k$, the dominated convergence theorem ensures that $K_{2}$ converges to zero as $X \to X_{0}$.
Hence, we deduce $F_{+}^{1} \in C(\mathbb R;L^{1}(\mathbb R^{n}))$.
Similarly,  $F_{+}^{3} \in C(\mathbb R;L^{1}(\mathbb R^{n}))$ holds.
To investigate the continuity of $F_{+}^{2}$, 
we set $\tilde{J}(X,\xi)=\chi(-(\xi_1-\alpha)^2+2q_{+}\Phi(X)) $ and observe that for $X,X_{0} \in \mathbb R$,
\begin{align*}
&\|F_{+}^{2}(X)-F_{+}^{2}(X_{0})\|_{L^{1}(\mathbb R^{n})}
\\
&=\int_{\mathbb R^{n}}\left| G(\sqrt{(\xi_1-\alpha)^2-2q_{+}\Phi(X)+2q_{+}\Phi(0)}+\alpha,\xi')\tilde{J}(X,\xi) \right.
\\
& \mspace{70mu} \left.  -G(\sqrt{(\xi_1-\alpha)^2-2q_{+}\Phi(X_{0})+2q_{+}\Phi(0)}+\alpha,\xi')\tilde{J}(X_{0},\xi) \right| d\xi
\\
&\leq \int_{\mathbb R^{n}}\left| G(\sqrt{(\xi_1-\alpha)^2-2q_{+}\Phi(X)+2q_{+}\Phi(0)}+\alpha,\xi')\tilde{J}(X,\xi) \right.
\\
& \mspace{70mu} \left.  -G^{k}(\sqrt{(\xi_1-\alpha)^2-2q_{+}\Phi(X)+2q_{+}\Phi(0)}+\alpha,\xi')\tilde{J}(X,\xi) \right| d\xi
\\
&\quad + \int_{\mathbb R^{n}}\left| G^{k}(\sqrt{(\xi_1-\alpha)^2-2q_{+}\Phi(X)+2q_{+}\Phi(0)}+\alpha,\xi')\tilde{J}(X,\xi) \right.
\\
& \mspace{70mu} \left.  -G^{k}(\sqrt{(\xi_1-\alpha)^2-2q_{+}\Phi(X_{0})+2q_{+}\Phi(0)}+\alpha,\xi')\tilde{J}(X_{0},\xi) \right| d\xi
\\
&\quad + \int_{\mathbb R^{n}}\left| G^{k}(\sqrt{(\xi_1-\alpha)^2-2q_{+}\Phi(X_{0})+2q_{+}\Phi(0)}+\alpha,\xi')\tilde{J}(X_{0},\xi) \right.
\\
& \mspace{70mu} \left.  -G(\sqrt{(\xi_1-\alpha)^2-2q_{+}\Phi(X_{0})+2q_{+}\Phi(0)}+\alpha,\xi')\tilde{J}(X_{0},\xi) \right| d\xi
\\
&=:\tilde{K}_{1}+\tilde{K}_{2}+\tilde{K}_{3}.
\end{align*}
Using the change of variable $\sqrt{(\xi_{1}-\alpha)^{2}-2q_{+}\Phi+2q_{+}\Phi(0)}=(\zeta_{1}-\alpha)$, we can estimate $\tilde{K}_{1}$ for the case $n\geq2$ as
\begin{align*}
\tilde{K}_{1}&\!=\! 2\!\int_{\sqrt{2q_{+}\Phi(0)-2q_{+}\Phi(X)}+\alpha}^{\sqrt{2q_{+}\Phi(0)}+\alpha}
\frac{\xi_1-\alpha}{\sqrt{(\xi_1-\alpha)^2\!+\!2q_{+}\Phi(X)\!-\!2q_{+}\Phi(0)}}\!\left(\! \int_{\mathbb R^{n-1}}\!\! |G(\xi)\!-\!G^{k}(\xi)| d\xi' \!\!\right) \! d\xi_1
\\
& \leq 2 \sup_{X \in \mathbb R }\left(\int_{\sqrt{2q_{+}\Phi(0)-2q_{+}\Phi(X)}+\alpha}^{\sqrt{2q_{+}\Phi(0)}+\alpha} \frac{|\xi_1-\alpha|^{p'}}{((\xi_1-\alpha)^2+2q_{+}\Phi(X)-2q_{+}\Phi(0))^{\frac{p'}{2}}} \,d\xi_1\right)^{\frac{1}{p'}}
\\
& \quad \times \| G-G^{k} \|_{L^{p}(\alpha,\sqrt{2q_{+}\Phi(0)}+\alpha;L^{1}(\mathbb R^{n-1}))} 
\\
&\leq  C \| G-G^{k} \|_{L^{p}(\alpha,\sqrt{2q_{+}\Phi(0)}+\alpha;L^{1}(\mathbb R^{n-1}))}    \to 0 \quad \text{as $k\to +\infty$},
\end{align*}
where $p'<2$ is the H\"older conjugate of $p>2$, and $C$ is some positive constant. 
The convergence holds for the case $n=1$ as well.
Similarly, $\tilde{K}_{3} \to 0$ as $k \to +\infty$. 
Thus $\tilde{K}_{1}$ and $\tilde{K}_{3}$ can be arbitrarily small for suitably large $k$. 
For the fixed $k$, the dominated convergence theorem ensures that $\tilde{K}_{2}$ converges to zero as $X \to X_{0}$.
Hence, we deduce that $F_{+}^{2} \in C(\mathbb R;L^{1}(\mathbb R^{n}))$ holds and so does 
$F_{+} \in C(\mathbb R;L^{1}(\mathbb R^{n}))$.
Now it is straightforward to show $F_{+} \in L^{1}_{loc}(\mathbb R\times\mathbb R^{n})$.

Let us show $F_{-} \in C(\mathbb R;L^{1}(\mathbb R^{n}))\cap L^{1}_{loc}(\mathbb R\times\mathbb R^{n})$ in the condition (i).
First $F_{-}(X,\cdot),F_{-}^{i}(X,\cdot) \in L^{1}(\mathbb R^{n})$ follows from \eqref{rho-} and \eqref{rho-'}.
To investigate the continuity of $F_{-}$, we observe that for $X,X_{0} \in \mathbb R$,
\begin{align*}
&\|F_{-}^{1}(X)-F_{-}^{1}(X_{0})\|_{L^{1}(\mathbb R^{n})}
\\
&=\int_{\mathbb R^{n}}\left| F_{-}^{\infty}(-\sqrt{(\xi_1-\alpha)^2+2q_{-}\Phi(X)}+\alpha,\xi')\chi(-(\xi_1-\alpha)) \right.
\\
& \mspace{70mu} \left. - F_{-}^{\infty}(-\sqrt{(\xi_1-\alpha)^2+2q_{-}\Phi(X_{0}))}+\alpha,\xi')\chi(-(\xi_1-\alpha)) \right| d\xi
\\
&\leq \int_{\mathbb R^{n}}\left| F_{-}^{\infty}(-\sqrt{(\xi_1-\alpha)^2+2q_{-}\Phi(X)}+\alpha,\xi')\chi(-(\xi_1-\alpha)) \right.
\\
& \mspace{70mu} \left. - F_{-}^{\infty k}(-\sqrt{(\xi_1-\alpha)^2+2q_{-}\Phi(X)}+\alpha,\xi')\chi(-(\xi_1-\alpha)) \right| d\xi
\\
&\quad + \int_{\mathbb R^{n}}\left| F_{-}^{\infty k}(-\sqrt{(\xi_1-\alpha)^2+2q_{-}\Phi(X)}+\alpha,\xi')\chi(-(\xi_1-\alpha)) \right.
\\
& \mspace{70mu}  \left. - F_{-}^{\infty k}(-\sqrt{(\xi_1-\alpha)^2+2q_{-}\Phi(X_{0}))}+\alpha,\xi')\chi(-(\xi_1-\alpha)) \right| d\xi
\\
&\quad + \int_{\mathbb R^{n}}\left| F_{-}^{\infty k}(-\sqrt{(\xi_1-\alpha)^2+2q_{-}\Phi(X_{0}))}+\alpha,\xi')\chi(-(\xi_1-\alpha)) \right.
\\
& \mspace{70mu} \left. - F_{-}^{\infty}(-\sqrt{(\xi_1-\alpha)^2+2q_{-}\Phi(X_{0}))}+\alpha,\xi')\chi(-(\xi_1-\alpha)) \right| d\xi
\\
&=:\hat{K}_{1}+\hat{K}_{2}+\hat{K}_{3}.
\end{align*}
Using the change of variables $\sqrt{(\xi_1-\alpha)^2+2q_{-}\Phi(X)}=- (\zeta_{1}-\alpha)$, we can estimate $\hat{K}_{1}$ for the case $n\geq 2$ as
\begin{align*}
\hat{K}_{1}&=\int_{{\mathbb R}^n}|F_{-}^{\infty}(\xi)-F_{-}^{\infty k}(\xi)| \frac{|\xi_1\!-\!\alpha|}{\sqrt{(\xi_1\!-\!\alpha)^2-2q_{-}\Phi(X)}} \chi(-(\xi_1\!-\!\alpha)) \chi((\xi_1\!-\!\alpha)^2\!-\!2q_{-}\Phi(X)) \,d\xi
\\
& \leq \int_{{\mathbb R}^n}|F_{-}^{\infty}(\xi)-F_{-}^{\infty k}(\xi)| \frac{|\xi_1-\alpha|}{\sqrt{(\xi_1-\alpha)^2-2q_{-}\Phi(X)}} \chi((\xi_1-\alpha)^2-4q_{-}\Phi(X)) \,d\xi
\\
& \quad + \int_{{\mathbb R}^n} |F_{-}^{\infty}(\xi)-F_{-}^{\infty k}(\xi)| \frac{|\xi_1-\alpha|}{\sqrt{(\xi_1-\alpha)^2-2q_{-}\Phi(X)}} 
\\
& \qquad \qquad \times \{\chi((\xi_1-\alpha)^2-2q_{-}\Phi(X)) - \chi((\xi_1-\alpha)^2-4q_{-}\Phi(X))  \} \,d\xi
\\
& \leq \sqrt{2}\|F_{-}^{\infty}-F_{-}^{\infty k}\|_{L^{1}(\mathbb R^{n})} 
\\
& \quad + \int_{-2\sqrt{q_{-}\Phi(0)}+\alpha}^{2\sqrt{q_{-}\Phi(0)}+\alpha}  \frac{|\xi_1-\alpha|\chi((\xi_1-\alpha)^2-2q_{-}\Phi(X)) }{\sqrt{(\xi_1-\alpha)^2-2q_{-}\Phi(X)}} \ \left( \int_{{\mathbb R}^{n-1}} |F_{-}^{\infty}(\xi)-F_{-}^{\infty k}(\xi)| d\xi' \right)d\xi_{1}
\\
&\leq  \sqrt{2}\|F_{-}^{\infty}-F_{-}^{\infty k}\|_{L^{1}(\mathbb R^{n})} 
\\
& \quad +\sup_{X \in \mathbb R }\left( \int_{-2\sqrt{q_{-}\Phi(0)}+\alpha}^{2\sqrt{q_{-}\Phi(0)}+\alpha}  \frac{|\xi_1-\alpha|^{p'}\chi((\xi_1-\alpha)^2-2q_{-}\Phi(X))}{((\xi_1-\alpha)^2-2q_{-}\Phi(X))^{\frac{p'}{2}}} \,d\xi_1\right)^{\frac{1}{p'}} 
\\
& \qquad \qquad \times \|F_{-}^{\infty}-F_{-}^{\infty k}\|_{L^{p}(-2\sqrt{q_{-}\Phi(0)}+\alpha,2\sqrt{q_{-}\Phi(0)}+\alpha;L^{1}(\mathbb R^{n-1}))}
\\
&\leq \sqrt{2}\|F_{-}^{\infty}-F_{-}^{\infty k}\|_{L^{1}(\mathbb R^{n})} 
+C \|F_{-}^{\infty}-F_{-}^{\infty k}\|_{L^{p}(-2\sqrt{q_{-}\Phi(0)}+\alpha,2\sqrt{q_{-}\Phi(0)}+\alpha;L^{1}(\mathbb R^{n-1}))}
\\
& \to 0 \ \ \text{as $k\to +\infty$},
\end{align*}
where $p'<2$ is the H\"older conjugate of $p>2$, and $C$ is some positive constant. 
The convergence holds for the case $n=1$ as well.
Similarly, $\hat{K}_{3} \to 0$ as $k \to +\infty$. 
Thus $\hat{K}_{1}$ and $\hat{K}_{3}$ can be arbitrarily small for suitably large $k$. 
For the fixed $k$, the dominated convergence theorem ensures that $\hat{K}_{2}$ converges to zero as $X \to X_{0}$.
Hence, we deduce $F_{-}^{1} \in C(\mathbb R;L^{1}(\mathbb R^{n}))$.
Similarly,  $F_{-}^{2} \in C(\mathbb R;L^{1}(\mathbb R^{n}))$ and hence 
$F_{-} \in C(\mathbb R;L^{1}(\mathbb R^{n}))$.
Now it is straightforward to show $F_{-} \in L^{1}_{loc}(\mathbb R \times\mathbb R^{n})$. 
Thus the condition (i) holds.

It is clear that  the condition (ii), i.e. $F_{\pm} \geq 0$, holds.
Let us prove \eqref{weak1} and \eqref{weak2} in the condition (iii). 
Obviously, \eqref{weak2} follows from the same manner as above.
It is also evident that the function $F_{\pm}^{k}$ defined by replacing $F_{\pm}^{\infty}$ and $G$ by  $F_{\pm}^{\infty k}$ and $G^{k}$ in \eqref{fform+} and \eqref{fform-} belongs to $C^{\infty}_{0}(\mathbb R^{n})$, and satisfies the weak form \eqref{weak1} for each $k$.
Using  the same change of variable as above and letting $k \to +\infty$, 
we see that $F_{\pm}$ also satisfies \eqref{weak1}.
Consequently, all the conditions (i)--(iii) hold.

The condition (iv) is validated by \eqref{fform+}, \eqref{fform-}, and \eqref{rho+'}--\eqref{phieq1}.
The proof is complete.
\end{proof}

\subsection{Properties of the functions $\rho_{\pm}(\Phi)$}\label{S2.3}

This subsection is devoted to the proof of  Lemma \ref{rhopm}.

\begin{proof}[Proof of Lemma \ref{rhopm}.]
It is obvious that \eqref{rho+} holds, since $\Phi$ is nonnegative.
We show \eqref{rho+1} and \eqref{rho-} only for the case $n \geq 2$, since the case $n=1$ can be shown in the same way.
Let us estimate $\rho_{+}^{0}(\Phi;\beta,G)$ by the H\"older inequality as follows:
\begin{align*}
&|\rho_{+}^{0}(\Phi;\beta,G)| 
\notag \\
& \leq 2  \int_{\sqrt{2q_{+}\beta-2q_{+}\Phi}+\alpha}^{\sqrt{2q_{+}\beta}+\alpha} \frac{\xi_1-\alpha}{\sqrt{(\xi_1-\alpha)^2+2q_{+}\Phi-2q_{+}\beta}}\left( \int_{\mathbb R^{n-1}} G(\xi)\, \,d\xi'  \right) d\xi_1
\\
&\leq 2 \sup_{\Phi \in [0,\beta]} \left( \int_{\sqrt{2q_{+}\beta-2q_{+}\Phi}+\alpha}^{\sqrt{2q_{+}\beta}+\alpha}\frac{|\xi_1-\alpha|^{p'}}{\{(\xi_1-\alpha)^2+2q_{+}\Phi-2q_{+}\beta\}^{\frac{p'}{2}}}d\xi_1 \right)^{\frac{1}{p'}} 
\\
&\mspace{150mu} \times
\|G\|_{L^{p}(\alpha,\sqrt{2q_{+}\beta}+\alpha;L^{1}(\mathbb R^{n-1}))}
\\
& \leq C\|G\|_{L^{p}(\alpha,\sqrt{2q_{+}\beta}+\alpha;L^{1}(\mathbb R^{n-1}))}
\end{align*}
for $\Phi \in [0,\beta]$, where $p'<2$ is the H\"older conjugate of $p>2$. 
Similarly,  we observe that 
\begin{align*}
&|\rho_{-}(\Phi)| 
\\
& = \int_{{\mathbb R}^n}F_{-}^\infty(\xi)\frac{|\xi_1-\alpha|}{\sqrt{(\xi_1-\alpha)^2-2q_{-}\Phi}} \chi((\xi_1-\alpha)^2-4q_{-}M) \,d\xi 
\\
&\quad +\int_{{\mathbb R}^n} F_{-}^\infty(\xi)\frac{|\xi_1-\alpha|}{\sqrt{(\xi_1-\alpha)^2-2q_{-}\Phi}}  \{\chi((\xi_1-\alpha)^2-2q_{-}\Phi)-\chi((\xi_1-\alpha)^2-4q_{-}M)\}\,d\xi 
\\
& \leq  \sqrt{2}\|F_{-}^\infty\|_{L^{1}(\mathbb R^{n})} 
\\
& \quad + \int_{-2\sqrt{q_{-}M}+\alpha}^{2\sqrt{q_{-}M}+\alpha} \frac{|\xi_1-\alpha|}{\sqrt{(\xi_1-\alpha)^2-2q_{-}\Phi}}  \chi((\xi_1-\alpha)^2-2q_{-}\Phi) \ \left( \int_{{\mathbb R}^{n-1}} F_{-}^\infty(\xi) d\xi' \right)d\xi_{1}
\\
& \leq \sqrt{2}\|F_{-}^\infty\|_{L^{1}(\mathbb R^{n})} +C(M)\|F_{-}^\infty\|_{L^{p}(-2\sqrt{q_{-}M}+\alpha,2\sqrt{q_{-}M}+\alpha;L^{1}(\mathbb R^{n-1}))}.
\end{align*}
Thus \eqref{rho+1} and \eqref{rho-} hold for the case $n \geq 2$.

It remains to show the regularity $\rho_{+}^{\infty}(\cdot) \in C([0,+\infty))$,
$\rho_{+}^{0}(\cdot;\beta,G) \in C([0,\beta])$, and $\rho_{-}(\cdot) \in C([0,+\infty))$.  
First we note that
\begin{align*}
\rho_{+}^{0}(\Phi)&= \int_{\mathbb R^{n}}  
F_{+}^\infty(-\sqrt{(\xi_1-\alpha)^2-2q_{+}{\Phi}}+\alpha,\xi')
\chi((\xi_{1}-\alpha)^{2}-2q_{+}\Phi)\chi(-(\xi_{1}-\alpha)) 
\\
& \qquad+F_{+}^\infty((\sqrt{(\xi_1-\alpha)^2-2q_{+}\Phi}+\alpha,\xi')
\chi((\xi_{1}-\alpha)^{2}-2q_{+}\Phi)\chi(\xi_{1}-\alpha) d\xi,
\\
\rho_{+}^{0}(\Phi;\beta,G)&= \int_{\mathbb R^{n}}  
G(\sqrt{(\xi_1-\alpha)^2-2q_{+}\Phi+2q_{+}\beta}+\alpha,\xi')
\chi(-(\xi_1-\alpha)^2+2q_{+}\Phi) d\xi,
\\
\rho_{-}(\Phi)&= \int_{\mathbb R^{n}} 
F_{-}^\infty(-\sqrt{(\xi_1-\alpha)^2+2q_{-}\Phi}+\alpha,\xi')\chi(-(\xi_1-\alpha))
\\
&\qquad +F_{-}^\infty(\sqrt{(\xi_1-\alpha)^2+2q_{-}\Phi}+\alpha,\xi')\chi(\xi_1-\alpha) d\xi.
\end{align*}
The regularity follows from just replacing $\Phi(X)$, $\Phi(X_{0})$, $\Phi(0)$, and $X \to X_{0}$ 
in the \footnote{It is written in the fourth and fifth paragraphs of the proof of Theorem \ref{existence1}.}proof of $F_{\pm} \in C({\mathbb R};L^{1}(\mathbb R^{n}))$ in Theorem \ref{existence1}
by $\Phi$, $\Phi_{0}$, $\beta$, and $\Phi \to \Phi_{0}$, respectively.
\end{proof}

\subsection{Uniqueness and nonuniqueness}\label{S2.4}

This section is devoted to the proof of Theorem \ref{uniqueness0}. 
For this purpose, we decompose $V$ into two parts as follows:
\begin{gather}\label{deco1}
V(\Phi;\beta,G)=V^{\infty}(\Phi)+V^{0}(\Phi;\beta,G), \quad 
V^{0}(\Phi;\beta,G) :=\int_{0}^{\Phi} e_{+}\rho_{+}^{0}(\varphi;\beta,G) d\vphi,
\end{gather}
where $\rho_{+}^{0}$ and $V^{\infty}$ are defined in \eqref{rho++} and \eqref{Vinfty}, respectively.

We first prove the uniqueness assuming either condition (i) or (ii) in Theorem \ref{uniqueness0}.
Namely, we show the following two lemmas, where $\beta_{*}$ is defined in \eqref{beta*}.

\begin{lem}\label{uniqueness1}
Suppose that the same assumptions in Theorem \ref{uniqueness0} hold.
Assume that $G\equiv 0$ holds on $(\alpha,\sqrt{2q_{+}\beta}+\alpha)\times \mathbb R^{n-1}$, and $\beta=\beta_{*}$ holds.
Then the solution of the problem \eqref{VP2} is unique.
\end{lem}
\begin{proof}
We claim that it is sufficient to show that the pair $(\beta,G)$ with \eqref{G-beta} is unique.
To prove this claim, let us first show the uniqueness of the solution of the following problem:
\begin{gather}\label{uniquePhi}
\D_{X}{\Phi}=\pm\sqrt{2V(\Phi;\beta,G)},  \quad \pm \,X<0, \quad 
\lim_{|X| \to +\infty} \Phi (X) =  0, \quad 
\Phi(0)=\beta.
\end{gather}
Suppose that  $\Phi^{1}$ and $\Phi^{2}$ are solutions of the problem, and $\Phi^{1} \not\equiv \Phi^{2}$ holds.
We may assume without loss of generality that
$0<\Phi^{1}(-x)<\Phi^{2}(-x)<\beta$ holds for some $-x \in(-\infty,0)$.
There exists a positive constant $\delta$
such that $\Phi^{1}(-x)=\Phi^{2}(-x-\delta)$.
Let 
\begin{gather*}
\Psi^{1}(X):=\Phi^{1}(-x+X), \qu \Psi^{2}(X):=\Phi^{2}(-x-\delta+X).
\end{gather*}
It is clear that $\Psi^{1}(X)$ and $\Psi^{2}(X)$ solve
\begin{gather*} 
\D_{X}{\Psi}=\sqrt{2V(\Psi;\beta,G)} \quad \text{for $X \in (0,x)$}, \quad \Psi(0)=\Phi^{1}(-x)=\Phi^{2}(-x-\delta).
\end{gather*}
From \eqref{G-beta2} and Lemma \ref{rhopm}, we see that 
$\sqrt{V(\cdot;\beta,G)}$ is Lipschitz continuous on $[\ve,\beta-\ve]$, and also $\Psi(0) \in [\ve,\beta-\ve]$ holds
for any suitably small $\ve>0$. 
Using these facts and the continuity of $\Psi$, we deduce that $\Psi^{1}=\Psi^{2}$ holds on $[0,x]$.
Then evaluating $\Phi^{1}$ and $\Phi^{2}$ at $X=x$  gives 
$$
\beta=\Phi^{1}(0)=\Psi^{1}(x)=\Psi^{2}(x)=\Phi^{2}(-\delta)<\beta,
$$
which is a contradiction. 
Therefore, $\Phi^{1}=\Phi^{2}$ holds and the solution of \eqref{uniquePhi} is unique.

Now we prove the claim above.
Suppose that the pair $(\beta,G)$ with \eqref{G-beta} is unique, 
and $(F_{\pm},\Phi)$ and  $(\tilde{F}_{\pm},\tilde{\Phi})$ are the solutions of the problem \eqref{VP2}.
Then $\beta=\Phi(0)=\tilde{\Phi}(0)$ and $G(\xi)=F_{+}(0,\xi)=\tilde{F}_{+}(0,\xi)$ hold due to Lemma \ref{need1}. 
Moreover, $\Phi$ and $\tilde{\Phi}$ solve the problem \eqref{phieq2} with $\Phi(0)=\beta$.
From this fact, \eqref{G-beta2}, and $\partial_{X}\Phi(X), \partial_{X}\tilde{\Phi}(X)\gtrless 0$ for $0\gtrless X$, 
it is seen that $\Phi$ and $\tilde{\Phi}$ solves \eqref{uniquePhi} and therefore $\Phi=\tilde{\Phi}$ holds.
Using this and the forms \eqref{fform+}--\eqref{fform-}, we also have $F_{\pm}=\tilde{F}_{\pm}$.
Thus the claim is valid.

Let us show the pair $(\beta,G)$ with \eqref{G-beta} is unique.
We suppose that another pair $(\tilde{\beta},\tilde{G})$ with \eqref{G-beta} exists, 
and show that $\tilde{\beta}=\beta$ and $\tilde{G}\equiv 0$ on $(\alpha,\sqrt{2q_{+}\beta}+\alpha)\times \mathbb R^{n-1}$.
First $\tilde{\beta} \leq \beta$ follows from \eqref{G-beta1} and the assumption $\beta=\beta_{*}$.
Then $V(\tilde{\beta};\beta,G)=V^{\infty}(\tilde{\beta}) \geq 0$ holds, since the pair $(\beta,G)$ satisfies \eqref{G-beta2} and $G\equiv 0$ holds on $(\alpha,\sqrt{2q_{+}\beta}+\alpha)\times \mathbb R^{n-1}$.
On the other hand, we see from \eqref{G-beta2} that 
\begin{gather*}
0=V(\tilde{\beta};\tilde{\beta},\tilde{G})=V^{\infty}(\tilde{\beta})+V^{0}(\tilde{\beta};\tilde{\beta},\tilde{G}) \geq V^{0}(\tilde{\beta};\tilde{\beta},\tilde{G}).
\end{gather*}
From this and the nonnegativity of $\tilde{G}$, we deduce $V^{0}(\tilde{\beta};\tilde{\beta},\tilde{G})=0$.
This implies that 
\begin{gather}\label{tG=0}
\tilde{G}\equiv0 \quad \text{on $(\alpha,\sqrt{2q_{+}\beta}+\alpha)\times \mathbb R^{n-1}$}.
\end{gather}

Let us complete the proof by showing $\tilde{\beta}=\beta$.
Owing to \eqref{tG=0}, it is clear that $V(\Phi;\tilde{\beta},\tilde{G})=V^{\infty}(\Phi)$ holds for $\Phi \in [0,\tilde{\beta}]$.
Furthermore, $V(\Phi;\beta,G)=V^{\infty}(\Phi)$ holds for $\Phi \in [0,\beta]$.
Using these and the fact that both the pairs $(\beta,G)$ and $(\tilde{\beta},\tilde{G})$ satisfy \eqref{G-beta2}, we can arrive at $\tilde{\beta}=\beta$.
\end{proof}

\begin{lem}\label{uniqueness2}
Suppose that the same assumptions in Theorem \ref{uniqueness0} hold.
Assume that  $G\equiv 0$ holds on $(\alpha,\sqrt{2q_{+}\beta}+\alpha)\times \mathbb R^{n-1}$, and ${V}^{\infty}(\Phi)\geq 0$ holds for $\Phi \in (\beta,\beta_{*})$.
Then the solution of the problem \eqref{VP2} is unique.
\end{lem}
\begin{proof}
We suppose that another pair $(\tilde{\beta},\tilde{G})$ with \eqref{G-beta} exists, 
and show that $\tilde{\beta}=\beta$ and $\tilde{G}\equiv 0$ on $(\alpha,\sqrt{2q_{+}\beta}+\alpha)\times \mathbb R^{n-1}$.
It is seen from \eqref{G-beta1} that $\tilde{\beta} \leq \beta_{*}$. 
It is also clear that  ${V}^{\infty}(\Phi)\geq 0$ for $\Phi \in [\beta,\beta_{*}]$ owing to the continuity of $V^{\infty}$ ensured by Lemma \ref{rhopm} and the assumption that ${V}^{\infty}(\Phi)\geq 0$ for $\Phi \in (\beta,\beta_{*})$.
Moreover, $V^{\infty}(\Phi)=V(\Phi;\beta,G)\geq0$ holds for $\Phi \in [0,\beta]$ thanks to \eqref{G-beta2} and the assumption that $G\equiv 0$ on $(\alpha,\sqrt{2q_{+}\beta}+\alpha)\times \mathbb R^{n-1}$.
Therefore, $V^{\infty}(\Phi)\geq0$ holds for $\Phi \in [0,\beta_{*}]$.
Then we see from \eqref{G-beta2} that $0=V(\tilde{\beta};\tilde{\beta},\tilde{G})=V^{\infty}(\tilde{\beta})+V^{0}(\tilde{\beta};\tilde{\beta},\tilde{G}) \geq V^{0}(\tilde{\beta};\tilde{\beta},\tilde{G})$. 
From this and the nonnegativity of $\tilde{G}$, we deduce $V^{0}(\tilde{\beta};\tilde{\beta},\tilde{G})=0$.
This implies that \eqref{tG=0} holds.
Now we can complete the proof in the same way as in the proof of Lemma \ref{uniqueness1}.
\end{proof}

Next we study the case that neither condition (i) nor (ii) holds.
This case can be divided into the following three subcases:
\begin{enumerate}[(a)]
\item $G\not\equiv 0$ on $(\alpha,\sqrt{2q_{+}\beta}+\alpha)\times \mathbb R^{n-1}$;
\item $G\equiv 0$ on $(\alpha,\sqrt{2q_{+}\beta}+\alpha)\times \mathbb R^{n-1}$, $\exists \Phi_0\in(\beta,\beta_{*})$ s.t. $V^{\infty}(\Phi_0)<0$, and $\frac{d}{d\Phi} V^{\infty}(\beta_{\sharp})<0$;
\item $G\equiv 0$ on $(\alpha,\sqrt{2q_{+}\beta}+\alpha)\times \mathbb R^{n-1}$, $\exists \Phi_0\in(\beta,\beta_{*})$ s.t. $V^{\infty}(\Phi_0)<0$, and $\frac{d}{d\Phi} V^{\infty}(\beta_{\sharp})=0$,
\end{enumerate}
where 
\begin{gather*}
\beta_{\sharp}:=\inf\{\Phi\in[0,\beta_{*})\;|\;V^{\infty}(\Phi)<0\}.
\end{gather*}
We prove Lemmas \ref{nonunique1}--\ref{nonunique3} which state the nonuniqueness in the subcases (a)--(c), respectively.

\begin{lem}\label{nonunique1}
Suppose that the same assumptions in Theorem \ref{uniqueness0} hold.
Assume that $G\not\equiv 0$ on $(\alpha,\sqrt{2q_{+}\beta}+\alpha)\times \mathbb R^{n-1}$.
Then the problem \eqref{VP2} has infinite many solutions.
\end{lem}

\begin{proof}
If pairs $(\beta,\tilde{G})$ with \eqref{G-beta} and $\tilde{G} \neq G$ exist,
we can have a solution $(\tilde{F}_{\pm},\tilde{\Phi})$ for each $(\beta,\tilde{G})$ by following the proof of Theorem \ref{existence1}.
It is also seen from the form \eqref{fform+} that those solutions are different from each other.
Therefore, it is sufficient to find infinite many $\tilde{G}$ so that $(\beta,\tilde{G})$ satisfies \eqref{G-beta}.

First let us find $\tilde{G}$ for the case $n\geq2$.
We take any changes of valuables $\zeta=g(\xi)$ so that $g$ is a smooth bijection map, and $\xi_{1}=\xi_{1}$ and $|\partial(\xi_{2},\ldots,\xi_{n})/\partial(\zeta_{2},\ldots,\zeta_{n})|=1$ hold. Then the pair $(\beta,\tilde{G}(\xi))=(\beta,G(g(\xi)))$ satisfies \eqref{G-beta}.
Indeed, $\rho_{+}(\Phi;\beta,G)=\rho_{+}(\Phi;\beta,\tilde{G})$ and therefore $V(\Phi;\beta,G)=V(\Phi;\beta,\tilde{G})$.
Thus we have infinite many solutions.

Next let us find $\tilde{G}$ for the case $n=1$. Set 
\begin{gather*}
A:=\frac{2e_+}{q_+}\int_\alpha^{\sqrt{2q_+\beta}+\alpha}G(\xi_{1})(\xi_1-\alpha)^2\,
d\xi_1.
\end{gather*}
It is clear that $A>0$, since $G\not\equiv 0$ holds on $(\alpha,\sqrt{2q_+\beta}+\alpha)$.
We also see that 
\begin{align*}
&A \leq C \int_\alpha^{\sqrt{2q_+\beta}+\alpha}
G(\xi_{1})\,d\xi_1
\leq  C\|G\|_{L^p((\alpha,\sqrt{2q_+\beta}+\alpha))}
<+\infty.
\end{align*}
We introduce a parameter $\tau\in(0,\frac{1}{2})$ to construct infinite many pairs $(\beta,\tilde{G}_{\tau})$ with  \eqref{G-beta}.
For any $\tau\in(0,\frac{1}{2})$, we can find $\alpha^*$ and $\alpha^0$ such that 
\begin{gather*}
\alpha<\alpha^*<\alpha^0<\sqrt{2q_+\beta}+\alpha,
\\
\frac{2e_+}{q_+}
\int_\alpha^{\alpha^*} G(\xi_{1})(\xi_1-\alpha)^2\,d\xi_1 =\tau A, \quad 
\frac{2e_+}{q_+}
\int_{\alpha^*}^{\alpha^0} G(\xi_{1})(\xi_1-\alpha)^2\,d\xi_1 =\tau A.
\end{gather*}
Now we define $H_\tau(\xi_{1})$ and $\tilde{G}_\tau(\xi_{1})$ by
\begin{align*}
H_\tau(\xi_{1})&:=\left\{
\begin{array}{ll}
-G(\xi_{1}) & \text{if $\xi_{1} \in (\alpha,\alpha^*)$},
\\
G(\xi_{1}) & \text{if $\xi_{1} \in (\alpha^*,\alpha^0)$},
\\
0 & \text{otherwise},
\end{array}
\right.
\\
\tilde{G}_\tau(\xi_{1})&:=G(\xi_{1})+H_\tau(\xi_{1}).
\end{align*}
It is clear that $\tilde{G}_\tau\in L^1((\alpha,+\infty)) \cap L^p_{loc}([\alpha,+\infty))$ and $\tilde{G}_\tau\geq 0$.
Therefore, it is sufficient to show that $(\beta,\tilde{G}_\tau)$ satisfies \eqref{G-beta} in order to complete the proof.

To this end, we compute $V^{0}(\Phi;\beta,G)$ defined in \eqref{deco1} as follows:
\begin{align*}
&V^{0}(\Phi;\beta,G)
\\
&=2e_{+} \int_{0}^{\Phi} \int_{\sqrt{2q_+\beta-2q_+\varphi}+\alpha}^{\sqrt{2q_+\beta}+\alpha}
G(\xi_{1})\frac{\xi_1-\alpha}{\sqrt{(\xi_1-\alpha)^2+2q_+\varphi-2q_+\beta}}\,
d\xi_1d\varphi
\\
&=2e_+\int_{\sqrt{2q_+\beta-2q_+\Phi}+\alpha}
^{\sqrt{2q_+\beta}+\alpha}
\Bigl(\int_{\frac{1}{2q_+}\{-(\xi_1-\alpha)^2+2q_+\beta\}}^\Phi
G(\xi_{1})\frac{\xi_1-\alpha}{\sqrt{(\xi_1-\alpha)^2+2q_+\varphi-2q_+\beta}}
\,d\varphi\Bigr)\,d\xi_1
\\
&=\frac{2e_+}{q_+}\int_{\sqrt{2q_+\beta-2q_+\Phi}+\alpha}
^{\sqrt{2q_+\beta}+\alpha}
G(\xi_{1})(\xi_1-\alpha)\sqrt{(\xi_1-\alpha)^2+2q_+\Phi-2q_+\beta}
\,d\xi_1
\quad \text{for $\Phi\in[0,\beta]$.}
\end{align*}
It is also seen from \eqref{V0} that for $\Phi\in[0,\beta]$,
\begin{gather}
V(\Phi;\beta,\tilde{G}_\tau)
=V(\Phi;\beta,G)
+V^{0}(\Phi;\beta,H_\tau), 
\label{VV0}\\
V^{0}(\Phi;\beta,H_\tau)
=\frac{2e_+}{q_+}\int_{\sqrt{2q_+\beta-2q_+\Phi}+\alpha}
^{\sqrt{2q_+\beta}+\alpha}
H_\tau(\xi_{1})(\xi_1-\alpha)\sqrt{(\xi_1-\alpha)^2+2q_+\Phi-2q_+\beta}
\,d\xi_1.
\notag
\end{gather}

Let us complete the the proof by showing that $(\beta,\tilde{G}_\tau)$ satisfies \eqref{G-beta}.
It is obvious that \eqref{G-beta1} holds. 
Furthermore, by the definitions of $H_{\tau}$, $\alpha^{*}$, and  $\alpha^{0}$, the following holds:
\begin{align*}
V^{0}(\beta;\beta,H_\tau)
&=\frac{2e_+}{q_+}\int_\alpha^{\sqrt{2q_+\beta}+\alpha}
H_\tau(\xi_{1})(\xi_1-\alpha)^2\,d\xi_{1}
\\
&=-\frac{2e_+}{q_+}
\int_\alpha^{\alpha^*}
G(\xi_{1})(\xi_1-\alpha)^2\,d\xi
+\frac{2e_+}{q_+}
\int_{\alpha^*}^{\alpha^0}
G(\xi_{1})(\xi_1-\alpha)^2\,d\xi_{1}
\\
&=-\tau A+\tau A
=0.
\end{align*}
From this, \eqref{VV0}, and the assumption that that $(\beta,G)$ satisfies \eqref{G-beta2}, we have the last condition in \eqref{G-beta2}, i.e $V(\beta;\beta,\tilde{G}_\tau)=0$.
To show another condition in \eqref{G-beta2}, it suffices to prove that $V^{0}(\Phi;\beta,H_\tau)\geq 0$ holds for $\Phi\in[0,\beta]$. 
We divide into the two cases $\Phi\in[0,\beta']$ and $\Phi\in(\beta',\beta]$, where $\beta':=\beta-\frac{1}{2q_+}(\alpha^*-\alpha)^2 \in (0,\beta)$. For the case $\Phi\in[0,\beta']$, it is seen that the integrant in $V^{0}(\Phi;\beta,H_\tau)$ is nonnegative and so is  $V^{0}(\Phi;\beta,H_\tau)$.
Let us consider the case $\Phi\in(\beta',\beta]$. We use the  function $f_\Phi(\xi_1)$ defined by
\begin{gather*}
f_\Phi(\xi_1)
:=\frac{\sqrt{(\xi_1-\alpha)^2+2q_+\Phi-2q_+\beta}}{\xi_1-\alpha}, \quad
\xi_1\in I:=[\sqrt{2q_+\beta-2q_+\Phi}+\alpha,\sqrt{2q_+\beta}+\alpha].
\end{gather*}
We note that $f_\Phi(\xi_1)$ is nonnegative and increasing on $I$, and $\alpha^{*} \in I$ holds. 
Then it follows from the fact $f_\Phi(\xi_{1}) \leq f_\Phi(\alpha^*)$ for $\xi_{1} \in I \cap (-\infty,\alpha^{*}]$ that
\begin{gather*}
(\xi_1-\alpha)\sqrt{(\xi_1-\alpha)^2+2q_+\Phi-2q_+\beta}
\leq f_\Phi(\alpha^*)(\xi_1-\alpha)^2.
\end{gather*}
From this, we observe that
\begin{align}
& \frac{2e_+}{q_+}\int_{\sqrt{2q_+\beta-2q_+\Phi}+\alpha}^{\alpha^*}
G(\xi_{1})(\xi_1-\alpha)\sqrt{(\xi_1-\alpha)^2+2q_+\Phi-2q_+\beta}
\,d\xi_{1}
\notag \\
& \leq\frac{2e_+}{q_+}f_\Phi(\alpha^*)
\int_{\alpha}^{\alpha^*}G(\xi_{1})(\xi_1-\alpha)^2\,d\xi_{1}
=f_\Phi(\alpha^*)\tau A.
\label{fPhi1}
\end{align}
Similarly, there holds that 
\begin{gather}\label{fPhi2}
\frac{2e_+}{q_+}\int_{\alpha^*}^{\alpha^0}
G(\xi_{1})(\xi_1-\alpha)\sqrt{(\xi_1-\alpha)^2+2q_+\Phi-2q_+\beta}\,d\xi_{1}
\geq f_\Phi(\alpha^*)\tau A.
\end{gather}
From \eqref{fPhi1} and \eqref{fPhi2}, we arrive at 
\begin{align*}
V^{0}(\Phi;\beta,H_\tau)
&=-\frac{2e_+}{q_+}\int_{\sqrt{2q_+\beta-2q_+\Phi}+\alpha}^{\alpha^*}
G(\xi_{1})(\xi_1-\alpha)\sqrt{(\xi_1-\alpha)^2+2q_+\Phi-2q_+\beta}
\,d\xi_{1}
\\
& \quad +\frac{2e_+}{q_+}\int_{\alpha^*}^{\alpha^0}
G(\xi_{1})(\xi_1-\alpha)\sqrt{(\xi_1-\alpha)^2+2q_+\Phi-2q_+\beta}\,d\xi_{1}
\\
& \geq -f_\Phi(\alpha^*)\tau A+f_\Phi(\alpha^*)\tau A =0.
\end{align*}
Hence, $V^{0}(\Phi;\beta,H_\tau)\geq 0$ holds for both the two cases $\Phi\in[0,\beta']$ and $\Phi\in(\beta',\beta]$.
Thus $(\beta,\tilde{G}_\tau)$ satisfies \eqref{G-beta2}.

Let us show that $(\beta,\tilde{G}_\tau)$ satisfies \eqref{G-beta3}. 
It is easy to see that $V^{0}(\Phi;\beta,H_\tau)=0$ if $\Phi \ll1$.
From \eqref{VV0} and the assumption that $(\beta,G)$ satisfies \eqref{G-beta3}, we have the first condition in \eqref{G-beta3}.
Furthermore, another condition  in \eqref{G-beta3} follows from 
\begin{gather*}
\int_{\beta/2}^\beta\frac{d\Phi}{\sqrt{V(\Phi;\beta,\tilde{G}_\tau)}}
\leq\int_{\beta/2}^\beta\frac{d\Phi}{\sqrt{V(\Phi;\beta,G)}}
<+\infty,
\end{gather*}
where we have used \eqref{VV0} and $V^{0}(\Phi;\beta,H_\tau)\geq 0$.
Thus  $(\beta,\tilde{G}_\tau)$ satisfies \eqref{G-beta3}.
The proof is complete.
\end{proof}

\begin{lem}\label{nonunique2}
Suppose that the same assumptions in Theorem \ref{uniqueness0} hold.
Assume that $G\equiv 0$ holds on $(\alpha,\sqrt{2q_{+}\beta}+\alpha)\times \mathbb R^{n-1}$, there exists $\Phi_0\in(\beta,\beta_{*})$ with $V^{\infty}(\Phi_0)<0$, and $\frac{d}{d\Phi} V^{\infty}(\beta_{\sharp})<0$ holds.
Then the problem \eqref{VP2} has infinite many solutions with trapped ions.
\end{lem}
\begin{proof}
It is sufficient to find a pair $(\tilde{\beta},\tilde{G})$ with \eqref{G-beta} and $\tilde{G}\not\equiv0$ on $(\alpha,\sqrt{2q_{+}\tilde{\beta}}+\alpha)\times \mathbb R^{n-1}$ thanks to Lemma \ref{nonunique1}.
We first recall that $\beta\leq \beta_{\sharp}$ and $V^{\infty}$ is defined in \eqref{Vinfty}.
Set $k:=-\frac{d}{d\Phi}V^{\infty}(\beta_{\sharp})>0$, where the positivity follows from the assumption $\frac{d}{d\Phi} V^{\infty}(\beta_{\sharp})<0$. 
We fix a constant $\tilde{\beta}\in(\beta_{\sharp},\min\{\beta_{\sharp}+\beta,\beta_{*}\})$ so that 
\begin{gather}\label{tbeta}
V^{\infty}(\Phi)\geq-\frac{\,1\,}{6}e_+\beta_{\sharp}k, \quad
\frac{d}{d\Phi}V^{\infty}(\Phi)\leq -\frac{\,1\,}{2}k, \quad 
\Phi\in(\beta_{\sharp},\tilde{\beta}],
\\
V^{\infty}(\tilde{\beta})<0,
\label{p0}
\end{gather}
where we have used $V^{\infty}(\beta_{\sharp})=0$ and $V^{\infty} \in C^{1}([0,+\infty))$ ensured in Lemma \ref{rhopm}  to find such constant $\tilde{\beta}$.
We also define the functions $H$ and $\tilde{G}$ by 
\begin{align*}
H=H(\xi)&:=\Biggl\{
\begin{array}{ll}
\frac{1}{2\sqrt{2q_+}\,e_+} & \hbox{if} \quad
(\xi_1,\xi')\in[\alpha,\sqrt{2q_+\beta_{\sharp}}+\alpha]\times[0,1]^{n-1}, \\
0 & \hbox{otherwise},
\end{array}
\\
\tilde{G}=\tilde{G}(\xi)&:=\lambda H(\xi), \quad
\lambda
:=-\frac{1}{V^{0}(\tilde{\beta};\tilde{\beta},H)}\,
V^{\infty}(\tilde{\beta}).
\end{align*}
It is clear that $\tilde{G} \in L^1_{loc}((\alpha,+\infty) \times \R^{n-1}) \cap L_{loc}^{p}([\alpha,+\infty);L^{1}(\mathbb R^{n-1}))$.
It is also seen by direct computation that
\begin{align}
\rho_{+}^{0}(\Phi;\tilde{\beta},H)=\Biggl\{
\begin{array}{ll}
0& \hbox{if} \quad
\Phi\in[0,\tilde{\beta}-\beta_{\sharp}], \\
\frac{1}{e_+}
\sqrt{\Phi-(\tilde{\beta}-\beta_{\sharp})} & \hbox{if} \quad
\Phi\in[\tilde{\beta}-\beta_{\sharp},\tilde{\beta}],
\end{array}
\label{rho+01} \\
V^{0}(\Phi;\tilde{\beta},H)=\Biggl\{
\begin{array}{ll}
0& \hbox{if} \quad
\Phi\in[0,\tilde{\beta}-\beta_{\sharp}], \\
\frac{\,2\,}{3}
\{\Phi-(\tilde{\beta}-\beta_{\sharp})\}^{\frac{3}{2}} & \hbox{if} \quad
\Phi\in[\tilde{\beta}-\beta_{\sharp},\tilde{\beta}],
\end{array}
\label{V01}
\end{align}
where $\rho_{+}^{0}(\Phi;\beta,G)$ and $V^{0}(\Phi;\beta,G)$ are defined in \eqref{rho++} and \eqref{deco1}, respectively.
Then we have $V^{0}(\tilde{\beta};\tilde{\beta},H)=\frac{2}{3}\beta_{\sharp}^{\frac{3}{2}}>0$.
This and \eqref{p0} mean that $\lambda>0$, $\tilde{G}\not\equiv 0$, and $\tilde{G} \geq 0$ on $(\alpha,\sqrt{2q_+\tilde{\beta}}+\alpha)\times\R^{n-1}$.
Therefore, it sufficient to show that $(\tilde{\beta},\tilde{G})$ satisfies \eqref{G-beta} in order to complete the proof.

First \eqref{G-beta1} follows from $\tilde{\beta}<\beta_{*}$  and the definition of $\beta_{*}$.
Let us prove \eqref{G-beta2}. Recalling \eqref{deco1}, we have the last condition in \eqref{G-beta2}, i.e.
\begin{gather}\label{Vbeta0}
V(\tilde{\beta};\tilde{\beta},\tilde{G})
=V^{\infty}(\tilde{\beta})+V^{0}(\tilde{\beta};\tilde{\beta},\tilde{G})
=V^{\infty}(\tilde{\beta})
+\lambda V^{0}(\tilde{\beta};\tilde{\beta},H)
=0.
\end{gather}
We also show another condition i.e. $V(\Phi;\tilde{\beta},\tilde{G})>0$ for $\Phi\in(0,\tilde{\beta})$.
To this end, we divide into the three cases $\Phi\in(0,\beta)$, $\Phi\in[\beta,\beta_{\sharp}]$, and $\Phi\in(\beta_{\sharp},\tilde{\beta})$.
For the case $\Phi\in(0,\beta)$, we first see $V^{\infty}(\Phi)=V(\Phi;\beta,G)>0$ from the assumptions that the pair $(\beta,G)$ satisfies \eqref{G-beta2} and $G\equiv 0$ on $(\alpha,\sqrt{2q_{+}\beta}+\alpha)\times \mathbb R^{n-1}$. 
This fact with \eqref{V01} yields
$V(\Phi;\tilde{\beta},\tilde{G})=V^{\infty}(\Phi)+V^{0}(\Phi;\tilde{\beta},\tilde{G})>0$. 
For the case $\Phi\in[\beta,\beta_{\sharp}]$, it is seen from the definition of $\beta_{\sharp}$ that $V^{\infty}(\Phi)\geq 0$.
From this, \eqref{V01}, and the fact $\tilde{\beta}-\beta_{\sharp}< \beta\leq\Phi$,
we deduce that $V(\Phi;\tilde{\beta},\tilde{G})=V^{\infty}(\Phi)+V^{0}(\Phi;\tilde{\beta},\tilde{G})>0$.
Let us consider the case $\Phi\in(\beta_{\sharp},\tilde{\beta})$.
It follows from \eqref{rho+01} that $\rho_{+}^{0}(\Phi;\tilde{\beta},H)$ is strictly increasing on $[\beta_{\sharp},\tilde{\beta}]$ and therefore
\begin{gather*}
\frac{d}{d\Phi}V^{0}(\Phi;\tilde{\beta},\tilde{G})
=\rho_{+}^{0}(\Phi;\tilde{\beta},\tilde{G})
=\lambda\rho_{+}^{0}(\Phi;\tilde{\beta},H)
\leq\lambda\rho_{+}^{0}(\tilde{\beta};\tilde{\beta},H). 
\end{gather*}
It is seen from \eqref{rho+01} and \eqref{V01} that $V^{0}(\tilde{\beta};\tilde{\beta},H)=\frac{2}{3}\beta_{\sharp}^{\frac{3}{2}}$ and $\rho_{+}^{0}(\tilde{\beta};\tilde{\beta},H)=\frac{3}{2e_+\beta_{\sharp}}V^{0}(\tilde{\beta};\tilde{\beta},H)$. 
Recalling $\lambda =-V^{\infty}(\tilde{\beta})/V^{0}(\tilde{\beta};\tilde{\beta},H)$, we know that
\begin{gather*}
\lambda\rho_{+}^{0}(\tilde{\beta};\tilde{\beta},H)
=\frac{3}{2e_+\beta_{\sharp}}\lambda V^{0}(\tilde{\beta};\tilde{\beta},H)
=-\frac{3}{2e_+\beta_{\sharp}}V^{\infty}(\tilde{\beta}).
\end{gather*}
This and \eqref{tbeta} imply that
\begin{gather*}
\frac{d}{d\Phi}V^{0}(\Phi;\tilde{\beta},\tilde{G})
\leq\lambda\rho_{+}^{0}(\tilde{\beta};\tilde{\beta},H)
=-\frac{3}{2e_+\beta_{\sharp}}V^{\infty}(\tilde{\beta})
\leq\frac{\,1\,}{4}k. 
\end{gather*}
From this and \eqref{tbeta}, we conclude that
\begin{gather*}
\frac{d}{d\Phi}V(\Phi;\tilde{\beta},\tilde{G})
=\frac{d}{d\Phi}V^{\infty}(\Phi)
+\frac{d}{d\Phi}V^{0}(\Phi,\tilde{\beta},\tilde{G})
\leq -\frac{\,1\,}{2}k+\frac{\,1\,}{4}k
=-\frac{\,1\,}{4}k
<0,
\end{gather*}
which together with \eqref{Vbeta0} gives $V(\Phi;\tilde{\beta},\tilde{G})>0$.
Hence, $V(\Phi;\tilde{\beta},\tilde{G})>0$ holds for all the three cases $\Phi\in(0,\beta)$, $\Phi\in[\beta,\beta_{\sharp}]$, and $\Phi\in(\beta_{\sharp},\tilde{\beta})$.
Thus $(\tilde{\beta},\tilde{G})$ satisfies \eqref{G-beta2}.

Finally, we show that $(\tilde{\beta},\tilde{G})$ satisfies \eqref{G-beta3}.
It is easy to see that $V(\Phi;\tilde{\beta},\tilde{G})=V^{\infty}(\Phi)=V(\Phi;\beta,G)$ holds on $[0,\tilde{\beta}-\beta_{\sharp}]$, since $V^{0}(\Phi;\tilde{\beta},\tilde{G})\equiv 0$ holds on $[0,\tilde{\beta}-\beta_{\sharp}]$.
From this and the assumption that $(\beta,G)$ satisfies \eqref{G-beta}, 
we have the first condition in \eqref{G-beta3}, i.e.
\begin{gather*}
\int_0^{\tilde{\beta}/2}\frac{d\Phi}{\sqrt{V(\Phi;\tilde{\beta},\tilde{G})}}
\geq \int_0^{\min\{\tilde{\beta}-\beta_{\sharp},\tilde{\beta}/2\}}\frac{d\Phi}{\sqrt{V(\Phi;\tilde{\beta},\tilde{G})}}
=\int_0^{\min\{\tilde{\beta}-\beta_{\sharp},\tilde{\beta}/2\}}\frac{d\Phi}{\sqrt{V(\Phi;\beta,G)}}
=+\infty.
\end{gather*}
Furthermore, applying the Taylor theorem with \eqref{Vbeta0} and
$\frac{d}{d\Phi}V(\tilde{\beta};\tilde{\beta},\tilde{G})<0$, we obtain another condition, i.e.
\begin{gather*}
\int_{\tilde{\beta}/2}^{\tilde{\beta}}
\frac{d\Phi}{\sqrt{V(\Phi;\tilde{\beta},\tilde{G})}}
<+\infty.
\end{gather*}
Thus $(\tilde{\beta},\tilde{G})$ satisfies \eqref{G-beta3}.
The proof is complete.
\end{proof}

\begin{lem}\label{nonunique3}
Suppose that the same assumptions in Theorem \ref{uniqueness0} hold.
Assume that $G\equiv 0$ holds on $(\alpha,\sqrt{2q_{+}\beta}+\alpha)\times \mathbb R^{n-1}$, there exists $\Phi_0\in(\beta,\beta_{*})$ with $V^{\infty}(\Phi_0)<0$, and $\frac{d}{d\Phi} V^{\infty}(\beta_{\sharp})=0$ holds.
Then the problem \eqref{VP2} has infinite many solutions with trapped ions.
\end{lem}
\begin{proof}
It is sufficient to find a pair $(\tilde{\beta},\tilde{G})$ with \eqref{G-beta} and $\tilde{G}\not\equiv0$ on $(\alpha,\sqrt{2q_{+}\tilde{\beta}}+\alpha)\times \mathbb R^{n-1}$ thanks to Lemma \ref{nonunique1}.
It is seen from the definition of $\beta_{\sharp}$ that $V^{\infty}(\beta_{\sharp})=0$.
Furthermore, we note that $\beta\leq \beta_{\sharp}$.
We fix a constant $\beta_1\in(\beta_{\sharp},\beta_*)$ so that $V^{\infty}(\beta_1)<0$ holds,
where we have used the assumption that there exists $\Phi_0\in(\beta,\beta_{*})$ with $V^{\infty}(\Phi_0)<0$ to find such $\beta_{1}$.
Furthermore, we define the function $W^{\infty}$ by
\begin{gather*}
W^{\infty}=W^{\infty}(\Phi):=\Biggl\{
\begin{array}{ll}
0 & \hbox{if} \quad \Phi=\beta_{\sharp}, \\
\frac{1}{\Phi-\beta_{\sharp}}V^{\infty}(\Phi) & \hbox{if} \quad
\Phi\in(\beta_{\sharp},\beta_1].
\end{array}
\end{gather*}
Using $V^{\infty}(\beta_{\sharp})=0$ and the assumption $\frac{d}{d\Phi}V^{\infty}(\beta_{\sharp})=0$, we observe that
\begin{gather*}
W^{\infty}(\Phi)
=\frac{V^{\infty}(\Phi)-V^{\infty}(\beta_{\sharp})}{\Phi-\beta_{\sharp}}
\to \frac{dV^{\infty}}{d\Phi}(\beta_{\sharp})=0 \quad
\text{as $\Phi\to\beta_{\sharp}+0.$}
\end{gather*}
Therefore, $W^{\infty}$ is continuous on $[\beta_{\sharp},\beta_1]$ and hence we can define $k:=-\min_{\Phi\in[\beta_{\sharp},\beta_1]}W^{\infty}(\Phi)$. 
We note that $-k \leq W^{\infty}(\beta_1) =\frac{1}{\beta_1-\beta_{\sharp}}V^{\infty}(\beta_1)<0$. 
Now we fix a constant $\tilde{\beta}$ so that 
\begin{gather}
\tilde{\beta}\in(\beta_{\sharp},\beta_1], \quad
W^{\infty}(\tilde{\beta})=-k<0.
\label{tbeta2}
\end{gather}
From this and the definition of $W^{\infty}$, we deduce that
\begin{gather}\label{p2}
V^{\infty}(\tilde{\beta})<0, \quad
V^{\infty}(\Phi)\geq -k(\Phi-\beta_{\sharp}), \quad \Phi\in(\beta_{\sharp},\tilde{\beta}).
\end{gather}

Now we define the functions $H$ and $G$ by
\begin{align*}
H=H(\xi)&:=\Biggl\{
\begin{array}{ll}
\frac{1}{2e_+\sqrt{2q_+}} & \hbox{if} \ (\xi_1,\xi')\!\in\!\Bigl[\sqrt{2q_+(\tilde{\beta}-\beta)}+\alpha,
\sqrt{2q_+(\tilde{\beta}-\frac{1}{2}\beta)}+\alpha\Bigr]\times[0,1]^{n-1}, \\
0 & \hbox{otherwise},
\end{array}
\\
G=\tilde{G}(\xi)&:=\lambda H(\xi), \quad 
\lambda :=-\frac{1}{V^{0}(\tilde{\beta};\tilde{\beta},H)}\,V^{\infty}(\tilde{\beta}).
\end{align*}
It is clear that $\tilde{G} \in L^1_{loc}((\alpha,+\infty) \times \R^{n-1}) \cap L_{loc}^{p}([\alpha,+\infty);L^{1}(\mathbb R^{n-1}))$.
It is also seen by direct computation that
\begin{align}
\rho_{+}^{0}(\Phi;\tilde{\beta},H)
&=\left\{
\begin{array}{ll}
0& \hbox{if} \quad
\Phi\in[0,\frac{1}{2}\beta], \\
\frac{1}{e_+}
\sqrt{\Phi-\frac{1}{2}\beta} & \hbox{if} \quad
\Phi\in[\frac{1}{2}\beta,\beta], \\
\frac{1}{e_+}
(\sqrt{\Phi-\frac{1}{2}\beta}
-\sqrt{\Phi-\beta}) & \hbox{if} \quad
\Phi\in[\beta,\tilde{\beta}],
\end{array}\right.
\label{rho+02}\\
V^{0}(\Phi;\tilde{\beta},H)
&=\left\{
\begin{array}{ll}
0& \hbox{if} \quad
\Phi\in[0,\frac{1}{2}\beta], \\
\frac{\,2\,}{3}
(\Phi-\frac{1}{2}\beta)^{\frac{3}{2}} & \hbox{if} \quad
\Phi\in[\frac{1}{2}\beta,\beta], \\
\frac{\,2\,}{3}
\{(\Phi-\frac{1}{2}\beta)^{\frac{3}{2}}
-(\Phi-\beta)^{\frac{3}{2}}\} & \hbox{if} \quad
\Phi\in[\beta,\tilde{\beta}],
\end{array}\right.
\label{V02}
\end{align}
where $\rho_{+}^{0}(\Phi;\beta,G)$ and $V^{0}(\Phi;\beta,G)$ are defined in \eqref{rho++} and \eqref{deco1}, respectively.
Then we have $V^{0}(\tilde{\beta};\tilde{\beta},H)>0$.
This and the first inequality in \eqref{p2} mean that $\lambda>0$, $\tilde{G}\not\equiv 0$, and $G \geq 0$ on $(\alpha,\sqrt{2q_+\tilde{\beta}}+\alpha)\times\R^{n-1}$.
Therefore, to complete the proof, it sufficient to show that $(\tilde{\beta},\tilde{G})$ satisfies \eqref{G-beta}.

First \eqref{G-beta1} follows from $\tilde{\beta}\leq\beta_1<\beta_{*}$  and the definition of $\beta_{*}$.
Let us prove \eqref{G-beta2}. 
Recalling \eqref{deco1}, we have the last condition in \eqref{G-beta2}, i.e.
\begin{gather}\label{Vbeta2}
V(\tilde{\beta};\tilde{\beta},\tilde{G})
=V^{\infty}(\tilde{\beta})+V^{0}(\tilde{\beta};\tilde{\beta},\tilde{G})
=V^{\infty}(\tilde{\beta})
+\lambda V^{0}(\tilde{\beta};\tilde{\beta},H)
=0.
\end{gather}
We also show another condition, i.e. $V(\Phi;\tilde{\beta},\tilde{G})>0$ for $\Phi\in(0,\tilde{\beta})$.
To this end, we divide into three cases: $\Phi\in(0,\beta)$, $\Phi\in[\beta,\beta_{\sharp}]$, and $\Phi\in(\beta_{\sharp},\tilde{\beta})$.
For the first case $\Phi\in(0,\beta)$, we observe that $V^{\infty}(\Phi)=V(\Phi;\beta,G)>0$ from the assumptions that the pair $(\beta,G)$ satisfies \eqref{G-beta2} and $G\equiv 0$ on $(\alpha,\sqrt{2q_{+}\beta}+\alpha)\times \mathbb R^{n-1}$. 
This fact with \eqref{V02} yields
$V(\Phi;\tilde{\beta},\tilde{G})=V^{\infty}(\Phi)+V^{0}(\Phi;\tilde{\beta},\tilde{G})>0$. 
For the second case $\Phi\in[\beta,\beta_{\sharp}]$, it is seen from the definition of $\beta_{\sharp}$ that $V^{\infty}(\Phi)\geq 0$.
From this, \eqref{V02}, and $\Phi \geq \beta$,
we deduce that $V(\Phi;\tilde{\beta},\tilde{G})=V^{\infty}(\Phi)+V^{0}(\Phi;\tilde{\beta},\tilde{G})>0$.
Let us consider the last case $\Phi\in(\beta_{\sharp},\tilde{\beta})$.
It is seen by direct computation that
\begin{gather*}
\frac{d^2}{d\Phi^{2}}V^{0}(\Phi;\tilde{\beta},H)
=\frac{1}{2}\left(\frac{1}{\sqrt{\Phi-\frac{1}{2}\beta}}
-\frac{1}{\sqrt{\Phi-\beta}}\right)<0, \quad
\Phi\in(\beta_{\sharp},\tilde{\beta}).
\end{gather*}
This means that $V^{0}(\Phi;\tilde{\beta},\tilde{G})$ is convex upward on $[\beta_{\sharp},\tilde{\beta}]$, and hence the following holds:
\begin{gather*}
V^{0}((1-\theta)\beta_{\sharp}+\theta\tilde{\beta};\tilde{\beta},\tilde{G}) \geq (1-\theta) V^{0}(\beta_{\sharp};\tilde{\beta},\tilde{G})
+ \theta V^{0}(\tilde{\beta};\tilde{\beta},\tilde{G}), \quad \theta\in[0,1].
\end{gather*}
Taking $\theta=(\Phi-\beta_{\sharp})(\tilde{\beta}-\beta_{\sharp})^{-1}$ for $\Phi \in (\beta_{\sharp},\tilde{\beta})$, we see that
\begin{align}
V^{0}(\Phi;\tilde{\beta},\tilde{G}) 
&\geq \frac{V^{0}(\tilde{\beta};\tilde{\beta},\tilde{G})}{\tilde{\beta}-\beta_{\sharp}} (\Phi-\beta_{\sharp}) 
+ \frac{V^{0}(\beta_{\sharp};\tilde{\beta},\tilde{G})}{\tilde{\beta}-\beta_{\sharp}} (\tilde{\beta} - \Phi)
\notag \\
&> \frac{V^{0}(\tilde{\beta};\tilde{\beta},\tilde{G})}{\tilde{\beta}-\beta_{\sharp}} (\Phi-\beta_{\sharp})
=-\frac{V^{\infty}(\tilde{\beta})}{\tilde{\beta}-\beta_{\sharp}}(\Phi-\beta_{\sharp})
=k(\Phi-\beta_{\sharp}),
\label{V03}
\end{align}
where we have used $V^{0}(\beta_{\sharp};\tilde{\beta},\tilde{G})>0$ in deriving the second inequality;
we have used \eqref{Vbeta2} in deriving the first equality;
we have used \eqref{tbeta2} and the definition of $W^{\infty}$ in deriving the last equality.
Now it follows from \eqref{p2} and \eqref{V03} that
\begin{gather*}
V(\Phi;\tilde{\beta},\tilde{G})
=V^{\infty}(\Phi)+V^{0}(\Phi;\tilde{\beta},\tilde{G})
>-k(\Phi-\beta_{\sharp})+k(\Phi-\beta_{\sharp})
=0.
\end{gather*}
Hence, $V(\Phi;\tilde{\beta},\tilde{G})>0$ holds for all the three cases $\Phi\in(0,\beta)$, $\Phi\in[\beta,\beta_{\sharp}]$, and $\Phi\in(\beta_{\sharp},\tilde{\beta})$.
Thus $(\tilde{\beta},\tilde{G})$ satisfies \eqref{G-beta2}.

Finally, we show that $(\tilde{\beta},\tilde{G})$ satisfies \eqref{G-beta3}.
It is easy to see that $V(\Phi;\tilde{\beta},\tilde{G})=V^{\infty}(\Phi)=V(\Phi;\beta,G)$ holds on $[0,\tilde{\beta}/2]$, since $V^{0}(\Phi;\tilde{\beta},\tilde{G})\equiv 0$ holds on $[0,\tilde{\beta}/2]$.
From this and the assumption that $(\beta,G)$ satisfies \eqref{G-beta}, 
we have the first condition in \eqref{G-beta3}, i.e.
\begin{gather*}
\int_0^{\tilde{\beta}/2}\frac{d\Phi}{\sqrt{V(\Phi;\tilde{\beta},\tilde{G})}}
= \int_0^{\tilde{\beta}/2}\frac{d\Phi}{\sqrt{V^{\infty}(\Phi)}}
=\int_0^{\tilde{\beta}/2}\frac{d\Phi}{\sqrt{V(\Phi;\beta,G)}}
=+\infty.
\end{gather*}
To obtain another condition, it suffices to prove $\frac{d}{d\Phi}V(\tilde{\beta};\tilde{\beta},\tilde{G})
< 0$. Indeed, using this, \eqref{Vbeta2}, and the Taylor theorem, 
we obtain the desired condition, i.e.
\begin{gather*}
\int_{\tilde{\beta}/2}^{\tilde{\beta}}
\frac{d\Phi}{\sqrt{V(\Phi;\tilde{\beta},\tilde{G})}}
<+\infty.
\end{gather*}

Let us show $\frac{d}{d\Phi}V(\tilde{\beta};\tilde{\beta},\tilde{G}) < 0$.
It is seen from \eqref{tbeta2} and \eqref{p2} that for any $\Phi\in(\beta_{\sharp},\tilde{\beta})$,
\begin{gather*}
\frac{d}{d\Phi}V^{\infty}(\tilde{\beta})
=\lim_{\Phi\to\tilde{\beta}-0}
\frac{V^{\infty}(\Phi)-V^{\infty}(\tilde{\beta})}{\Phi-\tilde{\beta}}
\leq -k.
\end{gather*}
We note that the first inequality in \eqref{V03} is equivalent to 
\begin{gather*}
V^{0}(\Phi;\tilde{\beta},\tilde{G}) 
\geq  k' (\Phi-\tilde{\beta}) + V^{0}(\tilde{\beta};\tilde{\beta},\tilde{G}),
\quad 
k':=\frac{V^{0}(\tilde{\beta};\tilde{\beta},\tilde{G})-V^{0}(\beta_{\sharp};\tilde{\beta},\tilde{G})}{\tilde{\beta}-\beta_{\sharp}}.
\end{gather*}
This implies that
\begin{gather*}
\frac{d}{d\Phi}V^{0}(\tilde{\beta};\tilde{\beta},\tilde{G})
=\lim_{\Phi\to\tilde{\beta}-0}
\frac{V^{0}(\Phi;\tilde{\beta},\tilde{G})
-V^{0}(\tilde{\beta};\tilde{\beta},\tilde{G})}{\Phi-\tilde{\beta}}
\leq k'.
\end{gather*}
On the other hand, we also see from $V^{0}(\beta_{\sharp};\tilde{\beta},\tilde{G})>0$, \eqref{Vbeta2}, and \eqref{tbeta2} that 
\begin{gather*}
k'<\frac{V^{0}(\tilde{\beta};\tilde{\beta},\tilde{G})}
{\tilde{\beta}-\beta_{\sharp}}
=-\frac{V^{\infty}(\tilde{\beta})}{\tilde{\beta}-\beta_{\sharp}}
=-W^{\infty}(\tilde{\beta})
=k. 
\end{gather*}
From these, we arrive at 
\begin{gather*}
\frac{d}{d\Phi}V(\tilde{\beta};\tilde{\beta},\tilde{G})
=\frac{d}{d\Phi}V^{\infty}(\tilde{\beta})
+\frac{d}{d\Phi}V^{0}(\tilde{\beta};\tilde{\beta},\tilde{G})
\leq -k+k'<0.
\end{gather*}
Thus $(\tilde{\beta},\tilde{G})$ satisfies \eqref{G-beta3}.
The proof is complete.
\end{proof}

\subsection{An example with all the conditions in Theorem \ref{existence1}}\label{S2.5}

This subsection provides examples satisfying all the conditions in Theorem \ref{existence1}. 
For simplicity, we let $\alpha=0$. We choose $F_{\pm}^{\infty}$ and $G$ as
\begin{align*}
F^\infty_+(\xi)&\!:=\!\left\{
\begin{array}{ll}
\!\!\!\frac{1}{2e_+\sqrt{2q_+}} & \hbox{if} \
\xi\!\in\!\Bigl(\bigl[-2\sqrt{2q_+},-\sqrt{2q_+}\bigr]
\!\cup\!\bigl[\sqrt{2q_+},2\sqrt{2q_+}\bigr]\Bigr)
\times[0,1]^{n-1}, \\
\!\!\!0 & \hbox{otherwise}, 
\end{array}\right.
\\
F^\infty_-(\xi)&\!:=\!\left\{
\begin{array}{ll}
\!\!\!\frac{1}{2e_-\sqrt{2q_-}} & \hbox{if} \
\xi\!\in\!\Bigl(\!\!\bigl[-\frac{19}{10}\sqrt{2q_-},-\sqrt{2q_-}\bigr]
\!\cup\!\bigl[\frac{-1}{10}\sqrt{2q_-},\frac{1}{10}\sqrt{2q_-}\bigr]
\!\cup\!\bigl[\!\sqrt{2q_-},\frac{19}{10}\sqrt{2q_-}\bigr]\!\!\Bigr) \\
& \qquad\qquad\qquad\qquad
\times[0,1]^{n-1}, \\
\!\!\!0 & \hbox{otherwise}, 
\end{array}\right.
\\
G(\xi)&:=0.
\end{align*}
It is clear that 
\begin{gather*}
F^\infty_+\in L^1(\R^n), \quad
F^\infty_-\in L^1(\R^n)\cap L^3_{loc}(\R;L^1(\R^{n-1})), \quad
F^\infty_\pm\geq 0, \\
e_+\int_{\mathbb R^n}F^\infty_+(\xi)\,d\xi
=e_-\int_{\mathbb R^n}F^\infty_-(\xi)\,d\xi
=1,
\\
G\in L^1_{loc}((0,+\infty)\times\R^{n-1})
\cap L^3_{loc}([0,+\infty);L^1(\R^{n-1})), \quad
G\geq 0,
\end{gather*}
where the second line is the same as \eqref{netrual1}.
It remains to show \eqref{G-beta}. 

To this end, let us find $\beta>0$ so that $V(\beta;\beta,0)=V^{\infty}(\beta)=0$.  
It is seen by direct computation that
\begin{align*}
\rho^\infty(\Phi)
&:=e_+\rho^\infty_+(\Phi)-e_-\rho_-(\Phi)
\\
&=\Biggl\{
\begin{array}{ll}
\sqrt{4+\Phi}-\sqrt{1+\Phi}
-\sqrt{{\textstyle\frac{1}{100}}-\Phi}
-\sqrt{{\textstyle\frac{361}{100}}-\Phi}
+\sqrt{1-\Phi}
& \hbox{if} \ \
\Phi\in I_{1}:=[0,\frac{1}{100}], \\
\sqrt{4+\Phi}-\sqrt{1+\Phi}
-\sqrt{{\textstyle\frac{361}{100}}-\Phi}
+\sqrt{1-\Phi}
& \hbox{if} \ \
\Phi\in I_{2}:=[\frac{1}{100},1],
\end{array}
\\
V^\infty(\Phi)
&=\int_0^\Phi\rho^\infty(\varphi)\,d\varphi
\\
&=\Biggl\{
\begin{array}{ll}
\!\!\!\frac{\,3\,}{2}\Bigl\{
(4\!+\!\Phi)^{\frac{3}{2}}-(1\!+\!\Phi)^{\frac{3}{2}}
+({\textstyle\frac{1}{100}}\!-\!\Phi)^{\frac{3}{2}}
+({\textstyle\frac{361}{100}}\!-\!\Phi)^{\frac{3}{2}}
-(1\!-\!\Phi)^{\frac{3}{2}}-{\textstyle\frac{1286}{100}}
\Bigr\}
& \hbox{if} \ \
\Phi\in I_{1},
\\[7pt]
\!\!\!\frac{\,3\,}{2}\Bigl\{
(4\!+\!\Phi)^{\frac{3}{2}}-(1\!+\!\Phi)^{\frac{3}{2}}
+({\textstyle\frac{361}{100}}\!-\!\Phi)^{\frac{3}{2}}
-(1\!-\!\Phi)^{\frac{3}{2}}-{\textstyle\frac{1286}{100}}
\Bigr\} 
& \hbox{if} \ \
\Phi\in I_{2},
\end{array}
\end{align*}
where $\rho_{+}^{\infty}(\Phi)$, $\rho_{-}(\Phi)$, and $V^{\infty}(\Phi)$ are defined in \eqref{rho++}, \eqref{rho--}, and \eqref{Vinfty}, respectively.
We note that $\rho^\infty$ is continuous on $I_{1} \cup I_{2}$, and smooth except the points $\Phi=\frac{1}{100},1$;
$\rho^\infty$ is strictly increasing on $I_{1}$ and strictly decreasing on $I_{2}$.

Evaluating $\rho^\infty(\Phi)$ at $\Phi=0,\frac{1}{100},1$ and using the monotonicity of $\rho^\infty$, we obtain that 
$\rho^\infty\left(\frac{1}{100}\right)$ $>\rho^\infty(0)=0$ and $\rho^\infty(1)<0$.
These together with the intermediate value theorem ensure there exist $\beta_{0} \in {\mathring{I}_{2}}$ such that $\rho^\infty(\beta_{0})=0$. This fact implies that 
\begin{gather*}
\rho^\infty(\Phi)\left\{
\begin{array}{ll}
>0 & \hbox{if} \ \ \Phi\in(0,\beta_0), \\
=0 & \hbox{if} \ \ \Phi=0 \ \hbox{or} \ \beta_0, \\
<0 & \hbox{if} \ \ \Phi\in(\beta_0,1]. 
\end{array}\right.
\end{gather*}
Owing to  $\frac{d}{d\Phi}V^\infty(\Phi)=\rho^\infty(\Phi)$,
it follows that $V^\infty$ is strictly increasing on $[0,\beta_{0}]$ and strictly decreasing on $[\beta_{0},1]$.
On the other hand, evaluating $V^\infty(\Phi)$ at $\Phi=0,\beta_{0},1$ and using the monotonicity of $V^\infty$, we obtain that $V^\infty(\beta_{0})>V^\infty(0)=0$ and $V^\infty(1)<0$.
These together with the intermediate value theorem ensure there exist $\beta_{1} \in (\beta_{0},1)$ such that $V^\infty(\beta_{1})=0$. This fact implies that 
\begin{gather}\label{Vinfty1}
V^\infty(\Phi)\left\{
\begin{array}{ll}
>0 & \hbox{if} \ \ \Phi\in(0,\beta_1), \\
=0 & \hbox{if} \ \ \Phi=0 \ \hbox{or} \ \beta_1, \\
<0 & \hbox{if} \ \ \Phi\in(\beta_1,1]. 
\end{array}\right.
\end{gather}
Now we set $\beta:=\beta_{1}$. 

We complete the proof by showing \eqref{G-beta}. 
First we note that $V(\Phi;\beta,G)=V^\infty(\Phi)$ due to $G\equiv 0$.
It is clear that \eqref{G-beta1} holds.
Furthermore, \eqref{G-beta2} immediately follows from \eqref{Vinfty1}.
Let us show \eqref{G-beta3}. We know that $V^\infty(0)=\frac{d}{d\Phi}V^\infty(0)=\rho^\infty(0)=0$.
This with the aid of the Taylor theorem gives the equality in \eqref{G-beta3}.
The inequality in \eqref{G-beta3} also follows from $\frac{d}{d\Phi}V^\infty(\beta_1)=\rho^\infty(\beta_1)<0$.
Consequently, all the conditions in Theorem \ref{existence1} are valid.

\begin{rem}
These $F_{\pm}^{\infty}$, $G$, and $\beta$ also satisfy the assumption of Lemma \ref{nonunique2}, and hence 
we have infinite many solutions of the problem \eqref{VP2}.
These solutions have trapped ions.
\end{rem}

\section{Shock Waves}\label{S3}
The main purpose of this section is to investigate {\it shock waves} which
are special solutions of the problem \eqref{eq1}--\eqref{bc2} with
\begin{gather*}
\Phi^{l}>0 . 
\end{gather*}
We look for solutions of the \footnote{The constant $\alpha$ can be determined uniquely by $F_{\pm}^{l}$ and $F_{\pm}^{r}$ except the case that $\int_{\mathbb R^{n}} (F_{\pm}^{r}(\xi)-F_{\pm}^{l}(\xi))  d\xi=\int_{\mathbb R^{n}} \xi_{1} (F_{\pm}^{r}(\xi)-F_{\pm}^{l}(\xi)) d\xi=0$. See \eqref{nalpha0}.}form 
\begin{gather*}
f_{\pm}(t,x,\xi)=F_{\pm}(x-\alpha t, \xi), \quad \phi(t,x)=\Phi(x-\alpha t) \quad \text{for some $\alpha \in \mathbb R$,}
\end{gather*}
and the potential $\Phi$ is strictly decreasing.
It is seen by direct computations that $(F_{\pm},\Phi)=(F_{\pm}(X,\xi),\Phi(X))$ solves the following problem:
\begin{subequations}\label{nVP2}
\begin{gather}
(\xi_{1}-\alpha) \partial_{X} F_{\pm} \pm q_{\pm}\partial_{X} \Phi  \partial_{\xi_{1}} F_{\pm} =0, \ \ X \in \mathbb R, \ \xi \in \mathbb R^{n},
\label{neq3}
\\
\partial_{XX} \Phi 
= e_{+}\int_{\mathbb R^{n}} F_{+} d\xi - e_{-}\int_{\mathbb R^{n}} F_{-}  d\xi , \ \ X \in \mathbb R,
\label{neq4}\\
\lim_{X \to-\infty} F_{\pm} (X,\xi) =  F_{\pm}^{l}(\xi), \quad
\lim_{X \to+\infty} F_{\pm} (X,\xi) =  F_{\pm}^{r}(\xi), 
\quad \xi \in \mathbb R^{n},
\label{nbc3} \\
\lim_{X \to -\infty} \Phi (X) =  \Phi^{l}, \quad
\lim_{X \to +\infty} \Phi (X) =  0.
\label{nbc4} 
\end{gather}
\end{subequations}

It is worth pointing out that for the case that the potential $\Phi$ is strictly increasing, 
we can reduce it to the case that $\Phi$ is strictly decreasing by replacing 
$(F_{\pm},\Phi,e_{\pm},q_{\pm})$ by $(F_{\mp},-\Phi,e_{\mp},q_{\mp})$ in \eqref{nVP2}.
Therefore, we can focus on the study of the case that $\Phi$ is strictly decreasing.

\begin{defn}\label{nDefS1}
We say that $(F_{\pm},\Phi)$ is a solution of the problem \eqref{nVP2} if it satisfies the following:
\begin{enumerate}[(i)]
\item $F_{\pm} \in L^{1}_{loc}({\mathbb R}\times \mathbb R^{n}) \cap C({\mathbb R};L^{1}(\mathbb R^{n}))$
and $\Phi \in C^{2}({\mathbb R})$;
\item $F_{\pm}(X,\xi)\geq 0$, $\partial_{X}\Phi(X)<0$, and $\Phi(0)={\Phi^{l}}/{2}>0$;
%
%
\item $F_{\pm}$ solve
\begin{subequations}\label{nweak0}
\begin{gather}
(F_{\pm},(\xi_1-\alpha)\D_X\psi \pm q_{\pm} \D_X{\Phi}\D_{\xi_1}\psi)_{L^2({\mathbb R}\times {\mathbb R}^n)}
=0 \quad \hbox{for $\forall \psi \in C_0^1(\mathbb{R}\times \mathbb{R}^n)$},  
\label{nweak1}\\
\lim_{X\to -\infty} \| F_{\pm}(X,\cdot)-F_{\pm}^{l} \|_{L^{1}(\mathbb R^{n})}=\lim_{X\to +\infty} \| F_{\pm}(X,\cdot)-F_{\pm}^{r} \|_{L^{1}(\mathbb R^{n})}=0\text{\it ;}
\label{nweak2}
\end{gather}
\end{subequations}
\item $\Phi$ solves \eqref{neq4} with \eqref{nbc4} in the classical sense.
\end{enumerate}
\end{defn}

The condition (ii) does not allow the variant of the solution caused by translation.
The equation \eqref{nweak1} is a standard weak form of the equations \eqref{neq3}.
It is possible to replace {\it the classical sense} in the condition (iv) by {\it the weak sense}.
Indeed, a weak solution $\Phi$ of the problem of \eqref{neq4} with \eqref{nbc4}
is a classical solution due to $F_{\pm} \in C({\mathbb R};L^{1}(\mathbb R^{n}))$.
Any solution $(F_{\pm},\Phi)$ satisfies the neutral condition
\begin{gather}\label{nnetrual2}
\int_{-\infty}^{+\infty} \left( \int_{\mathbb R^{n}} e_{+}F_{+} - e_{-}F_{-} d\xi \right) dX= 0.
\end{gather}
This equality will be shown in Lemma \ref{nneed1} (see also Theorem \ref{nexistence1}).

We first discuss some necessary conditions which are used to state our main theorem.
To solve the Poisson equation \eqref{neq4} with \eqref{nbc4},
we must require the quasi-neutral condition
\begin{gather}\label{nnetrual1}
e_{+}\int_{\mathbb R^{n}} F_{+}^{l}(\xi) d\xi = e_{-}\int_{\mathbb R^{n}} F_{-}^{l}(\xi) d\xi, \quad
e_{+}\int_{\mathbb R^{n}} F_{+}^{r}(\xi) d\xi = e_{-}\int_{\mathbb R^{n}} F_{-}^{r}(\xi) d\xi.
\end{gather}  
By integrating \eqref{neq3} over ${\mathbb R} \times {\mathbb R}^{n}$, 
it is seen that we must choose $\alpha$ so that
\begin{gather}\label{nalpha0}
\alpha \int_{\mathbb R^{n}} (F_{\pm}^{r}(\xi)-F_{\pm}^{l}(\xi))  d\xi=\int_{\mathbb R^{n}} \xi_{1} (F_{\pm}^{r}(\xi)-F_{\pm}^{l}(\xi)) d\xi.
\end{gather}
Furthermore, the following are necessary conditions:
\begin{subequations}\label{Flr0}
\begin{align}
F_{+}^l(\xi_1+\alpha,\xi')=F_{+}^l(-\xi_1+\alpha,\xi'), &\quad (\xi_{1},\xi')\in(0,\sqrt{2q_{+}\Phi^{l}})\times\R^{n-1},
\label{Flr5}\\
F_{-}^r(\xi_1+\alpha,\xi')=F_{-}^r(-\xi_1+\alpha,\xi'),  &\quad (\xi_{1},\xi')\in(0,\sqrt{2q_{-}\Phi^{l}})\times\R^{n-1}
\label{Flr6}
\end{align}
and
\begin{align}
F_{+}^{r}(-\sqrt{\xi_{1}^{2}-2q_{+}\Phi^{l}}+\alpha,\xi')=F_{+}^{l}(\xi_{1}+\alpha,\xi'), &\quad (\xi_{1},\xi') \in (-\infty,-\sqrt{2q_{+}\Phi^{l}})\times\R^{n-1},
\label{Flr1} \\
F_{+}^{r}(\sqrt{\xi_1^2-2q_{+}\Phi^{l}}+\alpha,\xi')=F_{+}^{l}(\xi_{1}+\alpha,\xi'), & \quad  (\xi_{1},\xi') \in (\sqrt{2q_{+}\Phi^{l}},+\infty)\times\R^{n-1},
\label{Flr2} \\
F_{-}^{l}(-\sqrt{\xi_{1}^{2}-2q_{-}\Phi^{l}}+\alpha,\xi')=F_{-}^{r}(\xi_{1}+\alpha,\xi'), &\quad (\xi_{1},\xi') \in (-\infty,-\sqrt{2q_{-}\Phi^{l}})\times\R^{n-1},
\label{Flr3} \\
F_{-}^{l}(\sqrt{\xi_1^2-2q_{-}\Phi^{l}}+\alpha,\xi')=F_{-}^{r}(\xi_{1}+\alpha,\xi'), & \quad  (\xi_{1},\xi') \in (\sqrt{2q_{-}\Phi^{l}},+\infty)\times\R^{n-1}.
\label{Flr4}
\end{align}
\end{subequations}
These facts will be shown in Lemma \ref{nneed1}.

Our approach is the same as in Section \ref{S2}. We reduce the problem \eqref{nVP2} to the following problem:
\begin{equation}\lb{nphieq2}
(\D_{X} \Phi)^2=2V(\Phi), \quad 
\lim_{X \to -\infty} \Phi (X) =  \Phi^{l}, \quad
\lim_{X \to +\infty} \Phi (X) =  0.
\end{equation}
The Sagdeev potential $V$ is defined by
\begin{gather}
V=V(\Phi):=\int_{0}^{\Phi} \left(e_{+}\rho_{+}(\varphi) - e_{-}\rho_{-}(\varphi) \right)  d\vphi, \quad \Phi \geq 0,
\label{nV0}
\end{gather}
where the expected densities $\rho_{\pm}$ are defined by
\begin{align}
\rho_{+}(\Phi)
&=\int_{{\mathbb R}^n}F_{+}^l(\xi)\frac{|\xi_1-\alpha|}{\sqrt{(\xi_1-\alpha)^2-2q_{+}(\Phi^{l}-\Phi)}} \chi((\xi_1-\alpha)^2-2q_{+}(\Phi^{l}-\Phi)) \,d\xi,
\label{nrho++}\\
\rho_{-}(\Phi)
&=\int_{{\mathbb R}^n}F_{-}^r(\xi)\frac{|\xi_1-\alpha|}{\sqrt{(\xi_1-\alpha)^2-2q_{-}\Phi}} \chi((\xi_1-\alpha)^2-2q_{-}{\Phi}) \,d\xi.
\label{nrho--}
\end{align}
We will see in \eqref{nrho+'} and  \eqref{nrho-'} below that $\rho_{\pm}$ really express the densities of the ion and electron.
The functions $\rho_{\pm}$ are well-defined for $F_{+}^{l}, F_{-}^{r} \in L^{1}_{loc}(\mathbb R^{n})$ with $F_{+}^{l}, F_{-}^{r} \geq 0$, since all the integrants are nonnegative. 
The properties of $\rho_{\pm}$ are summarized in the following lemma.
We omit the proof, since it is the same as that of Lemma \ref{rhopm}.
\begin{lem}\label{nrhopm}
Let $\alpha \in \mathbb R$, $\Phi_{l}>0$, and $F_{\pm}^{l},F_{\pm}^{r} \in L^1(\R^n) \cap L_{loc}^{p}({\mathbb R};L^{1}(\mathbb R^{n-1}))$ for some $p>2$. 
Then the following estimates hold for $\Phi \in [0,\Phi^{l}]$:
\begin{align}
|\rho_{+}(\Phi)| &\leq  \sqrt{2} \|F_{+}^{l}\|_{L^{1}(\mathbb R^{n})} +C\|F_{+}^{l}\|_{L^{p}(-2\sqrt{q_{+}\Phi^{l}}+\alpha,2\sqrt{q_{+}\Phi^{l}}+\alpha;L^{1}(\mathbb R^{n-1}))},
\label{nrho+}\\
|\rho_{-}(\Phi)|  &\leq  \sqrt{2} \|F_{-}^{r}\|_{L^{1}(\mathbb R^{n})} +C\|F_{-}^{r}\|_{L^{p}(-2\sqrt{q_{-}\Phi^{l}}+\alpha,2\sqrt{q_{-}\Phi^{l}}+\alpha;L^{1}(\mathbb R^{n-1}))},
\label{nrho-}
\end{align}
where $C$ is a positive constant depending only on $\alpha$, $p$, $\Phi^{l}$, and $q_{\pm}$.
Furthermore, the functions $\rho_{\pm}(\cdot)$ belong to $C([0,\Phi^{l}])$.
\end{lem}

We are now in a position to state our main theorem.

\begin{thm} \label{nexistence1}
Let $p>2$, $\alpha \in \mathbb R$, $\Phi_{l}>0$, $F_{\pm}^{l},F_{\pm}^{r} \in L^1(\R^n) \cap L_{loc}^{p}({\mathbb R};L^{1}(\mathbb R^{n}))$,
$|\xi_{1}|F_{\pm}^{l}$, $|\xi_{1}|F_{\pm}^{r} \in L^1(\R^n)$, and
$F_{\pm}^{l},F_{\pm}^{r} \geq 0$. 
Assume that the necessary conditions \eqref{nnetrual1}--\eqref{Flr0} hold.
Then the problem \eqref{nVP2} has a solution $(F_{\pm},\Phi)$ 
if and only if $\Phi^{l}$ satisfies
\begin{subequations}\label{Phil}
\begin{gather}
V(\Phi)>0 \ \ \text{for $\Phi \in (0,\Phi^{l})$}, \quad  V(\Phi^{l})=0,
\label{Phil1}\\
\int_{0}^{\Phi^{l}/2} \frac{d\Phi}{\sqrt{V(\Phi)}} =+ \infty, \quad 
\int_{\Phi^{l}/2}^{\Phi^{l}} \frac{d\Phi}{\sqrt{V(\Phi)}} =+ \infty.
\label{Phil2}
\end{gather}
\end{subequations}
Furthermore, the solution is unique, the neutral condition \eqref{nnetrual2} holds, 
and $F_{\pm}$ can be written by 
\begin{align}
F_{+}(X,\xi)
&=F_{+}^l(-\sqrt{(\xi_1-\alpha)^2+2q_{+}(\Phi^{l}-\Phi(X))}+\alpha,\xi')\chi(-(\xi_1-\alpha))
\notag \\
&\quad +F_{+}^l(\sqrt{(\xi_1-\alpha)^2+2q_{+}(\Phi^{l}-\Phi(X))}+\alpha,\xi')\chi(\xi_1-\alpha),
\label{nfform+} \\
F_{-}(X,\xi)
&=F_{-}^r(-\sqrt{(\xi_1-\alpha)^2+2q_{-}\Phi(X)}+\alpha,\xi')\chi(-(\xi_1-\alpha))
\notag\\
&\quad +F_{-}^r(\sqrt{(\xi_1-\alpha)^2+2q_{-}\Phi(X)}+\alpha,\xi')\chi(\xi_1-\alpha). 
\label{nfform-}
\end{align}
\end{thm}

\begin{rem} {\rm
We do not use the information of solution to write the necessary and sufficient condition \eqref{Phil}.
Indeed, it depends only on $e_{\pm}$, $q_{\pm}$, $F_{\pm}^{l}$, and $F_{\pm}^{r}$. 
In subsection \ref{S3.3}, we will provide an example satisfying all the conditions in Theorem \ref{nexistence1}, and hence have a shock wave. 
Furthermore, the necessary and sufficient condition \eqref{Phil} is easy to check by computers as follows. 
If $V(\Phi)$ is a $C^{2}$-funciton around $\Phi=0,\Phi^{l}$,  
the condition \eqref{Phil2} is equivalent to that $\frac{d}{d\Phi}V(0)=\frac{d}{d\Phi}V(\Phi^{l})=0$ thanks to the Taylor theorem. 
Therefore, \eqref{Phil} is verified if the graph of $V(\Phi)$ is drawn as in Figure \ref{nfigV} below.
It is also worth pointing out that the shock wave is always unique whereas the solitary wave is not unique for almost cases. 
This difference comes from the monotonicity of the potential $\Phi$. It cannot trap ions.
}
\end{rem}

\begin{figure}[H]
\begin{center}
    \includegraphics[width=8cm, bb=0 0 1720 1239]{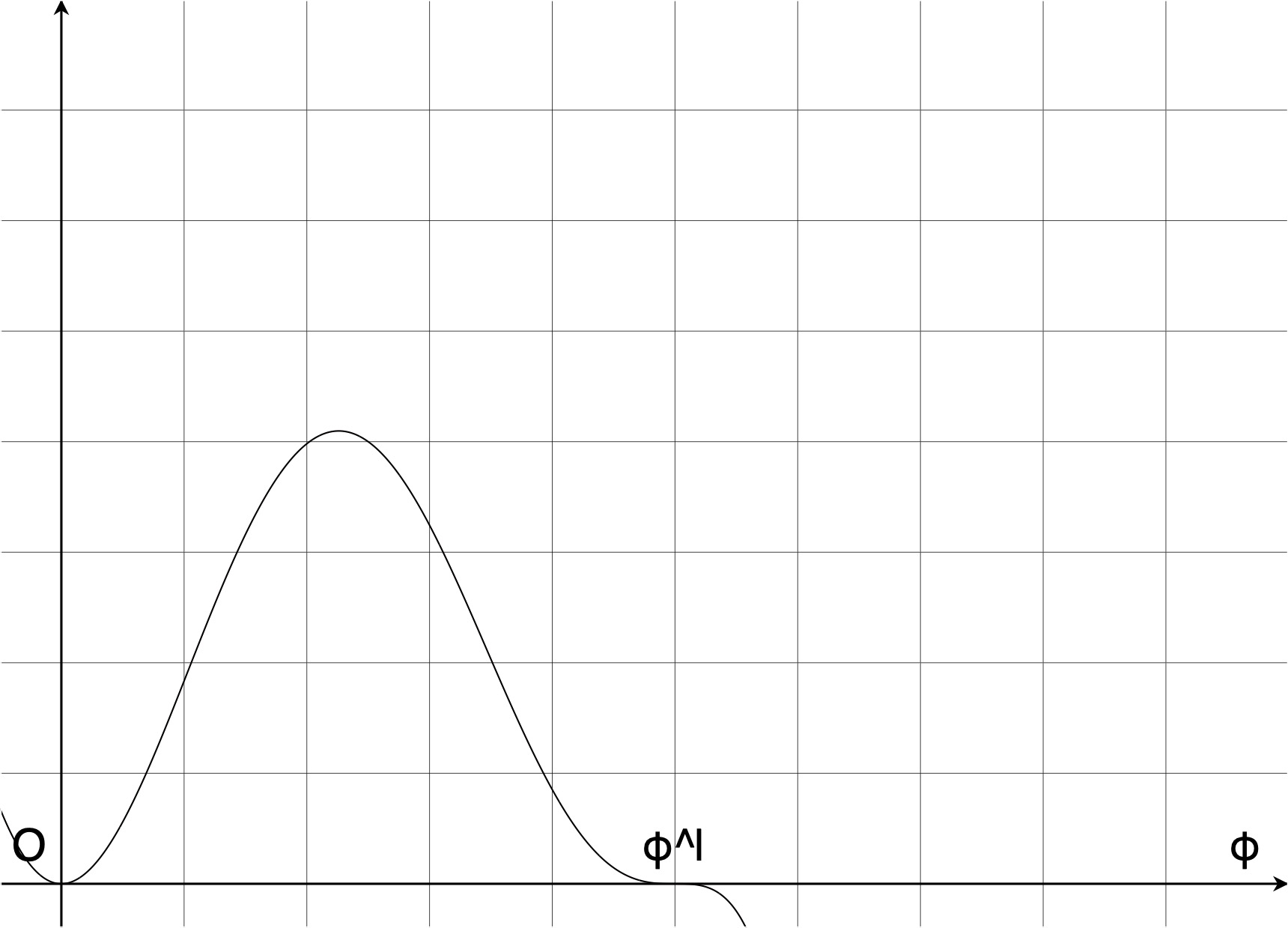}
\end{center}
  \caption{the graph of $V$}  
  \label{nfigV}
\end{figure}

As mentioned in Section \ref{S2}, it is often assumed in plasma physics that the electron density obeys the Boltzmann relation $\rho_{-}:=\int_{\mathbb R^{n}} f_{-} d\xi=\rho e^{-\kappa\phi}$ in \eqref{eq2}.
The corresponding shock waves solve
\begin{subequations}\label{nVP3}
\begin{gather}
(\xi_{1}-\alpha) \partial_{X} F_{+} + q_{+}\partial_{X} \Phi  \partial_{\xi_{1}} F_{+} =0, \ \ X \in \mathbb R, \ \xi \in \mathbb R^{n},
\label{neq5}
\\
\partial_{XX} \Phi 
= e_{+}\int_{\mathbb R^{n}} F_{+} d\xi - e_{-}\rho e^{-\kappa \Phi}, \ \ X \in \mathbb R,
\label{neq6}\\
\lim_{X \to-\infty} F_{+} (X,\xi) =  F_{+}^{l}(\xi), \quad
\lim_{X \to+\infty} F_{+} (X,\xi) =  F_{+}^{r}(\xi), 
\quad \xi \in \mathbb R^{n},
\label{nbc5} \\
\lim_{X \to -\infty} \Phi (X) =  \Phi^{l}, \quad
\lim_{X \to +\infty} \Phi (X) =  0.
\label{nbc6} 
\end{gather}
\end{subequations}
Theorem \ref{nexistence1} is also applicable to the problem \eqref{nVP3} by suitably choosing $F_{-}^{l}$ and $F_{-}^{r}$ in the problem \eqref{nVP2}. 
Namely, the following corollary holds. We omit the proof, since it is the same as that of Corollary \ref{cor1}.

\begin{cor}
Suppose that $(F_{\pm},\Phi)$ is a solution of the problem \eqref{nVP2} with 
\begin{gather*}
F_{-}^{r}(\xi)=\frac{\rho}{\displaystyle \int_{\mathbb R^{n}} e^{\frac{-\kappa}{2q_{-}}|\xi|^2} d\xi}
e^{\frac{-\kappa}{2q_{-}}\{(\xi_{1}-\alpha)^{2}+|\xi'|^{2}\}},
\quad
F_{-}^{l}(\xi)=\frac{\rho}{\displaystyle \int_{\mathbb R^{n}} e^{\frac{-\kappa}{2q_{-}}|\xi|^2} d\xi}
e^{\frac{-\kappa}{2q_{-}}\{(\xi_{1}-\alpha)^{2}+2q_{-}\Phi^{l}+|\xi'|^{2}\}}.
\end{gather*}
Then $(F_{+},\Phi)$ is a solution of the problem \eqref{nVP3}.
\end{cor}

This section is organized as follows. 
In subsection \ref{S3.1}, we find necessary conditions for the solvability of the problem \eqref{nVP2}.
Subsection \ref{S3.2} gives the proof of the unique solvability stated in Theorem \ref{nexistence1}.
Subsection \ref{S3.3} provides an example with all the conditions in Theorem \ref{nexistence1}.

\subsection{Necessary conditions and uniqueness}\label{S3.1}

In this section, we investigate necessary conditions for the solvability of the problem \eqref{nVP2} as well as the uniqueness provided that the solution exists.

\begin{lem}\label{nneed1}
Let $\alpha \in \mathbb R$, $\Phi_{l}>0$, $F_{\pm}^{l},F_{\pm}^{r} \in L^{1}(\mathbb R^{n})$, $|\xi_{1}|F_{\pm}^{l},|\xi_{1}|F_{\pm}^{r} \in L^1(\R^n)$, and $F_{\pm}^{l},F_{\pm}^{r}\geq 0$.
Suppose that  the problem \eqref{nVP2} has a solution $(F_{\pm},\Phi)$.
Then the conditions \eqref{nnetrual2} and \eqref{nnetrual1} hold;
$F_{\pm}^{l}$ and $F_{\pm}^{r}$ satisfies the condition \eqref{Flr0}; 
$\alpha$ satisfies \eqref{nalpha0};
$F_{\pm}$ are written as \eqref{nfform+} and \eqref{nfform-} by $\Phi$;
$\alpha$ satisfies \eqref{nalpha0}; 
$\rho_{\pm} \in C([0,\Phi^{l}])$;
$\Phi$ solves the problem \eqref{nphieq2} with $\Phi(0)=\Phi^{l}/2$;
$V$ satisfies \eqref{Phil}; 
the solution $(F_{\pm},\Phi)$ is unique.
\end{lem} 

\begin{proof}
We prove \eqref{nnetrual2}.
Owing to \eqref{nweak2} and $F_{\pm} \in C({\mathbb R}; L^{1}(\mathbb R^{n}))$, 
it follows from \eqref{neq4} that $\D_{XX} \Phi$ is bounded and 
therefore $\D_{X} \Phi$ is uniformly continuous on ${\mathbb R}$.
Then the fact together with \eqref{nbc4} and the condition (ii) in Definition \ref{nDefS1} ensures that
\begin{gather}\label{nlim1}
\lim_{|X| \to +\infty}\D_{X}\Phi(X)=0.
\end{gather}
Integrating \eqref{neq4} from $-\infty$ to $+\infty$ and using \eqref{nlim1}, we arrive at \eqref{nnetrual2}.
Let us next show \eqref{nnetrual1}. It is seen from \eqref{nnetrual2} 
that there exists a sequence $\{X_{k}\}_{k=1}^{+\infty}$ such that $X_{k} \to +\infty$
and $\lim_{k \to +\infty}(e_{+}\int_{\mathbb R^{n}}F_{+} d\xi-e_{-}\int_{\mathbb R^{n}}F_{+} d\xi)(X_{k})=0$. 
On the other hand, it follows from \eqref{nweak2} that
$\lim_{X \to +\infty}(e_{+}\int_{\mathbb R^{n}}F_{+} d\xi-e_{-}\int_{\mathbb R^{n}}F_{+} d\xi)=e_{+}\int_{\mathbb R^{n}}F_{+}^{\infty} d\xi - e_{-}\int_{\mathbb R^{n}}F_{-}^{\infty} d\xi$.
Therefore, \eqref{nnetrual1} must hold.

Next we show that $F_{\pm}^{r}$ and $F_{\pm}^{l}$ satisfies the conditions in \eqref{Flr0},
and $F_{\pm}$ are written as \eqref{nfform+} and \eqref{nfform-} by $\Phi$.
Regarding $\Phi$ as a given function and then applying the characteristics method to \eqref{nweak1}, 
we see that the values of $F_{+}$ must be the same on the following characteristics curve:
\begin{gather*}
\frac{1}{2}(\xi_{1}-\alpha)^{2}-q_{+}\Phi(X)=c,
\end{gather*}
where $c$ is some constant. We draw the illustration of characteristics for $\alpha=0$ in Figure \ref{nfig+}.
\begin{figure}[H]
\begin{center}
    \includegraphics[width=9.5cm, bb=0 0 1720 1220]{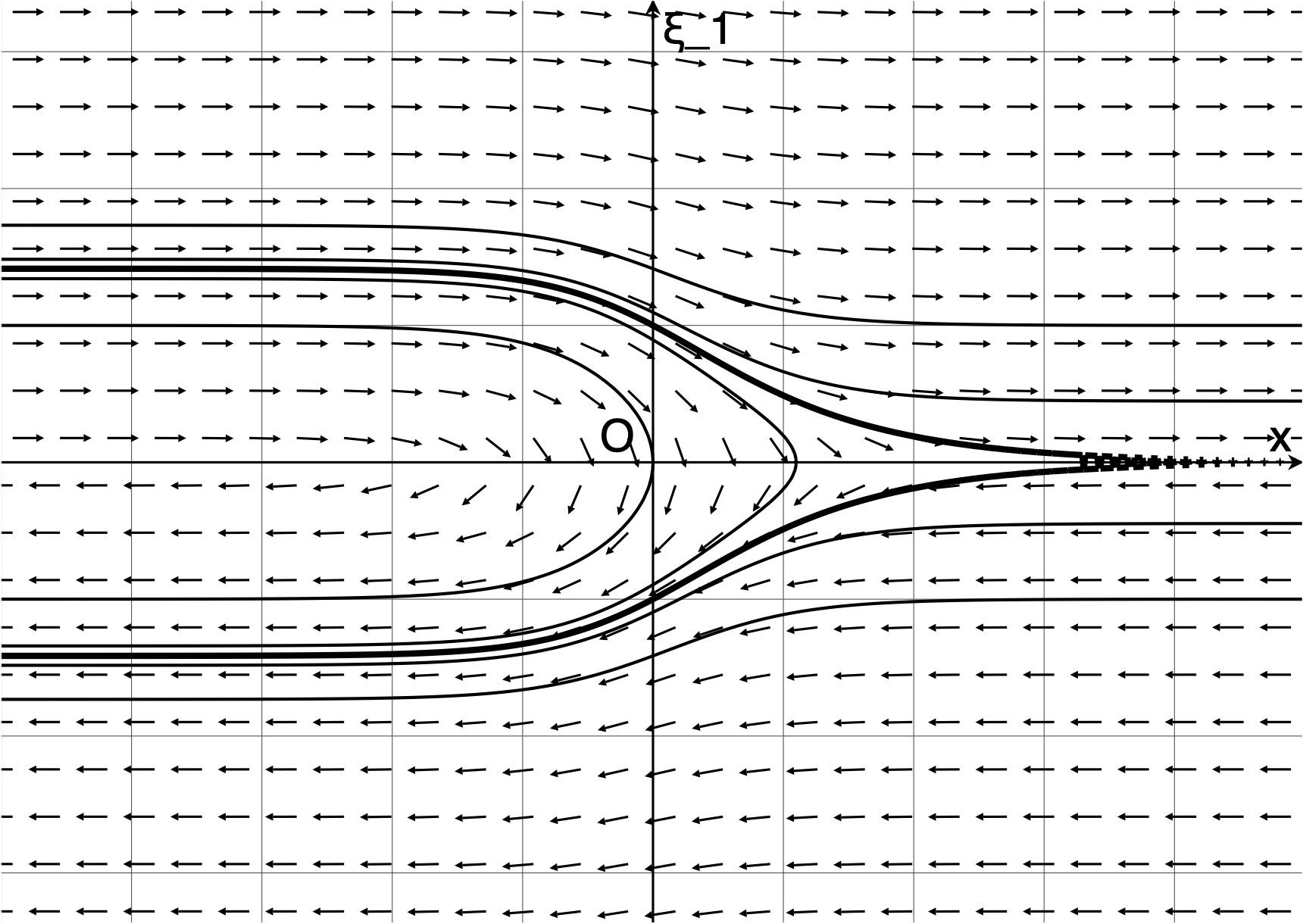}
\end{center}
  \caption{the characteristics of $F_{+}$}  
  \label{nfig+}
\end{figure}
The curve can also be written as
\begin{gather*}
\frac{1}{2}(\xi_{1}-\alpha)^{2}+q_{+}(\Phi^{l}-\Phi(X))=c+q_{+}\Phi^{l}.
\end{gather*}
It tells us that
\begin{align}
F_{+}(y,-\sqrt{\eta_{1}^{2}-2q_{+}(\Phi^{l}-\Phi(y))}+\alpha,\eta')=F_{+}^{l}(\eta_{1}+\alpha,\eta'), &\quad (y,\eta_{1},\eta') \in {\cal X}_{+}^{1},
\label{nchara1}
 \\
F_{+}(y,\pm\sqrt{\eta_1^2-2q_{+}(\Phi^{l}-\Phi(y))}+\alpha,\eta')=F_{+}^{l}(\eta_{1}+\alpha,\eta'), &\quad (y,\eta_{1},\eta') \in {\cal X}_{+}^{2},
\label{nchara2}
 \\
F_{+}(y,\sqrt{\eta_1^2-2q_{+}(\Phi^{l}-\Phi(y))}+\alpha,\eta')=F_{+}^{l}(\eta_{1}+\alpha,\eta'), & \quad  (y,\eta_{1},\eta') \in {\cal X}_{+}^{3},
\label{nchara3}
\end{align}
where 
\begin{align*}
{\cal X}_{+}^{1}&:={\R}\times(-\infty,-\sqrt{2q_{+}\Phi^{l}})\times\R^{n-1},
\\
{\cal X}_{+}^{2}&:=\{(y,\eta_1,\eta')\in\R\times\R\times\R^{n-1}\;|\;2q_{+}(\Phi^{l}-\Phi(y))<\eta_1^2<2q_{+}\Phi^{l}\}, \quad
\\
{\cal X}_{+}^{3}&:={\R}\times(\sqrt{2q_{+}\Phi^{l}},+\infty)\times\R^{n-1}. 
\end{align*}
Now we conclude from these three equalities that $F_{+}$ must be written as \eqref{nfform+}, i.e.
\begin{align*}
F_{+}(X,\xi)
&=F_{+}^l(-\sqrt{(\xi_1-\alpha)^2+2q_{+}(\Phi^{l}-\Phi(X))}+\alpha,\xi')\chi(-(\xi_1-\alpha))
\\
&\quad +F_{+}^l(\sqrt{(\xi_1-\alpha)^2+2q_{+}(\Phi^{l}-\Phi(X))}+\alpha,\xi')\chi(\xi_1-\alpha).
\end{align*}
Furthermore, \eqref{nchara2} means that \eqref{Flr5} must hold.
Letting $y \to +\infty$ in \eqref{nchara1} and \eqref{nchara3}, we also see that $F_{+}^{l}$ and $F_{+}^{r}$ must satisfy
\begin{align*}
F_{+}^{r}(-\sqrt{\eta_{1}^{2}-2q_{+}\Phi^{l}}+\alpha,\eta')=F_{+}^{l}(\eta_{1}+\alpha,\eta'), 
 \\
F_{+}^{r}(\sqrt{\eta_1^2-2q_{+}\Phi^{l}}+\alpha,\eta')=F_{+}^{l}(\eta_{1}+\alpha,\eta'), 
\end{align*}
which mean \eqref{Flr1} and \eqref{Flr2}.

Next we treat $F_{-}$. The values of $F_{-}$ must be the same on the following characteristics curve:
\begin{gather*}
\frac{1}{2}(\xi_{1}-\alpha)^{2}+q_{-}\Phi(X)=c,
\end{gather*}
where $c$ is some constant. We draw the illustration of characteristics for $\alpha=0$ in Figure \ref{nfig-}.
\begin{figure}[H]
\begin{center}
    \includegraphics[width=9.5cm, bb=0 0 1720 1220]{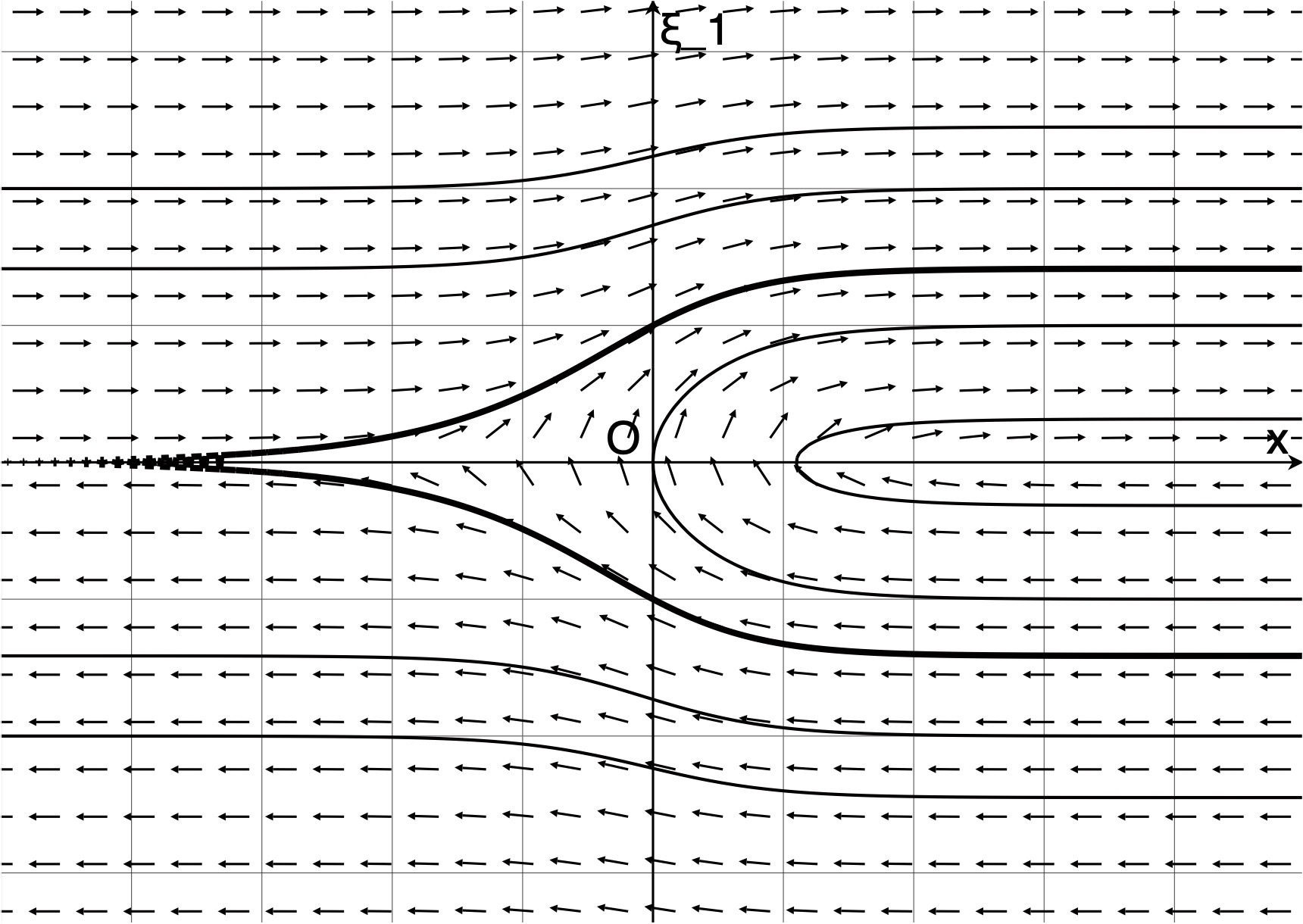}
\end{center}
  \caption{the characteristics of $F_{-}$}  
  \label{nfig-}
\end{figure}
It tells us that
\begin{align}
F_{-}(y,-\sqrt{\eta_{1}^{2}-2q_{-}\Phi(y)}+\alpha,\eta')=F_{-}^{r}(\eta_{1}+\alpha,\eta'), &\quad (y,\eta_{1},\eta') \in {\cal X}_{-}^{1},
\label{nchara4}
 \\
F_{-}(y,\pm\sqrt{\eta_1^2-2q_{-}{\Phi}(y)}+\alpha,\eta')=F_{-}^{r}(\eta_{1}+\alpha,\eta'), &\quad (y,\eta_{1},\eta') \in {\cal X}_{-}^{2},
\label{nchara5}
 \\
F_{-}(y,\sqrt{\eta_1^2-2q_{-}{\Phi}(y)}+\alpha,\eta')=F_{-}^{r}(\eta_{1}+\alpha,\eta'), & \quad  (y,\eta_{1},\eta') \in {\cal X}_{-}^{3},
\label{nchara6}
\end{align}
where 
\begin{align*}
{\cal X}_{-}^{1}&:={\R}\times(-\infty,-\sqrt{2q_{-}\Phi^{l}})\times\R^{n-1},
\\
{\cal X}_{-}^{2}&:=\{(y,\eta_1,\eta')\in\R\times\R\times\R^{n-1}\;|\;2q_{-}{\Phi}(y)<\eta_1^2<2q_{-}\Phi^{l}\}, \quad
\\
{\cal X}_{-}^{3}&:={\R}\times(\sqrt{2q_{-}\Phi^{l}},+\infty)\times\R^{n-1}. 
\end{align*}
Now we conclude from these three equalities that $F_{-}$ must be written as \eqref{nfform-}, i.e.
\begin{align*}
F_{-}(X,\xi)
&=F_{-}^r(-\sqrt{(\xi_1-\alpha)^2+2q_{-}\Phi(X)}+\alpha,\xi')\chi(-(\xi_1-\alpha))
\\
&\quad +F_{-}^r(\sqrt{(\xi_1-\alpha)^2+2q_{-}\Phi(X)}+\alpha,\xi')\chi(\xi_1-\alpha). 
\end{align*}
Furthermore, \eqref{nchara5} means that \eqref{Flr6} must hold.
Letting $y \to -\infty$ in \eqref{nchara4} and \eqref{nchara6}, we also see that $F_{+}^{l}$ and $F_{+}^{r}$ must satisfy
\begin{align*}
F_{-}^{l}(-\sqrt{\eta_{1}^{2}-2q_{-}\Phi^{l}}+\alpha,\eta')=F_{-}^{r}(\eta_{1}+\alpha,\eta'), 
 \\
F_{-}^{l}(\sqrt{\eta_1^2-2q_{-}\Phi^{l}}+\alpha,\eta')=F_{-}^{r}(\eta_{1}+\alpha,\eta'), 
\end{align*}
which mean \eqref{Flr3} and \eqref{Flr4}.

In order to obtain \eqref{nalpha0}, it suffices to show that 
\begin{gather}\label{alpha'}
\int_{\mathbb R^{n}}  (\cF_{\pm}^{r}(\xi)-\cF_{\pm}^{l}(\xi)) d\xi =0,
\end{gather}
where $\cF_{\pm}^{r}(\xi):=\xi_{1} F_{\pm}^{r}(\xi_{1}+\alpha,\xi')$ and $\cF_{\pm}^{l}(\xi):=\xi_{1}F_{\pm}^{l}(\xi_{1}+\alpha,\xi')$.
Indeed, 
\begin{gather*}
\int_{\mathbb R^{n}} \xi_{1} (F_{\pm}^{r}(\xi)-F_{\pm}^{l}(\xi)) d\xi
-\alpha\int_{\mathbb R^{n}} (F_{\pm}^{r}(\xi)-F_{\pm}^{l}(\xi)) d\xi
=\int_{\mathbb R^{n}}  (\cF_{\pm}^{r}(\xi)-\cF_{\pm}^{l}(\xi)) d\xi =0.
\end{gather*}
Let us show only \eqref{alpha'} with $+$, since another case can be similarly shown.
We divide the integrals as follows:
\begin{gather*}
\int_{\mathbb R^{n}}  {\cal F}_{+}^{r} d\xi 
=I_{1}+I_{2}
:= \int_{\xi_{1}<0} {\cal F}_{+}^{r} d\xi
+\int_{\xi_{1}>0} {\cal F}_{+}^{r} d\xi,
\\
\int_{\mathbb R^{n}}  {\cal F}_{+}^{l} d\xi 
=J_{1}+J_{2}+J_{3}
:= \int_{\xi_{1}<-\sqrt{2q_{+}\Phi^{l}}} {\cal F}_{+}^{l} d\xi
+ \int_{|\xi_{1}|<\sqrt{2q_{+}\Phi^{l}}} {\cal F}_{+}^{l} d\xi
+\int_{\xi_{1}>\sqrt{2q_{+}\Phi^{l}}} {\cal F}_{+}^{l} d\xi.
\end{gather*}
Using \eqref{Flr1} and the change of variable $\sqrt{\xi_1^{2}-2q_{+}\Phi^{l}}=-\zeta_{1}$, we arrive at $I_{1}=J_{1}$. 
Similarly, $I_{2}=J_{3}$ holds thanks to \eqref{Flr2}.
Furthermore,  $J_{2}=0$ immediately follows from \eqref{Flr5}.
From these, we deduce \eqref{alpha'} with $+$.
Thus $\alpha$ must satisfy \eqref{nalpha0}.

Now we can reduce the problem \eqref{nVP2} to a problem of an ordinary differential equation only for $\Phi$.
Integrating \eqref{nfform+} over $\mathbb R^{n}$ and using the change of variables $\sqrt{(\xi_1-\alpha)^2+2q_{+}(\Phi^{l}-\Phi)}$ $=-(\zeta_{1}-\alpha)$ and $\sqrt{(\xi_1-\alpha)^2+2q_{+}(\Phi^{l}-\Phi)}=\zeta_{1}-\alpha$ 
for the first and second terms on the right hand side, respectively, we see that 
\begin{align}
& \int_{{\mathbb R}^n}F_{+}(X,\xi) d\xi
\notag \\
&=\int_{{\mathbb R}^n}F_{+}^l(\xi)\frac{|\xi_1-\alpha|}{\sqrt{(\xi_1-\alpha)^2-2q_{+}(\Phi^{l}-\Phi(X))}} \chi((\xi_1-\alpha)^2-2q_{+}(\Phi^{l}-\Phi(X))) \,d\xi
\notag \\
&=\rho_{+}(\Phi(X)),
\label{nrho+'}
\end{align}
where $\rho_{+}$ is the same function defined in \eqref{nrho++}.
On the other hand, integrating \eqref{nfform-} over $\mathbb R^{n}$ and using the change of variables
$\sqrt{(\xi_{1}-\alpha)^{2}+2q_{-}\Phi}=-(\zeta_{1}-\alpha)$ and 
$\sqrt{(\xi_{1}-\alpha)^{2}+2q_{-}\Phi}=\zeta_{1}-\alpha$ for the first and second terms on the right hand side, 
respectively, we see that
\begin{align}
\int_{{\mathbb R}^n}F_{-}(X,\xi) d\xi
&=\int_{{\mathbb R}^n}F_{-}^r(\xi)\frac{|\xi_1-\alpha|}{\sqrt{(\xi_1-\alpha)^2-2q_{-}\Phi(X)}} \chi((\xi_1-\alpha)^2-2q_{-}{\Phi(X)}) \,d\xi
\notag \\
&=\rho_{-}(\Phi(X)),
\label{nrho-'}
\end{align}
where $\rho_{-}$ is the same function defined in \eqref{nrho--}.
Substituting \eqref{nrho+'} and \eqref{nrho-'} into \er{neq4}, we arrive at
\begin{equation}\lb{nphieq1} 
\partial_{XX} \Phi 
= e_{+}\rho_{+}(\Phi) - e_{-}\rho_{-}(\Phi), \quad X\in \mathbb R.
\end{equation}
Now we claim that $\rho_{\pm} \in C([0,\Phi^{l}])$. 
Indeed, owing to $ F_{\pm} \in C({\R};L^{1}(\R^{n}))$ in Definition \ref{nDefS1} and \eqref{nweak2}, there hold that
\begin{align*}
\rho_{\pm}(\Phi(X))= \| F_{\pm}(X) \|_{L^{1}} \in C({\R}),
\ \
\lim_{X \to -\infty}\| F_{\pm}(X) \|_{L^{1}} = \| F_{\pm}^{l} \|_{L^{1}},
\ \
\lim_{X \to +\infty}\| F_{\pm}(X) \|_{L^{1}} = \| F_{\pm}^{r} \|_{L^{1}}.
\end{align*}
These with \footnote{To find $\rho_{+}(0)$, we have used \eqref{Flr1}--\eqref{Flr2}.
To find $\rho_{-}(\Phi^{l})$, we have used \eqref{Flr3}--\eqref{Flr4}.}$\rho_{\pm}(0)=\int_{\mathbb R^{n}} F_{\pm}^{r} d\xi$ and $\rho_{\pm}(\Phi^{l})=\int_{\mathbb R^{n}} F_{\pm}^{l} d\xi$
imply $\rho_{\pm} \in C([0,\Phi^{l}])$ with the aid of \eqref{nbc4},
$ \Phi \in C(\R)$, and $\partial_{X}\Phi(X)< 0$ in Definition \ref{nDefS1}.
Thus the claim is valid.
Multiply \er{nphieq1} by $\D_{X}\Phi$, integrate it over $(X,+\infty)$, and use \eqref{nbc4} and \eqref{nlim1} to obtain the first equation in \eqref{nphieq2}, i.e.
\begin{equation*}
(\D_{X} \Phi)^2=2V(\Phi), 
\quad V(\Phi)=\int_{0}^{\Phi} \left(e_{+}\rho_{+}(\varphi) - e_{-}\rho_{-}(\varphi) \right)  d\vphi,
\end{equation*}
where $V$ is the same function defined in \eqref{nV0}.
Thus $\Phi$ must solve the problem \eqref{nphieq2} with $\Phi(0)=\Phi^{l}/2$.

We prove that $V$ satisfies \eqref{Phil}.
From \eqref{nphieq2}, \eqref{nlim1}, and $\partial_{x} \Phi(X)< 0$ in Definition \ref{nDefS1}, the following holds:
\begin{gather*}
V(\Phi)>0 \ \ \text{for $\Phi \in (0,\Phi^{l})$}, \quad  V(\Phi^{l})=0,
\end{gather*}
which are the same conditions as the conditions in \eqref{Phil1}.
Thus \eqref{Phil1} must hold.
Let us also show \eqref{Phil2}, i.e. 
\begin{gather*}
\int_{0}^{\Phi^{l}/2} \frac{d\Phi}{\sqrt{V(\Phi)}} =+ \infty, \quad 
\int_{\Phi^{l}/2}^{\Phi^{l}} \frac{d\Phi}{\sqrt{V(\Phi)}} =+ \infty.
\end{gather*}
Using $\Phi(0)=\Phi^{l}/2$ in Definition \ref{nDefS1} and the first equation in \eqref{nphieq2}, we see that 
the first integral is unbounded as follows:
\begin{align*}
\int_{0}^{\Phi^{l}/2} \frac{d\Phi}{\sqrt{2V(\Phi)}}
=\int_{0}^{+\infty} \frac{-\partial_{X} \Phi (X)}{\sqrt{2V(\Phi(X))} } dX
=\int_{0}^{+\infty} 1dX
=+\infty.
\end{align*}
Similarly, it is seen that the second integral is unbounded. 
Thus \eqref{Phil2} must hold.

Finally, let us show the uniqueness of solutions $(F_{\pm},\Phi)$.
Recall that $\Phi$ solves the problem \eqref{nphieq2} with $\Phi(0)=\Phi^{l}/2$.
The problem is equivalent to the following owing to $\partial_{X}\Phi<0$:
\begin{gather}
\D_{X}{\Phi}=-\sqrt{2V(\Phi)}, \quad 
\lim_{X \to -\infty} \Phi (X) =  \Phi^{l}, \quad
\lim_{X \to +\infty} \Phi (X) =  0, \quad 
\Phi(0)=\frac{1}{2}\Phi^{l}.
\label{neq1}
\end{gather}
On the other hand, $V$ belong to $C^{1}([0,\Phi^{l}])$ due to $\rho_{\pm} \in C([0,\Phi^{l}])$ shown above.
Moreover, $\sqrt{V}$ is Lipschitz continuous on $[\ve,\Phi^{l}-\ve]$ for any suitably small $\ve>0$ 
thanks to \eqref{Phil1} shown above. 
Therefore, the solution of \eqref{neq1} is unique.
Suppose that $(F_{\pm}^{1},\Phi^{1})$ and  $(F_{\pm}^{2},\Phi^{2})$ are solutions of \eqref{nVP2}.
Then $\Phi^{1}=\Phi^{2}$ follows from the uniqueness of solutions of \eqref{neq1}.
By using this and the forms \eqref{nfform+}--\eqref{nfform-}, it is also seen that $F_{\pm}^{1}=F_{\pm}^{2}$.
The proof is complete.
\end{proof}

\subsection{Solvability}\label{S3.2}

In this subsection, we prove Theorem \ref{nexistence1}.

\begin{proof}[Proof of Theorem \ref{nexistence1}.]
Due to Lemma \ref{nneed1}, it suffices to show the solvability of the problem \eqref{nVP2} provided that \eqref{Phil} holds.
We first construct $\Phi$ by solving the problem \eqref{neq1}, which is equivalent to the problem \eqref{nphieq2} with $\Phi(0)=\Phi^{l}/2$. To this end, we consider the following problem:
\begin{gather}\label{neq1'}
\D_{X}{\Phi}=-\sqrt{2V(\Phi)}, \quad 
\Phi(0)=\frac{1}{2}\Phi^{l}.
\end{gather}
Let us solve it for $X>0$.
It follows from \eqref{Phil1} and Lemma \ref{nrhopm} that $\sqrt{V}$ is Lipschitz continuous on $[\ve,\Phi^{l}-\ve]$ for any suitably small $\ve>0$.
Therefore, the problem \eqref{neq1'} has a solution $\Phi$ which is strictly decreasing unless $\Phi$ attains zero.
Suppose that $\Phi$ attains zero at a point $X=X_{*}<+\infty$. We observe that 
\begin{align*}
X_{*}&=\int_{0}^{X_{*}} 1dX
=\int_{0}^{X_{*}} \frac{-\partial_{x} \Phi (X)}{\sqrt{2V(\Phi(X))} } dX 
= \int_{0}^{\Phi^{l}/2} \frac{d\Phi}{\sqrt{2V(\Phi)}}=+\infty,
\end{align*}
where we have used \eqref{Phil2} in deriving the last equality. 
This is a contradiction.
Thus $\Phi$ cannot attain zero at some finite point.
Now it is clear that the problem \eqref{neq1'} is solvable uniquely for any $X>0$ and $\lim_{X \to +\infty} \Phi(X)=0$.
Similarly, we can solve uniquely the problem \eqref{neq1'} for $X<0$ and also see $\lim_{X \to -\infty} \Phi(X)=\Phi^{l}$.
Consequently, the problem \eqref{nphieq2} with $\Phi(0)=\Phi^{l}/2$ has a unique solution $\Phi \in C^{2}(\R)$.
Furthermore, it is easy to show by differentiating the first equation in \eqref{nphieq2} and using $\D_{X} \Phi<0$ that $\Phi$ satisfies \eqref{nphieq1}.

Now by using $\Phi$, we define $F_{+}$  and  $F_{-}$ as \eqref{nfform+} and \eqref{nfform-}, respectively, i.e. 
\begin{align*}
F_{+}(X,\xi)
&=F_{+}^l(-\sqrt{(\xi_1-\alpha)^2+2q_{+}(\Phi^{l}-\Phi(X))}+\alpha,\xi')\chi(-(\xi_1-\alpha))
\\
&\quad +F_{+}^l(\sqrt{(\xi_1-\alpha)^2+2q_{+}(\Phi^{l}-\Phi(X))}+\alpha,\xi')\chi(\xi_1-\alpha),
\\
F_{-}(X,\xi)
&=F_{-}^r(-\sqrt{(\xi_1-\alpha)^2+2q_{-}\Phi(X)}+\alpha,\xi')\chi(-(\xi_1-\alpha))
\\
&\quad +F_{-}^r(\sqrt{(\xi_1-\alpha)^2+2q_{-}\Phi(X)}+\alpha,\xi')\chi(\xi_1-\alpha).
\end{align*}
One can see that $F_{\pm}$ satisfy the conditions (i)--(iii) in Definition \ref{nDefS1} 
in much the same way as in the proof of Theorem \ref{existence1}. 
The condition (iv) is also validated by \eqref{nfform+}, \eqref{nfform-}, and \eqref{nrho+'}--\eqref{nphieq1}.
The proof is complete.
\end{proof}

\subsection{An example with all the conditions in Theorem \ref{nexistence1}}\label{S3.3}

This subsection provides an example satisfying all the conditions in Theorem \ref{nexistence1}. 
For simplicity, we treat the case $\alpha=0$. 
We choose $F_{\pm}^{l}$ and $F_{\pm}^{r}$ for any $\Phi^{l}>0$ as
\begin{align*}
F^l_+(\xi)&:=\Biggl\{
\begin{array}{ll}
\frac{1}{2e_+\sqrt{q_+}} & \hbox{if} \ \
\xi\in\Bigl(\bigl[-\frac{3}{2}\sqrt{q_+\Phi^l},-\sqrt{q_+\Phi^l}\bigr]
\cup\bigl[\sqrt{q_+\Phi^l},\frac{3}{2}\sqrt{q_+\Phi^l}\bigr]\Bigr)
\times[0,1]^{n-1}, \\
0 & \hbox{otherwise}, 
\end{array}
\\
F^r_+(\xi)&:=\Biggl\{
\begin{array}{ll}
\frac{1}{2e_+\sqrt{q_+}} & \hbox{if} \ \
\xi\in\bigl[-\frac{1}{2}\sqrt{q_+\Phi^l},\frac{1}{2}\sqrt{q_+\Phi^l}\bigr]
\times[0,1]^{n-1}, \\
0 & \hbox{otherwise}, 
\end{array}
\\
F^l_-(\xi)&:=\Biggl\{
\begin{array}{ll}
\frac{1}{2e_-\sqrt{q_-}} & \hbox{if} \ \
\xi\in\bigl[-\frac{1}{2}\sqrt{q_-\Phi^l},\frac{1}{2}\sqrt{q_-\Phi^l}\bigr]
\times[0,1]^{n-1}, \\
0 & \hbox{otherwise}. 
\end{array}
\\
F^r_-(\xi)&:=\Biggl\{
\begin{array}{ll}
\frac{1}{2e_-\sqrt{q_-}} & \hbox{if} \ \
\xi\in\Bigl(\bigl[-\frac{3}{2}\sqrt{q_-\Phi^l},-\sqrt{q_-\Phi^l}\bigr]
\cup\bigl[\sqrt{q_-\Phi^l},\frac{3}{2}\sqrt{q_-\Phi^l}\bigr]\Bigr)
\times[0,1]^{n-1}, \\
0 & \hbox{otherwise}, 
\end{array}
\end{align*}
It is clear that
\begin{gather*}
F^l_\pm,F^r_\pm\in L^1(\R^n)\cap L^3_{loc}(\R;L^1(\R^{n-1})), \quad 
|\xi_1|F^l_\pm,|\xi_1|F^r_\pm\in L^1(\R^n), \quad
F^l_\pm,F^r_\pm\geq 0,
\\
e_{\pm}\int_{\mathbb R^{n}}F^l_{\pm}(\xi)\,d\xi
=e_{\pm}\int_{\mathbb R^{n}}F^r_{\pm}(\xi)\,d\xi
=\frac{\sqrt{\Phi^l}}{2}.
\end{gather*}
It is easy to see that \eqref{nalpha0} holds for $\alpha=0$, and also \eqref{Flr0} holds.
It remains to show \eqref{Phil}. 

To this end, we compute as follows:
\begin{align*}
\rho(\Phi)&:=e_+\rho_+(\Phi)-e_-\rho_-(\Phi)
\\
&=\Biggl\{
\begin{array}{ll}
\sqrt{2\Phi+{\textstyle\frac{1}{4}}\Phi^l}
-\sqrt{-2\Phi+{\textstyle\frac{9}{4}}\Phi^l}
+\sqrt{-2\Phi+\Phi^l} 
& \hbox{if} \quad
\Phi\in[0,\frac{1}{2}\Phi^l], \\
\sqrt{2\Phi+{\textstyle\frac{1}{4}}\Phi^l}
-\sqrt{2\Phi-\Phi^l}
-\sqrt{-2\Phi+{\textstyle\frac{9}{4}}\Phi^l} 
& \hbox{if} \quad
\Phi\in[\frac{1}{2}\Phi^l,\Phi^l],
\end{array}
\end{align*}
where $\rho_{+}(\Phi)$ and $\rho_{-}(\Phi)$ are defined in \eqref{nrho++} and \eqref{nrho--}, respectively.
We note that $\rho(\Phi)=-\rho(\Phi^l-\Phi)$ holds; $\rho=\rho(\Phi)$ is continuous on $[0,\Phi^l]$ and smooth except the point
$\Phi=\frac{1}{2}\Phi^l$; the Sagdeev potential $V(\Phi)=\int_{0}^{\Phi} \rho(\varphi) \,d\varphi$ defined in \eqref{nV0} satisfies $V(\Phi)=V(\Phi^l-\Phi)$.

Let us show \eqref{Phil1}. Using $V(0)=0$ and $V(\Phi)=V(\Phi^l-\Phi)$, we obtain $V(0)=V(\Phi^l)=0$. 
Moreover, it is easy to see that $\frac{d}{d\Phi}V(\Phi)>0$ on $(0,\frac{1}{2}\Phi^l)$.
From this and $V(\Phi)=V(\Phi^l-\Phi)$, we deduce \eqref{Phil1}.
To show \eqref{Phil2}, it suffices to show the first condition in \eqref{Phil2} owing to $V(\Phi)=V(\Phi^l-\Phi)$.
We know that $\rho\in C([0,\frac{1}{2}\Phi^l]) \cap C^\infty([0,\frac{1}{2}\Phi^l))$ which means
$V\in C^1([0,\frac{1}{2}\Phi^l]) \cap C^\infty([0,\frac{1}{2}\Phi^l))$.
On the other hand, $V(0)=\frac{d}{d\Phi}V(0)=\rho(0)=0$ holds. 
Therefore, the application of the Taylor theorem ensures the first condition in \eqref{Phil2}.
Consequently, all the conditions in Theorem \ref{nexistence1} are valid.

\newcommand{\Ruby}[2]{\stackrel{\scriptsize\textcircled{\tiny #2}}{#1}}

\section{Wave Trains}\label{S4}

The main purpose of this section is to investigate {\it waves trains} which
are special solutions of the Vlasov--Poisson system \eqref{eq1}--\eqref{eq2}
so that the following hold:
\begin{enumerate}[(i)]
\item $f_{\pm}(t,x,\xi)=F_{\pm}(x-\alpha t, \xi), \quad \phi(t,x)=\Phi(x-\alpha t) \quad \text{for some $\alpha \in \mathbb R$}$;
\item $(F_{\pm}(X+\gamma,\xi),\Phi(X+\gamma))=(F_{\pm}(X,\xi),\Phi(X)) \quad \text{for some $\gamma > 0$ and any $X \in \mathbb R$}$;
\item $\Phi$ takes an extremum at $X=0$ and a reference point of $\Phi$ is located at $X=0$, i.e. $\D_{X}\Phi(0)=\Phi(0)=0$;
\item $\Phi$ is positive and has a unique local maximum on $(0,\gamma)$.
\end{enumerate}
Notice that the condition (iii) can be assumed to hold without loss of generality.
From the conditions (iii) and (iv), it is clear that $\Phi$ takes a local minimum at $x=0$.
It is also worth pointing out that for the case that the potential $\Phi$ is negative and has a unique local minimum on $(0,\gamma)$, 
we can reduce it to the case that the condition (iv) holds by replacing 
$(F_{\pm},\Phi,e_{\pm},q_{\pm})$ by $(F_{\mp},-\Phi,e_{\mp},q_{\mp})$ in \eqref{tVP2}.

It is seen from direct computations that $(F_{\pm},\Phi)=(F_{\pm}(X,\xi),\Phi(X))$ solves the following problem:
\begin{subequations}\label{tVP2}
\begin{gather}
(\xi_{1}-\alpha) \partial_{X} F_{\pm} \pm q_{\pm}\partial_{X} \Phi  \partial_{\xi_{1}} F_{\pm} =0, \ \ X \in \mathbb R, \ \xi \in \mathbb R^{n},
\label{teq3}
\\
\partial_{XX} \Phi 
= e_{+}\int_{\mathbb R^{n}} F_{+} d\xi - e_{-}\int_{\mathbb R^{n}} F_{-}  d\xi , \ \ X \in \mathbb R,
\label{teq4}\\
(F_{\pm}(X+\gamma,\xi),\Phi(X+\gamma))=(F_{\pm}(X,\xi),\Phi(X)).
\label{tbc1}
\end{gather}
\end{subequations}

Let us give a definition of solutions of \eqref{tVP2}, where $\mathbb T_{\gamma}:= \mathbb R / \gamma\mathbb Z$.

\begin{defn}\label{tDefS1}
We say that $(F_{\pm},\Phi)$ is a solution of the problem \eqref{tVP2} if it satisfies the following:
\begin{enumerate}[(i)]
\item $F_{\pm} \in L^{1}_{loc}({\mathbb T}_{\gamma}\times \mathbb R^{n}) \cap C({\mathbb T}_{\gamma};L^{1}(\mathbb R^{n}))$
and $\Phi \in C^{2}({\mathbb T}_{\gamma})$;
\item $F_{\pm}(X,\xi)\geq 0$, $\Phi(0)=\partial_{X}\Phi(0)= 0$, and $\Phi(X)>0$ has a unique local maximum on $(0,\gamma)$;
\item $F_{\pm}$ solve
\begin{gather}
(F_{\pm},(\xi_1-\alpha)\D_X\psi \pm q_{\pm} \D_X{\Phi}\D_{\xi_1}\psi)_{L^2({\mathbb T}_{\gamma} \times {\mathbb R}^n)}
=0 \quad \hbox{for $\forall \psi \in C_0^1({\mathbb T}_{\gamma} \times \mathbb{R}^n)$}\text{\it ;}
\label{tweak1}
\end{gather}
%
\item $\Phi$ solves \eqref{teq4} in the classical sense.
\end{enumerate}
\end{defn}

The condition (ii) does not allow the variant of the solution caused by translation.
The equation \eqref{tweak1} is a standard weak form of the equations \eqref{teq3}.
It is possible to replace {\it the classical sense} in the condition (iv) by {\it the weak sense}.
Indeed, a weak solution $\Phi$ of the problem of \eqref{teq4} 
is a classical solution due to $F_{\pm} \in C({\mathbb T}_{\gamma};L^{1}(\mathbb R^{n}))$.
It follows from the conditions (i) and (iv) that any solution $(F_{\pm},\Phi)$ satisfies the neutral condition
\begin{gather}\label{tnetrual2}
\int_{0}^{\gamma} \left( \int_{\mathbb R^{n}} e_{+}F_{+} - e_{-}F_{-} d\xi \right) dX= 0.
\end{gather}

Our approach is the same as in Section \ref{S2}, 
but we have the freedom of the distribution even for the non-trapped ions and electrons.
We reduce the problem \eqref{tVP2} to the following problem:
\begin{gather}\lb{tphieq2}
(\D_{X} \Phi)^2=2V(\Phi;\beta,G,H_{\pm}).
\end{gather}
The Sagdeev potential $V$ is defined 
for the maximum $\beta>0$ of the potential $\Phi$, 
the distribution $G=G(\xi)\geq 0$ on the plane $\{x=\frac{\gamma}{2}\}$ of the trapped ions,
and the distribution $H_{\pm}=H_{\pm}(\xi)\geq 0$ on the plane $\{x=0\}$ of the non-trapped ions and electrons by
\begin{gather}
V=V(\Phi;\beta,G,H_{\pm}):=\int_{0}^{\Phi} \left( e_{+}\rho_{+}(\varphi;\beta,G,H_{+}) - e_{-}\rho_{-}(\varphi;H_{-} )\right)  d\vphi, \quad \Phi \in [0,\beta],
\label{tV0}
\end{gather}
where the expected densities $\rho_{\pm}$ are defined by
\begin{align}
\rho_{+}(\Phi;\beta,G,H_{+})
&:=\int_{{\mathbb R}^n} H_{+}(\xi)\frac{|\xi_1-\alpha|}{\sqrt{(\xi_1-\alpha)^2+2q_{+}\Phi}}\,d\xi
\notag\\
&\quad +2\int_{\mathbb R^{n-1}}\int_{\sqrt{2q_{+}\beta-2q_{+}\Phi}+\alpha}^{\sqrt{2q_{+}\beta}+\alpha}
G(\xi)\frac{\xi_1-\alpha}{\sqrt{(\xi_1-\alpha)^2+2q_{+}\Phi-2q_{+}\beta}}\,d\xi_1\,d\xi',
\label{trho++}
\\
\rho_{-}(\Phi;H_{-})
&:=\int_{{\mathbb R}^n} H_{-}(\xi)\frac{|\xi_1-\alpha|}{\sqrt{(\xi_1-\alpha)^2-2q_{-}\Phi}} \chi((\xi_1-\alpha)^2-2q_{-}{\Phi}) \,d\xi,
\label{trho--}
\end{align}
where we ignore simply the integral with respect to $\xi'$ in the last term in \eqref{trho++} for the case $n=1$.
We will see in \eqref{trho+'} and  \eqref{trho-'} below that $\rho_{\pm}$ really express the densities of the ion and electron.
The functions $\rho_{\pm}$ are well-defined for $H_{\pm} \in L^{1}_{loc}(\mathbb R^{n})$ with $H_{\pm} \geq 0$
and $G \in L^{1}_{loc}((\alpha,+\infty) \times \mathbb R^{n-1})$ with $G \geq 0$, since all the integrants are nonnegative. 
The properties of $\rho_{\pm}$ are summarized in the following lemma.
We omit the proof, since it is the same as that of Lemma \ref{rhopm}.

\begin{lem}\label{trhopm}
Let $\alpha \in \mathbb R$, $\beta>0$, $H_{+} \in L^1(\R^n)$,
$H_{-} \in L^1(\R^n) \cap L_{loc}^{p}({\mathbb R};L^{1}(\mathbb R^{n-1}))$, and
$G \in L^1_{loc}((\alpha,+\infty) \times \R^{n-1}) \cap L_{loc}^{p}([\alpha,+\infty);L^{1}(\mathbb R^{n-1}))$ for some $p>2$.
Then the following estimates hold for $\Phi \in [0,\beta]$:
\begin{align}
& |\rho_{+}(\Phi;\beta,G,H_{+})|  \leq \|H_{+}\|_{L^{1}(\mathbb R^{n})} + C\|G\|_{L^{p}(\alpha,\sqrt{2q_{+}\beta}+\alpha;L^{1}(\mathbb R^{n-1}))},
\label{trho+}\\
& |\rho_{-}(\Phi;H_{-})|  \leq  \sqrt{2} \|H_{-}\|_{L^{1}(\mathbb R^{n})} +C\|H_{-}\|_{L^{p}(-2\sqrt{q_{-}\beta}+\alpha,2\sqrt{q_{-}\beta}+\alpha;L^{1}(\mathbb R^{n-1}))},
\label{trho-}
\end{align}
where $C$ is a positive constant depending only on $\alpha$, $\beta$, $p$, and $q_{\pm}$.
Furthermore, the functions $\rho_{+}(\cdot;\beta,G,H_{+})$ and $\rho_{-}(\cdot;H_{-})$ belong to $C([0,\beta])$.
\end{lem}

We are now in a position to state our main theorem.

\begin{thm} \label{texistence1}
Let $p>2$ and $\alpha \in \mathbb R$.
Suppose that $\beta>0$, $H_{+} \in L^1(\R^n)$,
$H_{-} \in L^1(\R^n) \cap L_{loc}^{p}({\mathbb R};L^{1}(\mathbb R^{n-1}))$,
$G \in L^1_{loc}((\alpha,+\infty) \times \R^{n-1}) \cap L_{loc}^{p}([\alpha,+\infty);L^{1}(\mathbb R^{n-1}))$,
and $H_{\pm},  G\geq 0$.
Then the problem \eqref{tVP2} has a solution $(F_{\pm},\Phi)$ for some period $\gamma>0$
if and only if there exists a quadruplet $(\beta,G,H_{\pm})$ with
\begin{subequations}\label{tG-beta}
\begin{gather}
H_{-}(\xi_1+\alpha,\xi')=H_{-}(-\xi_1+\alpha,\xi'), \quad 
(\xi_{1},\xi')\in(0,\sqrt{2q_{-}\beta})\times\R^{n-1},
\label{tG-beta1}\\
V(\Phi;\beta,G,H_{\pm})>0 \ \ \text{for $\Phi \in (0,\beta)$}, \quad 
V(\beta;\beta,G,H_{\pm})=0,
\label{tG-beta2}\\
\int_{0}^{\beta} \frac{d\Phi}{\sqrt{V(\Phi;\beta,G,H_{\pm})}} <+ \infty.
\label{tG-beta3}
\end{gather}
\end{subequations}
The solution satisfies \eqref{tnetrual2}, 
\begin{gather}
\gamma = 2\int_{0}^{\beta} \frac{d\Phi}{\sqrt{V(\Phi;\beta,G,H_{\pm})}},
\label{tperiod1} \\
\max_{X \in [0,\gamma]}\Phi(X) = \Phi\left(\frac{\gamma}{2}\right)= \beta, 
\label{symm1}\\
\Phi\left(\frac{\gamma}{2}+X\right)=\Phi\left(\frac{\gamma}{2}-X\right), \quad \Phi(X)=\Phi(-X), \quad X \in \mathbb R.
\label{symm2}
\end{gather}
Furthermore, $F_{\pm}$ can be written on $[0,\gamma]\times \mathbb R^{n}$ by 
\begin{gather}\label{tfform+}
\begin{aligned}
&{F}_{+}(X,\xi)
\\
& =H_{+}(-\sqrt{(\xi_1-\alpha)^2-2q_{+}{\Phi}(X)}+\alpha,\xi')
\chi((\xi_{1}-\alpha)^{2}-2q_{+}\Phi(X))\chi(-(\xi_{1}-\alpha)) \\
& \qquad
+G(\sqrt{(\xi_1-\alpha)^2-2q_{+}\Phi(X)+2q_{+}\beta}+\alpha,\xi')
\chi(-(\xi_1-\alpha)^2+2q_{+}\Phi(X)) \\
& \qquad
+H_{+}((\sqrt{(\xi_1-\alpha)^2-2q_{+}\Phi(X)}+\alpha,\xi')
\chi((\xi_{1}-\alpha)^{2}-2q_{+}\Phi(X))\chi(\xi_{1}-\alpha)
\end{aligned}
\end{gather}
and
\begin{gather}\label{tfform-}
\begin{aligned}
F_{-}(X,\xi)
&=H_{-}(-\sqrt{(\xi_1-\alpha)^2+2q_{-}\Phi(X)}+\alpha,\xi')\chi(-(\xi_1-\alpha))
\\
&\quad +H_{-}(\sqrt{(\xi_1-\alpha)^2+2q_{-}\Phi(X)}+\alpha,\xi')\chi(\xi_1-\alpha). 
\end{aligned}
\end{gather}
\end{thm}

\begin{rem} {\rm
We do not use the information of solution to write the necessary and sufficient condition \eqref{tG-beta}.
Indeed, it depends only on $e_{\pm}$, $q_{\pm}$, $\alpha$, $\beta$, $G$, and $H_{\pm}$. 
Furthermore, the necessary and sufficient condition \eqref{tG-beta} is easy to check by computers as follows. 
The condition \eqref{tG-beta1} requires just symmetricity with respect to the plan $\{\xi_{1}=\alpha\}$.
If $V(\Phi;\beta,G,H_{\pm})$ is a $C^{2}$-function around $\Phi=0,\beta$,  
the condition \eqref{tG-beta3} is equivalent to that $\frac{d}{d\Phi}V(0;\beta,G,H_{\pm})>0$ and $\frac{d}{d\Phi}V(\beta;\beta,G,H_{\pm})<0$ thanks to the Taylor theorem. 
Therefore, \eqref{tG-beta2} and \eqref{tG-beta3} are verified if the graph of $V(\Phi;\beta,G,H_{\pm})$ is drawn as in Figure \ref{tfigV}.
}
\end{rem}

\begin{figure}[H]
\begin{center}
    \includegraphics[width=8cm, bb=0 0 1720 1294]{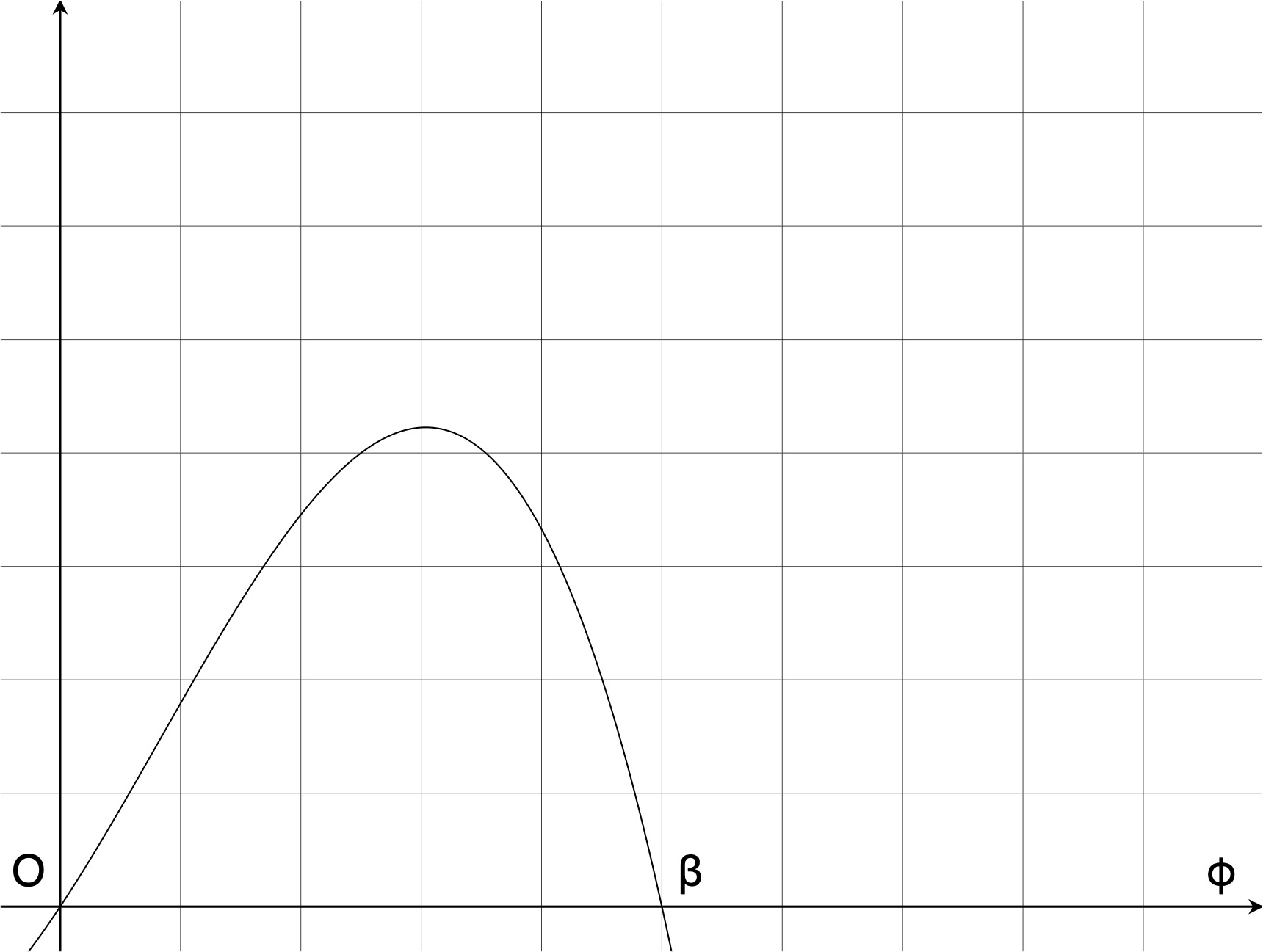}
\end{center}
  \caption{the graph of $V$}  
  \label{tfigV}
\end{figure}

The uniqueness always fails, because the degree of freedom is much higher compared to solitary waves.
Specifically, the following theorems hold.

\begin{thm}\label{tnonunique1}
For any $\alpha \in \mathbb R$ and $\gamma>0$,
the problem \eqref{tVP2} with the period $\gamma$ admits infinite many solutions $(F_{\pm},\Phi)$ with no trapped ions.
\end{thm}

\begin{thm}\label{tnonunique2}
For any $\alpha \in \mathbb R$ and $\gamma>0$,
the problem \eqref{tVP2} with the period $\gamma$ admits infinite many solutions $(F_{\pm},\Phi)$ with trapped ions. 
\end{thm}

As mentioned in Section \ref{S2}, it is often assumed in plasma physics that the electron density obeys the Boltzmann relation $\rho_{-}:=\int_{\mathbb R^{n}} f_{-} d\xi=\rho e^{-\kappa\phi}$ in \eqref{eq2}.
The corresponding train waves solve
\begin{subequations}\label{tVP3}
\begin{gather}
(\xi_{1}-\alpha) \partial_{X} F_{+} + q_{+}\partial_{X} \Phi  \partial_{\xi_{1}} F_{+} =0, \ \ X \in \mathbb R, \ \xi \in \mathbb R^{n},
\label{teq5}
\\
\partial_{XX} \Phi 
= e_{+}\int_{\mathbb R^{n}} F_{+} d\xi - e_{-}\rho e^{-\kappa \Phi}, \ \ X \in \mathbb R,
\label{teq6}\\
(F_{\pm}(X+\gamma,\xi),\Phi(X+\gamma))=(F_{\pm}(X,\xi),\Phi(X)).
\label{tbc2}
\end{gather}
\end{subequations}
Theorem \ref{texistence1} is also applicable to the problem \eqref{tVP3} by suitably choosing $H_{-}$.
Furthermore, we have infinite many solutions of the problem \eqref{tVP3} with the same period $\gamma$.
Namely, the following corollary holds. 

\begin{cor}\label{tnonunique3}
Suppose that $(\beta,G,H_{\pm})$ is a quadruplet with \eqref{tG-beta} and 
\begin{gather*}
H_{-}(\xi)=\frac{\rho}{\displaystyle \int_{\mathbb R^{n}} e^{\frac{-\kappa}{2q_{-}}|\xi|^2} d\xi}
e^{\frac{-\kappa}{2q_{-}}\{(\xi_{1}-\alpha)^{2}+|\xi'|^{2}\}}.
\end{gather*}
Let $(F_{\pm},\Phi)$ be a solution of the problem \eqref{tVP2}, 
which is obtained by the application of Theorem \ref{texistence1} with the quadruplet.
Then $(F_{+},\Phi)$ solves the problem \eqref{tVP3}.
There is a constant $\gamma_{\star}>0$ satisfying the following properties: 
for any $\alpha$ and $\gamma \in (0,\gamma_{\star}]$, 
there exist infinite many solutions of the problem \eqref{tVP3} with the period $\gamma$. 
\end{cor}

This section is organized as follows. 
In subsection \ref{S4.1}, we find necessary conditions for the solvability of the problem \eqref{tVP2}.
In subsection \ref{S4.2}, we show the solvability stated in Theorem \ref{texistence1}.
Subsection \ref{S4.3} gives the proofs of Theorems \ref{tnonunique1} and \ref{tnonunique2}.
We prove Corollary \ref{tnonunique3} in subsection \ref{S4.4}.

\subsection{Necessary conditions}\label{S4.1}

In this section, we investigate necessary conditions for the solvability of the problem \eqref{tVP2}.

\begin{lem}\label{tneed1}
Let $\alpha \in \mathbb R$.
Suppose that  the problem \eqref{tVP2} has a solution $(F_{\pm},\Phi)$ for some period $\gamma>0$.
Then the condition \eqref{tnetrual2} holds;
$H_{-}(\xi)=F_{-}(0,\xi)$ satisfies the condition \eqref{tG-beta1} with $\beta=\Phi_{max}:=\max_{X \in [0,\gamma]} \Phi(X)$; 
$F_{\pm}$ are written on $[0,\gamma]\times \mathbb R^{n}$ as \eqref{tfform+} and \eqref{tfform-} with $\beta=\Phi_{max}$, 
$G(\xi)=F_{+}(X_{*},\xi)$, and $H_{\pm}(\xi)=F_{\pm}(0,\xi)$, where \footnote{$X_{*}$ is uniquely determined owing to the condition (ii) in Definition \ref{tDefS1}.}$\Phi(X_{*})=\Phi_{\max}$;
$\rho_{\pm} \in C([0,\Phi_{max}])$;
$\Phi$ solves \eqref{tphieq2} with $\beta=\Phi_{max}$, 
$G(\xi)=F_{+}(X_{*},\xi)$, and $H_{\pm}(\xi)=F_{\pm}(0,\xi)$;
$V$ satisfies \eqref{tG-beta2} and \eqref{tG-beta3} with $\beta=\Phi_{max}$, $G(\xi)=F_{+}(X_{*},\xi)$, and $H_{\pm}(\xi)=F_{\pm}(0,\xi)$; the conditions \eqref{tperiod1}--\eqref{symm2} hold.
\end{lem}

\begin{proof}
First \eqref{tnetrual2} follows from integrating \eqref{teq4} and using the periodicity.

Next we show that $F_{-}$ satisfies the condition \eqref{tG-beta1} with $\beta=\Phi_{max}$,
and $F_{\pm}$ are written as \eqref{tfform+} and \eqref{tfform-} with $\beta=\Phi_{max}$, 
$G(\xi)=F_{+}(X_{*},\xi)$, and $H_{\pm}(\xi)=F_{\pm}(0,\xi)$.
Regarding $\Phi$ as a given function and then applying the characteristics method to \eqref{tweak1}, 
we see that the values of $F_{+}$ must be the same on the following characteristics curve:
\begin{gather*}
\frac{1}{2}(\xi_{1}-\alpha)^{2}-q_{+}\Phi(X)=c,
\end{gather*}
where $c$ is some constant. We draw the illustration of characteristics for $\alpha=0$ in Figure \ref{tfig+}.
\begin{figure}[H]
\begin{center}
    \includegraphics[width=9.5cm, bb=0 0 1720 1122]{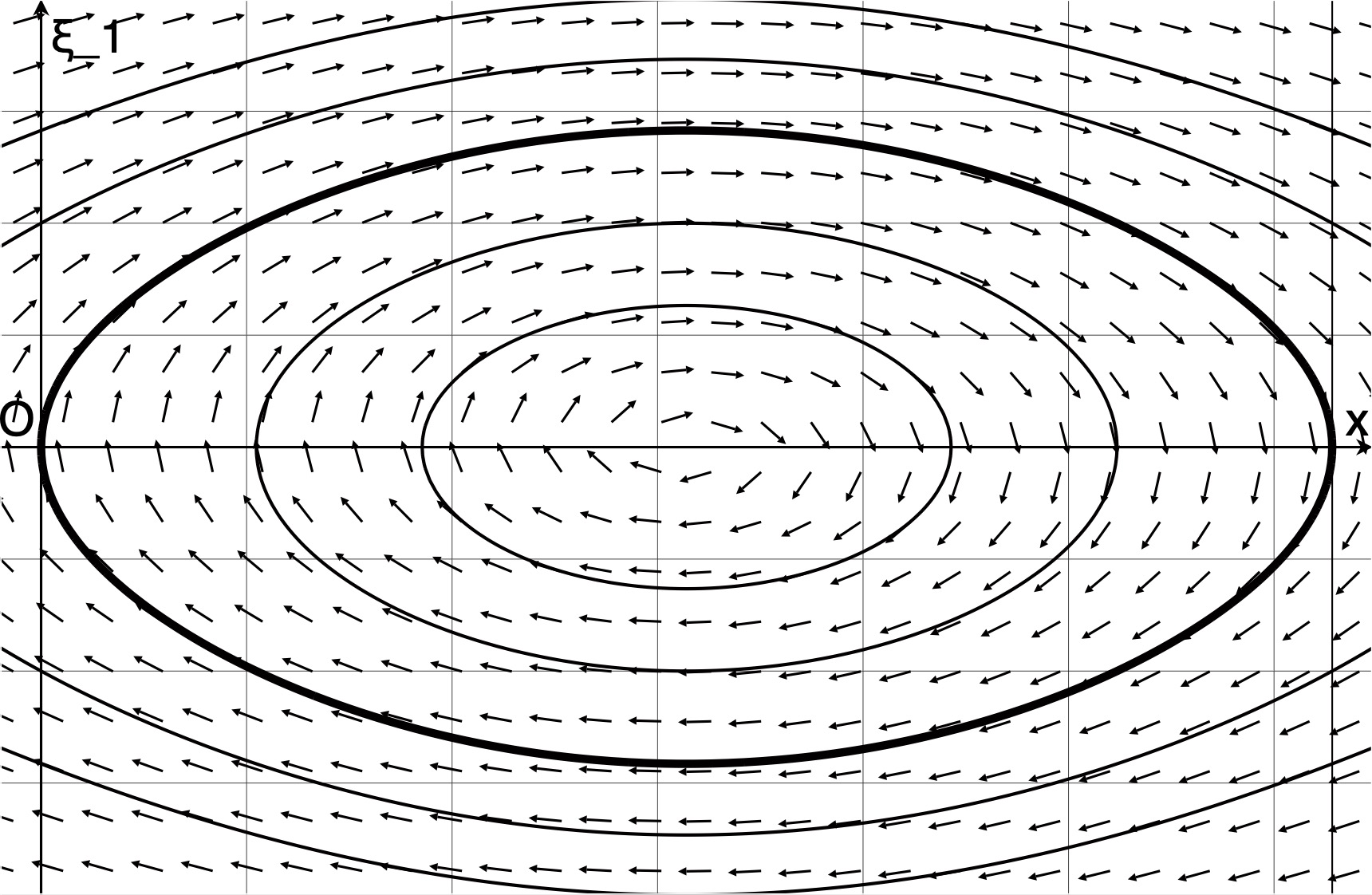}
\end{center}
  \caption{the characteristics of $F_{+}$}  
  \label{tfig+}
\end{figure}
It tells us that
\begin{align*}
F_{+}(y,-\sqrt{\eta_{1}^{2}+2q_{+}\Phi(y)}+\alpha,\eta')=F_{+}(0,\eta_{1}+\alpha,\eta'), &\quad (y,\eta_{1},\eta') \in {\cal X}_{+}^{1},
 \\
F_{+}(y,\pm\sqrt{{\eta_1^2-2q_{+}\Phi_{max}+2q_{+}\Phi}(y)}+\alpha,\eta')=F_{+}(X_{*},\eta_{1}+\alpha,\eta'), &\quad (y,\eta_{1},\eta') \in {\cal X}_{+}^{2},
 \\
F_{+}(y,\sqrt{\eta_{1}^{2}+2q_{+}\Phi(y)}+\alpha,\eta')=F_{+}(0,\eta_{1}+\alpha,\eta'), &\quad  (y,\eta_{1},\eta') \in {\cal X}_{+}^{3},
\end{align*}
where $\eta=(\eta_{1},\eta_{2},\eta_{3})=(\eta_{1},\eta')$ and
\begin{align*}
{\cal X}_{+}^{1}&:=[0,\gamma]\times\R_-^n, 
\\
{\cal X}_{+}^{2}&:=\{(y,\eta_1,\eta')\in [0,\gamma]\times\R^n\;|\; 2q_{+}\Phi_{max}-2q_{+}\Phi(y) <\eta_1^2<2q_{+}\Phi_{max}\}, 
\\
{\cal X}_{+}^{3}&:=[0,\gamma]\times\R_+^n. 
\end{align*}
Furthermore, we conclude from these three equalities that $F_{+}$ must be written as \eqref{tfform+} with $\beta=\Phi_{max}$, $G(\xi)=F_{+}(X_{*},\xi)$, and $H_{+}(\xi)=F_{+}(0,\xi)$, i.e.
\begin{align*}
&{F}_{+}(X,\xi)
\\
& =F_{+}(0,-\sqrt{(\xi_1-\alpha)^2-2q_{+}{\Phi}(X)}+\alpha,\xi')
\chi((\xi_{1}-\alpha)^{2}-2q_{+}\Phi(X))\chi(-(\xi_{1}-\alpha)) \\
& \qquad
+F_{+}(X_{*},\sqrt{(\xi_1-\alpha)^2-2q_{+}\Phi(X)+2q_{+}\Phi_{max}}+\alpha,\xi')
\chi(-(\xi_1-\alpha)^2+2q_{+}\Phi(X)) \\
& \qquad
+F_{+}(0,\sqrt{(\xi_1-\alpha)^2-2q_{+}\Phi(X)}+\alpha,\xi')
\chi((\xi_{1}-\alpha)^{2}-2q_{+}\Phi(X))\chi(\xi_{1}-\alpha).
\end{align*}

Next we treat $F_{-}$. The values of $F_{-}$ must be the same on the following characteristics curve:
\begin{gather*}
\frac{1}{2}(\xi_{1}-\alpha)^{2}+q_{-}\Phi(X)=c,
\end{gather*}
where $c$ is some constant. We draw the illustration of characteristics for $\alpha=0$ in Figure \ref{tfig-}.
\begin{figure}[H]
\begin{center}
    \includegraphics[width=9.5cm, bb=0 0 1987 1239]{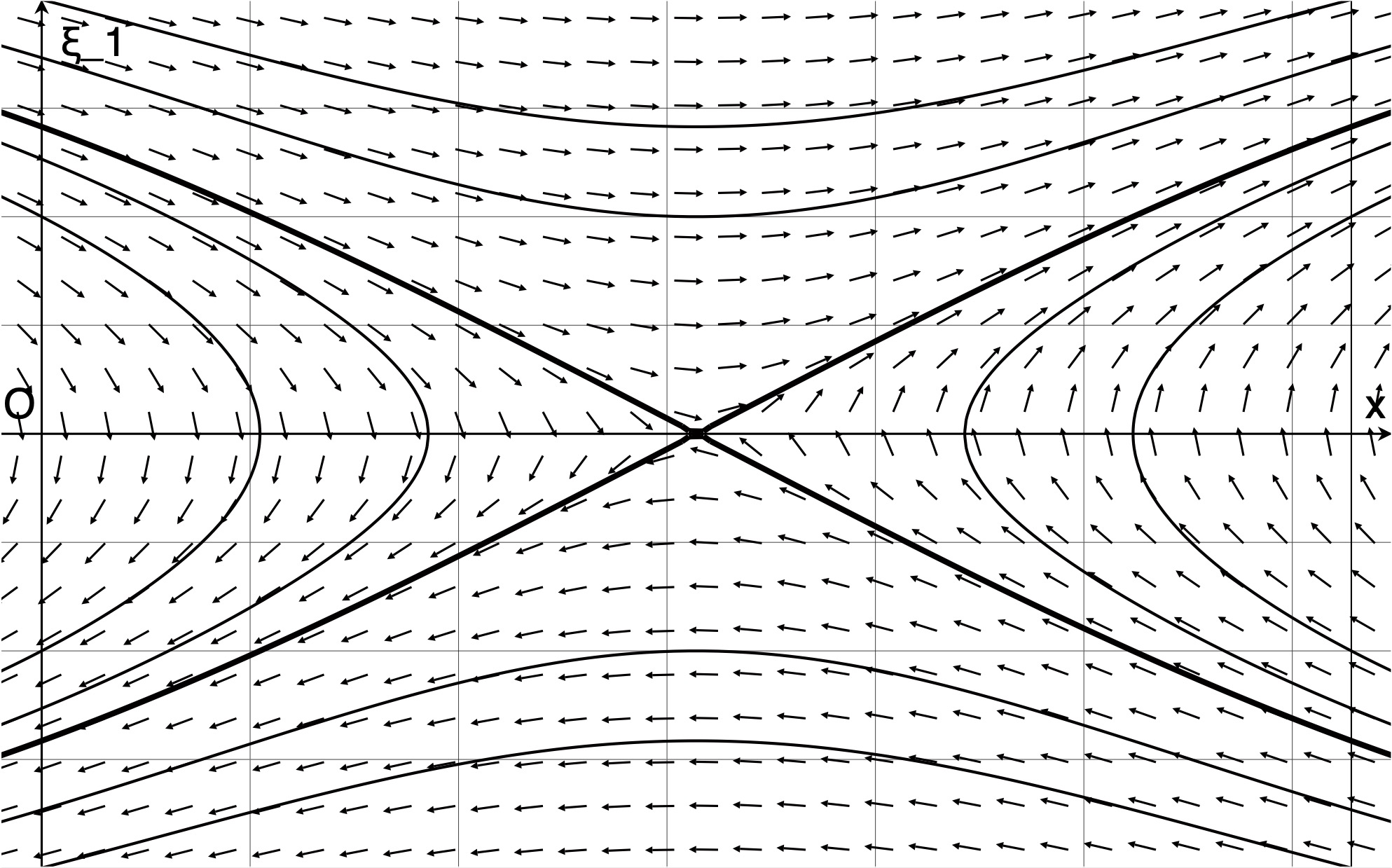}
\end{center}
  \caption{the characteristics of $F_{-}$}  
  \label{tfig-}
\end{figure}
It tells us that
\begin{align*}
F_{-}(y,-\sqrt{\eta_{1}^{2}-2q_{-}\Phi(y)}+\alpha,\eta')=F_{-}(0,\eta_{1}+\alpha,\eta'), &\quad (y,\eta_{1},\eta') \in {\cal X}_{-}^{1},
 \\
F_{-}(y,\pm\sqrt{\eta_1^2-2q_{-}{\Phi}(y)}+\alpha,\eta')=F_{-}(0,\eta_{1}+\alpha,\eta'), &\quad (y,\eta_{1},\eta') \in {\cal X}_{-}^{2},
 \\
F_{-}(y,\sqrt{\eta_1^2-2q_{-}{\Phi}(y)}+\alpha,\eta')=F_{-}(0,\eta_{1}+\alpha,\eta'), & \quad  (y,\eta_{1},\eta') \in {\cal X}_{-}^{3},
\end{align*}
where 
\begin{align*}
{\cal X}_{-}^{1}&:=[0,\gamma]\times(-\infty,-\sqrt{2q_{-}\Phi_{max}})\times\R^{n-1},
\\
{\cal X}_{-}^{2}&:=\{(y,\eta_1,\eta')\in[0,\gamma]\times\R\times\R^{n-1}\;|\;2q_{-}{\Phi}(y)<\eta_1^2<2q_{-}\Phi_{max}\}, \quad
\\
{\cal X}_{-}^{3}&:=[0,\gamma]\times(\sqrt{2q_{-}\Phi_{max}},+\infty)\times\R^{n-1}. 
\end{align*}
Due to the second equality, the following must hold:
\begin{align*}
F_{-}(0,\xi_1+\alpha,\xi')=F_{-}(0,-\xi_1+\alpha,\xi'), \quad 
(\xi_{1},\xi')\in(-\sqrt{2q_{-}\Phi_{max}},\sqrt{2q_{-}\Phi_{max}})\times\R^{n-1},
\end{align*}
which is the same equality as \eqref{tG-beta1} with $\beta=\Phi_{max}$ and $H_{-}(\xi)=F_{-}(0,\xi)$.
Furthermore, we conclude from these three equalities that $f$ must be written as \eqref{tfform-} with $H_{-}(\xi)=F_{-}(0,\xi)$, i.e.
\begin{align*}
F_{-}(X,\xi)
&=F_{-}(0,-\sqrt{(\xi_1-\alpha)^2+2q_{-}\Phi(X)}+\alpha,\xi')\chi(-(\xi_1-\alpha))
\\
&\quad +F_{-}(0,\sqrt{(\xi_1-\alpha)^2+2q_{-}\Phi(X)}+\alpha,\xi')\chi(\xi_1-\alpha). 
\end{align*}

Now we can reduce the problem \eqref{tVP2} to a problem of an ordinary differential equation only for $\Phi$.
Integrating \eqref{tfform+} over $\mathbb R^{n}$ and using the change of variables 
$\sqrt{(\xi_{1}-\alpha)^{2}-2q_{+}\Phi}=-(\zeta_{1}-\alpha)$,
$\sqrt{(\xi_{1}-\alpha)^{2}-2q_{+}\Phi+2q_{+}\Phi_{max})}=(\zeta_{1}-\alpha)$,
and $\sqrt{(\xi_{1}-\alpha)^{2}-2q_{+}\Phi}=\zeta_{1}-\alpha$ 
for the first, second, and third terms on the right hand side, respectively, we see that
\begin{align}
&\int_{\mathbb R^{n}} F_{+}(X,\xi) \,d\xi  
\notag\\
&= \int_{{\mathbb R}^n}F_{+}(0,\xi)\frac{|\xi_1-\alpha|}{\sqrt{(\xi_1-\alpha)^2+2q_{+}\Phi(X)}}\,d\xi
\notag\\
&\quad +2\int_{\mathbb R^{n-1}}\int_{\sqrt{2q_{+}\Phi_{max}-2q_{+}\Phi(X)}+\alpha}^{\sqrt{2q_{+}\Phi_{max}}+\alpha}
F_{+}(X_{*},\xi)\frac{\xi_1-\alpha}{\sqrt{(\xi_1-\alpha)^2+2q_{+}\Phi(X)-2q_{+}\Phi_{max}}}\,d\xi_1\,d\xi'
\notag \\
&=  \rho_{+}(\Phi(X);\Phi_{max},F_{+}(X_{*},\cdot),F_{+}(0,\cdot)),
\label{trho+'}
\end{align}
where $\rho_{+}$ is the same function defined in \eqref{trho++}.
On the other hand, integrating \eqref{tfform-} over $\mathbb R^{n}$ and using the change of variables 
$\sqrt{(\xi_{1}-\alpha)^{2}+2q_{-}\Phi}=-(\zeta_{1}-\alpha)$ and 
$\sqrt{(\xi_{1}-\alpha)^{2}+2q_{-}\Phi}=\zeta_{1}-\alpha$ for the first and second terms on the right hand side, respectively,
we see that
\begin{align}
\int_{\mathbb R^{n}} F_{-}(X,\xi) \,d\xi  
&=\int_{{\mathbb R}^n}F_{-}(0,\xi)\frac{|\xi_1-\alpha|}{\sqrt{(\xi_1-\alpha)^2-2q_{-}\Phi(X)}} \chi((\xi_1-\alpha)^2-2q_{-}\Phi(X)) \,d\xi
\notag \\
&=  \rho_{-}(\Phi(X);F_{-}(0,\cdot)),
\label{trho-'}
\end{align}
where $\rho_{-}$ is the same function defined in \eqref{trho--}.
Substituting \eqref{trho+'} and \eqref{trho-'} into \er{teq4}, we arrive at
\begin{equation}\lb{tphieq1} 
\partial_{XX} \Phi 
= e_{+}\rho_{+}(\Phi;\Phi_{max},F_{+}(X_{*},\cdot),F_{+}(0,\cdot)) - e_{-}\rho_{-}(\Phi;F_{-}(0,\cdot)), \quad X\in \mathbb R.
\end{equation}
Now we claim that $\rho_{\pm} \in C([0,\Phi_{max}])$.
Indeed, owing to $ F_{\pm} \in C(\mathbb T_{\gamma};L^{1}(\R^{n}))$ in Definition \ref{tDefS1}, there hold that
\begin{align*}
\rho_{\pm}(\Phi(X))= \| F_{\pm}(X) \|_{L^{1}} \in C(\mathbb T_{\gamma}).
\end{align*}
This implies $\rho_{\pm} \in C([0,\Phi_{max}])$ with the aid of 
$ \Phi \in C(\mathbb T_{\gamma})$ and the fact that $\partial_{X}\Phi(X)\gtrless 0$ for $ X_{*} \gtrless X$ with $X,X_{*} \in [0,\gamma]$.
Thus the claim is valid. 
Multiply \er{tphieq1} by $\D_{X}\Phi$, integrate it over $(0,X)$, and use the condition (ii) in Definition \ref{tDefS1} to obtain \eqref{tphieq2} with $\beta=\Phi_{max}$, $G(\xi)=F_{+}(X_{*},\xi)$, and $H_{\pm}(\xi)=F_{\pm}(0,\xi)$, i.e.
\begin{gather*}
(\D_{X} \Phi)^2=2V(\Phi;\Phi_{max},F_{+}(X_{*},\cdot),F_{\pm}(0,\cdot)), 
\\
\begin{aligned}
&V(\Phi;\Phi_{max},F_{+}(X_{*},\cdot),F_{\pm}(0,\cdot))
\\
&=\int_{0}^{\Phi} \left(e_{+}\rho_{+}(\varphi;\Phi_{max},F_{+}(X_{*},\cdot),F_{+}(0,\cdot)) - e_{-}\rho_{-}(\varphi;F_{-}(0,\cdot) )\right)  d\vphi,
\end{aligned}
\notag
\end{gather*}
where $V$ is the same function defined in \eqref{tV0}.
Thus $\Phi$ must solve \eqref{tphieq2}.

We prove that $V$ satisfies \eqref{tG-beta2} and \eqref{tG-beta3} 
with $\beta=\Phi_{max}$, $G(\xi)=F_{+}(X_{*},\xi)$, and $H_{\pm}(\xi)=F_{\pm}(0,\xi)$.
From \eqref{tphieq2} and the condition (ii) in Definition \ref{DefS1}, the following hold:
\begin{gather*}
V(\Phi;\Phi_{max},F_{+}(X_{*},\cdot),F_{\pm}(0,\cdot))>0 \ \ \text{for $\Phi \in (0,\Phi_{max})$}, 
\label{tV>0}\\
V(\Phi_{max};\Phi_{max},F_{+}(X_{*},\cdot),F_{\pm}(0,\cdot))=0,
\end{gather*}
where are the same as in \eqref{tG-beta2}. Thus \eqref{tG-beta2} must hold.
Next we show that 
\begin{gather*}
\int_{0}^{\Phi_{max}} \frac{d\Phi}{\sqrt{V(\Phi;\Phi_{max},F_{+}(X_{*},\cdot),F_{\pm}(0,\cdot))}} <+ \infty,
\end{gather*}
which is the same condition as \eqref{tG-beta3}. 
Using \eqref{tphieq2}, we arrive at
\begin{align*}
\int_{0}^{\Phi_{max}} \frac{d\Phi}{\sqrt{2V(\Phi;\Phi_{max},F_{+}(X_{*},\cdot),F_{\pm}(0,\cdot))}}
&=\int_{0}^{X_{*}} \frac{\partial_{X} \Phi (X)}{\sqrt{2V(\Phi(X);\Phi_{max},F_{+}(X_{*},\cdot),F_{\pm}(0,\cdot))} } dX
\\
&=\int_{0}^{X_{*}} 1dX<+\infty.
\end{align*}
Thus \eqref{tG-beta3} must hold.

Finally, we  prove \eqref{tperiod1}--\eqref{symm2}. We recall $\Phi(X_{*})=\Phi_{max}=\beta$ and then see from \eqref{tphieq2} and the condition (ii) in Definition \ref{tDefS1} that 
\begin{align*}
X_{*} &=  \int_{0}^{X_{*}} 1 dX
= \int_{0}^{X_{*}} \frac{\partial_{X} \Phi (X)}{\sqrt{2V(\Phi(X);\Phi_{max},F_{+}(X_{*},\cdot),F_{\pm}(0,\cdot))} }dX
\\
&= \int_{0}^{\beta} \frac{d\Phi}{\sqrt{2V(\Phi;\Phi_{max},F_{+}(X_{*},\cdot),F_{\pm}(0,\cdot))} }
\\
&= \int_{X_{*}}^{\gamma} \frac{-\partial_{X} \Phi (X)}{\sqrt{2V(\Phi(X);\Phi_{max},F_{+}(X_{*},\cdot),F_{\pm}(0,\cdot))} }dX
=  \int_{X_{*}}^{\gamma} 1 dX= \gamma -X_{*},
\end{align*}
which gives $X_{*}=\frac{\gamma}{2}$, \eqref{tperiod1}, and \eqref{symm1}.
To show \eqref{symm2}, we study the uniqueness of the solution $\Phi$ of 
the problem \eqref{tphieq2} with $\Phi(0)=\Phi(\gamma)=0$, $\Phi(\frac{\gamma}{2})=\Phi_{max}$,  and $\partial_{X}\Phi(X)\gtrless 0$ for $\frac{\gamma}{2} \gtrless X$, where $X \in [0,\gamma]$.
Suppose that $\Phi_1$ and $\Phi_2$ are solutions of the problem, and $\Phi_1\not\equiv\Phi_2$ holds.
It is easy to see from \eqref{symm1} and the condition (ii) in Definition \ref{tDefS1} that $\Phi_1$ and $\Phi_2$ solve
\begin{subequations}\label{tneq1}
\begin{gather}
\D_{X}{\Phi}=\pm\sqrt{2V(\Phi(X);\Phi_{max},F_{+}(X_{*},\cdot),F_{\pm}(0,\cdot))}, \quad \pm \left(X -\frac{\gamma}{2}\right) < 0, 
\\
\Phi(0)=\Phi(\gamma)=0, \quad \Phi\left(\frac{\gamma}{2}\right)=\Phi_{max}.
\end{gather}
\end{subequations}
In much the same way as in the \footnote{It is written in the first paragraph of the proof of Lemma \ref{uniqueness1}.}proof of the uniqueness of the solution of the problem \eqref{uniquePhi}, one can show that the solution of \eqref{tneq1} is unique.
On the other hand, $\Psi(X)=\Phi(\gamma-X)$ is also a solution of the problem \eqref{tneq1}.
Therefore, $\Phi(X)=\Phi(\gamma-X)$ holds.
This fact together with $\Phi \in C^{2}({\mathbb T}_{\gamma})$ leads to \eqref{symm2}.
The proof is complete.
\end{proof}

\subsection{Solvability}\label{S4.2}

In this subsection, we  prove Theorem \ref{texistence1}.

\begin{proof}[Proof of Theorem \ref{texistence1}.]
Due to Lemma \ref{tneed1}, it suffices to show the solvability of the problem \eqref{tVP2} for some period $\gamma$
provided that there exists a quadruplet $(\beta,G,H_{\pm})$ with \eqref{tG-beta}.

We find $\gamma$ and construct $\Phi \in C^{2}({\mathbb T}_{\gamma})$ by solving \eqref{tphieq2}. 
To this end, we solve the following problem for $Y>0$:
\begin{gather}
\D_{Y}{\Psi}(Y)=\sqrt{2V(\Psi(Y);\beta,G,H_{\pm})}, \quad \Psi(0)=\frac{1}{2}\beta.
\label{tpsieq1}
\end{gather}
It follows from Lemma \ref{trhopm} and \eqref{tG-beta2} that
$\sqrt{V(\cdot;\beta,G,H_{\pm})}$ is Lipschitz continuous on $[\ve,\beta-\ve]$ for any suitably small $\ve>0$.
Therefore, the problem \eqref{tpsieq1} has a solution $\Psi$ which is strictly increasing unless $\Psi$ attains $\beta$.
Suppose that $\Psi$ cannot attain $\beta$ at some finite point. We observe that 
\begin{align*}
+\infty&=\int_{0}^{+\infty} 1dY
=\int_{0}^{+\infty} \frac{\partial_{Y} \Psi (Y)}{\sqrt{2V(\Psi(Y);\beta,G,H_{\pm})} } dY 
=\int_{\beta/2}^{\beta} \frac{d\Phi}{\sqrt{2V(\Phi;\beta,G,H_{\pm})}}<+\infty,
\end{align*}
where we have used \eqref{tG-beta3} in deriving the last equality. 
This is a contradiction.
Thus $\Psi$ attains $\beta$ at a point $Y=Y_{*}<+\infty$.
Now it is clear that the problem \eqref{tpsieq1} is solvable on the interval $[0,Y_{*}]$.

Next we solve the problem \eqref{tpsieq1} for $Y<0$. Similarly as above, 
the problem \eqref{tpsieq1} has a solution $\Psi$ which is strictly increasing.
In the same way as above, 
it follows that the solution $\Psi$ attains zero at some point $Y=Y_{\#}<0$.
Therefore, setting $X=Y-Y_{\#}$, we conclude that $\Psi \in C^{1}([0,Y_{*}-Y_{\#}])$ solves
\begin{gather}
\D_{X}{\Psi}(X)=\sqrt{2V(\Psi(X);\beta,G,H_{\pm})}, \quad \Psi(0)=0, \quad \Psi(Y_{*}-Y_{\#})=\beta.
\label{tpsieq2}
\end{gather}

We now set $\gamma:=2(Y_{*}-Y_{\#})$ and define $\Phi$ by
\begin{gather}
\Phi(X)=\left\{
\begin{array}{ll}
\Psi(X) & \text{for $X \in [0,\frac{\gamma}{2}]$},
\\
\Psi(\gamma-X) & \text{for $X \in (\frac{\gamma}{2},\gamma]$}.
\end{array}
\right.
\end{gather}
It is straightforward to see that $\Phi \in C^{2}({\mathbb T}_{\gamma})$ satisfies \eqref{tphieq1} 
and the condition (ii) in Definition \ref{tDefS1}. 


Now by using $\Phi$ and $\beta=\Phi_{max}$, we define $F_{+}$  and  $F_{-}$ as \eqref{tfform+} and \eqref{tfform-}, respectively, i.e. 
\begin{align*}
{F}_{+}(X,\xi)
& =H_{+}(-\sqrt{(\xi_1-\alpha)^2-2q_{+}{\Phi}(X)}+\alpha,\xi')
\chi((\xi_{1}-\alpha)^{2}-2q_{+}\Phi(X))\chi(-(\xi_{1}-\alpha)) 
\\
& \qquad
+G(\sqrt{(\xi_1-\alpha)^2-2q_{+}\Phi(X)+2q_{+}\Phi_{max}}+\alpha,\xi')
\chi(-(\xi_1-\alpha)^2+2q_{+}\Phi(X)) 
\\
& \qquad
+H_{+}((\sqrt{(\xi_1-\alpha)^2-2q_{+}\Phi(X)}+\alpha,\xi')
\chi((\xi_{1}-\alpha)^{2}-2q_{+}\Phi(X))\chi(\xi_{1}-\alpha),
\\
F_{-}(X,\xi)
&=H_{-}(-\sqrt{(\xi_1-\alpha)^2+2q_{-}\Phi(X)}+\alpha,\xi')\chi(-(\xi_1-\alpha))
\\
&\quad +H_{-}(\sqrt{(\xi_1-\alpha)^2+2q_{-}\Phi(X)}+\alpha,\xi')\chi(\xi_1-\alpha).
\end{align*}
One can see that $F_{\pm}$ satisfy the conditions (i)--(iii) in Definition \ref{tDefS1} 
in much the same way as in the proof of Theorem \ref{existence1}. 
The condition (iv) is validated by \eqref{tfform+}--\eqref{tphieq1}.
The proof is complete.
\end{proof}

\subsection{Nonuniqueness}\label{S4.3}

This subsection is devoted to the proofs of Theorems \ref{tnonunique1} and \ref{tnonunique2}. 
First we show Theorem \ref{tnonunique1}.

\begin{proof}[Proof of Theorem \ref{tnonunique1}]
Due to Theorem \ref{texistence1},
if quadruplets $(\beta,G,H_{\pm})$ with \eqref{tG-beta} exist,
we have solutions $(F_{\pm},\Phi)$ of the problem \eqref{tVP2} 
replaced the period $\gamma$ by $\tilde{\gamma}$ for each $(\beta,G,H_{\pm})$, 
where
\begin{gather*}
\tilde{\gamma} := 2\int_{0}^{\beta} \frac{d\Phi}{\sqrt{V(\Phi;\beta,G,H_{\pm})}}.
\end{gather*}
Owing to $\Phi(0)=\beta$ and the forms \eqref{tfform+}--\eqref{tfform-} in Theorem \ref{texistence1}, 
if the quadruplets are different, then so are the solutions.
For the case that the period $\tilde{\gamma}$ is not equal to $\gamma$,
we take a new quadruplet $\bigl(\beta, \frac{\tilde{\gamma}^{2}}{\gamma^{2}} G, \frac{\tilde{\gamma}^{2}}{\gamma^{2}} H_{\pm}\bigr)$, 
and then see that it also satisfies \eqref{tG-beta}.
The solution constructed from the new quadruplet has a period $\gamma$, since $V\bigl(\Phi;\beta, \frac{\tilde{\gamma}^{2}}{\gamma^{2}} G, \frac{\tilde{\gamma}^{2}}{\gamma^{2}} H_{\pm}\bigr)=\frac{\tilde{\gamma}^{2}}{\gamma^{2}}V(\Phi;\beta,G,H_{\pm})$ holds.
Therefore, it is sufficient to find infinite many quadruplets $(\beta,G,H_{\pm})$ with \eqref{tG-beta}.

For any $\tau>0$ and $\beta>0$, we define 
 by
\begin{align*}
H_{\tau,+}(\xi)
&:=\Biggl\{
\begin{array}{ll}
\frac{1}{2e_+\sqrt{2q_+}} & \hbox{if} \quad
(\xi_1,\xi')\in[-\sqrt{2q_+\tau}+\alpha,\sqrt{2q_+\tau}+\alpha]\times[0,1]^{n-1}, \\
0 & \hbox{otherwise}, 
\end{array}
\\
H_{\tau,-}(\xi)
&:=\Biggl\{
\begin{array}{ll}
\frac{1}{2e_-\sqrt{2q_-}} & \hbox{if} \quad
(\xi_1,\xi')\in[-\sqrt{2q_-\beta+2q_-\tau}+\alpha,
-\sqrt{2q_-\beta}+\alpha]\times[0,1]^{n-1}, \\
\frac{1}{2e_-\sqrt{2q_-}} & \hbox{if} \quad
(\xi_1,\xi')\in[\sqrt{2q_-\beta}+\alpha,
\sqrt{2q_-\beta+2q_-\tau}+\alpha]\times[0,1]^{n-1}, \\
0 & \hbox{otherwise},
\end{array}
\\
G(\xi)&:=0.
\end{align*}
It is clear that 
\begin{gather*}
H_{\tau,+} \in L^1(\R^n), \quad 
H_{\tau,-} \in L^1(\R^n) \cap L_{loc}^{3}({\mathbb R};L^{1}(\mathbb R^{n-1})), \quad 
H_{\tau,\pm} \geq 0,
\\
G \in L^1_{loc}((\alpha,+\infty) \times \R^{n-1}) \cap L_{loc}^{3}([\alpha,+\infty);L^{1}(\mathbb R^{n-1})), \quad 
G \geq 0.
\end{gather*}

Let us show that the quadruplets $(\beta,G,H_{\tau,\pm})$ satisfy \eqref{tG-beta}.
From the definition of $H_{\tau,-}$, it is obvious that \eqref{tG-beta1} holds. 
By the direct computation, there hold for $\Phi \in [0,\beta]$ that
\begin{align*}
e_+\rho_+(\Phi;\beta,G,H_{\tau,+})&=\sqrt{\Phi+\tau}-\sqrt{\Phi}=:f_\tau(\Phi),
\\
e_-\rho_-(\Phi;H_{\tau,-})&
=f_\tau(\beta-\Phi),
\\
V(\Phi;\beta,G,H_{\tau,\pm})
&=\int_0^\Phi(f_\tau(\varphi)-f_\tau(\beta-\varphi))\,d\varphi,
\end{align*}
where $\rho_{\pm}$ and $V$ are defined in \eqref{tV0}--\eqref{trho--}.
Evaluating $V(\Phi;\beta,G,H_{\tau,\pm})$ at $\Phi=0,\beta$ gives 
$V(0;\beta,G,H_{\tau,\pm})=V(\beta;\beta,G,H_{\tau,\pm})=0$.
Using the fact that the function $f_\tau$ is strictly increasing on $[0,\beta]$, we observe that
\begin{gather}\label{bi1}
\frac{d}{d\Phi}V(\Phi;\beta,G,H_{\tau,\pm})
=f_\tau(\Phi)-f_\tau(\beta-\Phi)
\left\{
\begin{array}{ll}
>0, & \Phi\in[0,\frac{1}{2}\beta), \\
=0, & \Phi=\frac{1}{2}\beta, \\
<0, & \Phi\in(\frac{1}{2}\beta,\beta]. 
\end{array}\right.
\end{gather}
From these, we deduce that $V(\Phi;\beta,G,H_{\tau,\pm})$ is positive on $(0,\beta)$.
Hence, \eqref{tG-beta2} holds.
Furthermore, from \eqref{bi1}, we have
\begin{gather}\label{V<0}
\frac{d}{d\Phi}V(0;\beta,G,H_{\tau,\pm})>0, \quad
\frac{d}{d\Phi}V(\beta;\beta,G,H_{\tau,\pm})<0.
\end{gather}
These lead to \eqref{tG-beta3} with the aid of the Taylor theorem.
The proof is complete.
\end{proof}

Next we prove Theorem \ref{tnonunique2}.

\begin{proof}[Proof of Theorem \ref{tnonunique2}]
We make use of the quadruplets $(\beta,G,H_{\tau,\pm})$ in the proof of Theorem \ref{tnonunique1} again.
For these $(\beta,G,H_{\tau,\pm})$, we see that 
$V(\Phi;\beta,G,H_{\tau,\pm})=V(\Phi;\beta,0,H_{\tau,\pm})$ satisfies \eqref{V<0}, 
and $\beta_{*}$ defined in \eqref{beta*} is equal to $+\infty$.
In much the same way as in the proof of Lemma \ref{nonunique2},
one can find a pair $(\tilde{\beta},\tilde{G})$ so that $(\tilde{\beta},\tilde{G},H_{\pm,\tau})$ satisfies \eqref{tG-beta} and $G\not\equiv 0$.
Therefore, adjusting the period of solutions by considering constant multiples of $\tilde{G}$ and $H_{\tau,\pm}$ similarly as in the proof of Theorem \ref{tnonunique1}, we have infinite many solutions with the period $\gamma$ and trapped ions.
\end{proof}

\subsection{Nonuniqueness for the case with the Boltzmann relation}\label{S4.4}

This subsection is devoted to the proof of Corollary \ref{tnonunique3}. 
We can show the assertions expect the last sentence in Corollary \ref{tnonunique3} 
in the same way as in the proof of Corollary \ref{cor1}.
We focus ourself on the proof of the last sentence, that is, 
`` There is a constant $\gamma_{\star}>0$ satisfying the following properties: 
for any $\alpha$ and $\gamma \in (0,\gamma_{\star}]$, 
there exist infinite many solutions of the problem \eqref{tVP3} with the period $\gamma$ ''.

Set $\tau_\star:=\beta_\star:=\frac{1}{10\kappa}$.
For any $\tau \in (0,\tau_{\star}]$ and $\beta \in (0,\beta_{\star}]$, we define
\begin{align*}
H_{\tau,\beta,+}(\xi)
&:=\Biggl\{
\begin{array}{ll}
\frac{2e_-\rho}{3\kappa e_+\sqrt{2q_+}}A_{\tau,\beta} 
& \hbox{if} \quad \xi\in[\alpha,\sqrt{2q_{+}\tau}+\alpha]\times[0,1]^{n-1}, \\
0 & \hbox{otherwise}, 
\end{array}
\\
H_-(\xi)
&:=\frac{\rho}{\displaystyle
\int_{\R^n}e^{\frac{-\kappa}{2q_-}|\xi|^2}\,d\xi}
e^{\frac{-\kappa}{2q_-}\{(\xi_1-\alpha)^2+|\xi'|^2\}}, 
\\
G(\xi)&:= 0,
\end{align*}
where we will choose a constant $A_{\tau,\beta}>0$ in \eqref{ATB1} below, and $H_{-}$ is the same function in Corollary \ref{cor1} and hence gives the Boltzmann realtion.
It is clear that 
\begin{subequations}\label{regHG1}
\begin{gather}
H_{\tau,\beta,+}\in L^1(\R^n), \quad
H_-\in L^1(\R^n)\cap L^3_{loc}(\R;L^1(\R^{n-1})), \quad
H_{\tau,\beta,+},H_-\geq 0, 
\\
G\in L^1_{loc}((\alpha,+\infty)\times\R^{n-1})
\cap L^3_{loc}([\alpha,+\infty);L^1(\R^{n-1})), \quad
G\geq 0.
\end{gather}
\end{subequations}
It is seen by the direct computation that
\begin{align}
\rho_+(\Phi;\beta,G,H_{\tau,\beta,+})&=\frac{2}{3\kappa e_+}e_-\rho A_{\tau,\beta}(\sqrt{\Phi+\tau}-\sqrt{\Phi}),
\notag\\
\rho_-(\Phi;H_-)&=\rho e^{-\kappa\Phi},
\notag\\
V(\Phi;\beta,G,H_{\tau,\beta,+},H_-)&
=\frac{e_-\rho}{\kappa}\Bigl(A_{\tau,\beta}\{(\Phi+\tau)^{\frac{3}{2}}-\Phi^{\frac{3}{2}}-\tau^{\frac{3}{2}}\}-(1-e^{-\kappa\Phi})\Bigr),
\label{VTB0}
\end{align}
where $\rho_{\pm}$ and $V$ are defined in \eqref{tV0}--\eqref{trho--}. 
For notational convenience, we use 
\begin{gather}
\tilde{V}_{\tau,+}(\Phi):=(\Phi+\tau)^{\frac{3}{2}}-\Phi^{\frac{3}{2}}-\tau^{\frac{3}{2}}, \quad
\tilde{V}_-(\Phi):=1-e^{-\kappa\Phi}, 
\notag\\
\tilde{V}_{\tau,\beta}(\Phi):=A_{\tau,\beta}\tilde{V}_{\tau,+}(\Phi)-\tilde{V}_-(\Phi),
\label{VTB1}
\end{gather}
where $V(\Phi;\beta,G,H_{\tau,\beta,+},H_-)=\frac{e_-\rho}{\kappa} \tilde{V}_{\tau,\beta}(\Phi)$ holds.
Now we choose the constant $A_{\tau,\beta}$ as
\begin{gather}\label{ATB1}
A_{\tau,\beta}:=\frac{\tilde{V}_-(\beta)}{\tilde{V}_{\tau,+}(\beta)},
\end{gather}
which gives $\tilde{V}_{\tau,\beta}(\beta)=0$.

It is seen in Lemma \ref{LemA} below that for any $\tau \in (0,\tau_{\star}]$ and $\beta \in (0,\beta_{\star}]$,
the quadruplets $(\beta,G,H_{\tau,\beta,+},H_-)$ satisfy \eqref{tG-beta}. 
Therefore, we have infinite many solutions of the problem \eqref{tVP3} by virtue of Theorem \ref{texistence1} with $H_{-}$ chosen above.
Selecting suitably $\tau$ and $\beta$, we find infinite many solutions with the same period $\gamma$.
We observe the period of solutions obtained from the quadruplets $(\beta,G,H_{\tau,\beta,+},H_-)$ by using \eqref{tperiod1} as follows:
\begin{align}\label{gamma00}
\gamma&=2\int_0^\beta\frac{d\Phi}{\sqrt{V(\Phi;\beta,G,H_{\tau,\beta,+},H_-)}}
\notag \\
&=2\sqrt{\frac{\kappa}{e_-\rho}}
\int_0^\beta\frac{d\Phi}{\sqrt{\tilde{V}_{\tau,\beta}(\Phi)}}
=2\sqrt{\frac{\kappa}{e_-\rho}}
\int_0^1\tilde{W}_{\tau,\beta}(\Psi)\,d\Psi,
\end{align}
where the function $\tilde{W}_{\tau,\beta}(\Psi)$ is defined by 
\begin{gather}\label{WTB1}
\tilde{W}_{\tau,\beta}(\Psi)
:=\frac{\beta}{\sqrt{\tilde{V}_{\tau,\beta}(\beta\Psi)}}, \quad \Psi\in(0,1).
\end{gather}
Let us set
\begin{gather*}
\gamma_\star:=\tilde{\gamma}_\star\sqrt{\frac{\kappa}{e_-\rho}}, \quad \text{where \quad
$\displaystyle \tilde{\gamma}_\star:=2\int_0^1\tilde{W}_{\tau_\star,\beta_\star}(\Psi)\,d\Psi$}.
\end{gather*}
To complete the proof of Corollary \ref{tnonunique3},
it is sufficient to find infinite many pairs $(\tau,\beta) \in (0,\tau_{\star}] \times (0,\beta_{\star}]$ which satsify the following equality for each $\tilde{\gamma} \in (0,\tilde{\gamma}_{\star}]$:
\begin{gather}\label{GTB1}
\tilde{\gamma}=\tilde{\gamma}_{\tau,\beta}, \quad \text{where \quad
$\displaystyle \tilde{\gamma}_{\tau,\beta}:=2\int_0^1\tilde{W}_{\tau,\beta}(\Psi)\,d\Psi.$}
\end{gather}

First we show auxiliary properties in Lemmas \ref{LemA}--\ref{LemC}. 
Then we complete the proof of Corollary \ref{tnonunique3}.

\begin{lem}\label{LemA}
For any $(\tau,\beta) \in (0,\tau_{\star}] \times (0,\beta_{\star}]$,
the quadruplet $(\beta,G,H_{\tau,\beta,+},H_-)$ satisfies \eqref{tG-beta}, and
the problem \eqref{tVP3} has a solution.
\end{lem}
\begin{proof}
It is sufficient to show that the quadruplets $(\beta,G,H_{\tau,\beta,+},H_-)$ satisfy \eqref{tG-beta}.
Indeed, $(\beta,G,H_{\tau,\beta,+},H_-)$ also has the regularity \eqref{regHG1}, and then
Theorem \ref{texistence1} ensures the existence of a solution $(F_{\pm},\Phi)$ of the problem \eqref{tVP2} for each $(\tau,\beta)$.
Then we can deduce that $(F_{+},\Phi)$ solves the problem \eqref{tVP3} in much the same way as in the proof of Corollary \ref{cor1}.

From now on we show that all $(\beta,G,H_{\tau,\beta,+},H_-)$ satisfy \eqref{tG-beta}.
It is clear that \eqref{tG-beta1} holds. 
To show \eqref{tG-beta2}, we start from proving that
\begin{gather}\label{dVpm}
\frac{d}{d\Phi}\tilde{V}_{\tau,\beta}(0)>0, \quad
\frac{d}{d\Phi}\tilde{V}_{\tau,\beta}(\beta)<0,
\end{gather}
where $\tilde{V}_{\tau,\beta}$ is defined in \eqref{VTB1}.
It is seen that
\begin{gather*}
\frac{d}{d\Phi}\tilde{V}_{\tau,\beta}(0)
=\frac{f(\beta)}{\tilde{V}_{\tau,+}(\beta)}, \quad
\text{where} \quad
f(\beta):=\frac{3}{2}\tau^{\frac{1}{2}}(1-e^{-\kappa\beta})
-\kappa\{(\beta+\tau)^{\frac{3}{2}}
-\beta^{\frac{3}{2}}-\tau^{\frac{3}{2}}\}.
\end{gather*} 
It follows from direct computation that $f(0)=0$ and $f'(\beta)>0$ for any $\beta\in(0,\beta_{\star}]$.
Hence, $\frac{d}{d\Phi}\tilde{V}_{\tau,\beta}(0)>0$ holds.
On the other hand, 
\begin{gather*}
\frac{d}{d\Phi}\tilde{V}_{\tau,\beta}(\beta)
=\frac{g(\beta)}
{\tilde{V}_{\tau,+}(\beta)}, \quad
\text{where} \quad
g(\beta):=\tilde{V}_{\tau,+}'(\beta)\tilde{V}_-(\beta)
-\tilde{V}_{\tau,+}(\beta)\tilde{V}_-'(\beta).
\end{gather*}
To deduce that $\frac{d}{d\Phi}\tilde{V}_{\tau,\beta}(\beta)<0$ for any $\beta\in(0,\beta_{\star}]$, 
we show that $g(0)=0$ and $g'(\beta)<0$ holds for any $\beta\in(0,\beta_{\star}]$. 
It is obvious that $g(0)=0$. We observe that
\begin{align*}
g'(\beta)
&=\frac{3}{4}(1-e^{-\kappa\beta})
\{(\beta+\tau)^{-\frac{1}{2}}-\beta^{-\frac{1}{2}}\}
+\kappa^2 e^{-\kappa\beta}
\{(\beta+\tau)^{\frac{3}{2}}-\beta^{\frac{3}{2}}-\tau^{\frac{3}{2}}\}
\\
& \leq \frac{3}{4}(1-e^{-1})
\{(\beta+\tau)^{-\frac{1}{2}}-\beta^{-\frac{1}{2}}\}
+\kappa^2 
\{(\beta+\tau)^{\frac{3}{2}}-\beta^{\frac{3}{2}}-\tau^{\frac{3}{2}}\},
\end{align*}
where we have used $\beta \leq \beta_{\star}=\frac{1}{10\kappa}$.
We note that the rightmost is an increasing function with respect to $\beta$, and then see that
\begin{align}
\kappa^2 \{(\beta+\tau)^{\frac{3}{2}}-\beta^{\frac{3}{2}}-\tau^{\frac{3}{2}}\}
&\leq \kappa^2 \{(\beta_{\star}+\tau)^{\frac{3}{2}}-\beta_{\star}^{\frac{3}{2}}-\tau^{\frac{3}{2}}\}
\notag\\
&\leq \kappa^2 \{(\beta_{\star}+\tau)^{\frac{3}{2}}-\beta_{\star}^{\frac{3}{2}}\}
\leq 7\kappa^2\beta_\star^\frac{1}{2}\tau,
\label{Astar1}
\end{align}
where we have used $\tau \leq \tau_{\star}=\beta_{\star}$ in deriving the last inequality.
It is also seen by direct computation that 
\begin{gather}\label{Astar2}
g'(\beta) \leq 
\frac{3}{4}(1-e^{-1}) \{(\beta+\tau)^{-\frac{1}{2}}-\beta^{-\frac{1}{2}}\}
+ 7\kappa^2\beta_\star^\frac{1}{2}\tau<0, \quad \beta \in (0,\beta_{\star}].
\end{gather}
Hence, we deduce that $\frac{d}{d\Phi}\tilde{V}_{\tau,\beta}(\beta)<0$.
Thus \eqref{dVpm} holds.

Furthermore, 
we claim that $\frac{d^{2}}{d\Phi^{2}} \tilde{V}_{\tau,\beta}(\Phi_{0})<0$ holds for any roots $\Phi_{0} \in [0,\beta]$ of the function $\frac{d}{d\Phi} \tilde{V}_{\tau,\beta}(\Phi)$. 
This can be shown as follows.
The fact $\frac{d}{d\Phi} \tilde{V}_{\tau,\beta}(\Phi_{0})=0$ gives 
\begin{gather*}
(\Phi_0+\tau)^{\frac{1}{2}}-\Phi_0^{\frac{1}{2}}
=\frac{2}{3A_{\tau,\beta}}\kappa e^{-\kappa\Phi_0}.
\end{gather*}
By using this, we arrive at the claim, i.e.
\begin{gather*}
\frac{d^{2}}{d\Phi^{2}} \tilde{V}_{\tau,\beta}(\Phi_{0})
=-\frac{\kappa e^{-\kappa\Phi_0}}
{3\Phi_0^{\frac{1}{2}}(\Phi_0+\tau)^{\frac{1}{2}}}
\Bigl\{1-3\kappa\Phi_0^{\frac{1}{2}}(\Phi_0+\tau)^{\frac{1}{2}} \Bigr\}<0,
\end{gather*}
where we have used the fact that $\Phi_0\leq\beta\leq\beta_\star$ and $\tau\leq\tau_\star=\beta_\star=\frac{1}{10\kappa}$ in deriving the last inequality.

Now we show \eqref{tG-beta2} by using the claim and \eqref{dVpm}.
It follows from \eqref{VTB0}--\eqref{ATB1} that $V(0;\beta,G,H_{\tau,\beta,+},$ $H_-)=0$.
Let us show the positivity of $V(\Phi;\beta,G,H_{\tau,\beta,+},H_-)=\frac{e_-\rho}{\kappa} \tilde{V}_{\tau,\beta}(\Phi)$.
There exits a root of $\frac{d}{d\Phi} \tilde{V}_{\tau,\beta}(\Phi)$ in $[0,\beta]$ thanks to the intermediate value theorem with \eqref{dVpm}.
If the root is unique in $[0,\beta]$, we can deduce that the positivity owing to
$\tilde{V}_{\tau,\beta}(0)=\tilde{V}_{\tau,\beta}(\beta)=0$.
Let us prove the uniqueness.
Set $K:=\bigl\{\Phi \in [0,\beta] \, ; \, \frac{d}{d\Phi} \tilde{V}_{\tau,\beta}(\Phi)=0 \bigr\}$, $\Phi_{*}:=\inf K$, and $\Phi^{*}:=\sup K$.
We note that $\frac{d}{d\Phi} \tilde{V}_{\tau,\beta}(\Phi_{*})=\frac{d}{d\Phi} \tilde{V}_{\tau,\beta}(\Phi^{*})=0$ and $0<\Phi_{*}\leq \Phi^{*} \leq \beta$ hold.
We only need to show $\Phi_{*}=\Phi^{*}$. Suppose that $\Phi_{*}<\Phi^{*}$ holds.
From the above claim, we can find a constant $\delta>0$ such that 
\begin{gather*}
\frac{d}{d\Phi} \tilde{V}_{\tau,\beta}(\Phi)<0 \quad
\text{for $\Phi\in(\Phi_*,\Phi_*+\delta]$}, \qquad
\frac{d}{d\Phi} \tilde{V}_{\tau,\beta}(\Phi)>0 \quad
\text{for $\Phi\in[\Phi^*-\delta,\Phi^*)$},
\end{gather*}
where $\Phi_*+\delta<\Phi^*-\delta$ also holds.
We notice that $\frac{d}{d\Phi} \tilde{V}_{\tau,\beta}(\Phi_*+\delta)<0<\frac{d}{d\Phi} \tilde{V}_{\tau,\beta}(\Phi^*-\delta)$.
Hence, $\frac{d}{d\Phi}\tilde{V}_{\tau,\beta}(\Phi_{1}) = 0$ and
$\frac{d^{2}}{d\Phi^{2}}\tilde{V}_{\tau,\beta}(\Phi_{1}) \geq 0$ hold for a certain $\Phi_{1} \in [\Phi_*+\delta,\Phi^*-\delta]$, but this contradicts to the above claim. Thus $\Phi_{*}=\Phi^{*}$ holds. 
Consequently, we see the positivity of $V(\Phi;\beta,G,H_{\tau,\beta,+},H_-)$ and obtain \eqref{tG-beta2}.

Obviously, \eqref{tG-beta3} follows from 
\eqref{dVpm} and $V(\Phi;\beta,G,H_{\tau,\beta,+},H_-)=\frac{e_-\rho}{\kappa} \tilde{V}_{\tau,\beta}(\Phi)$ 
with the aid of the Taylor theorem.
The proof is complete.
\end{proof}

\begin{lem}\label{LemBa}
For each fixed $\beta \in (0,\beta_{\star}]$ and $\Psi \in (0,1)$, $\tilde{W}_{\tau,\beta}(\Psi)$  is a continuous function with respect to $\tau$ on the interval $(0,\tau_{*}]$. Furthermore, if $\tau_{1},\tau_{2} \in (0,\tau_{\star}]$ and $\tau_1<\tau_2$, then the following holds:
\begin{gather}\label{tW1}
\tilde{W}_{\tau_1,\beta}(\Psi)<\tilde{W}_{\tau_2,\beta}(\Psi).
\end{gather}
\end{lem}
\begin{proof}
It is straightforward to see the continuity of $\tilde{W}_{\tau,\beta}(\Psi)$ defined in \eqref{WTB1}.
It remains only to obtain the inequality \eqref{tW1}.
It suffices to show that for any $\Phi,\beta\in[0,+\infty)$ with $\Phi<\beta$ and any $\tau_1,\tau_2\in(0,\tau_\star]$ with $\tau_1<\tau_2$, 
\begin{gather}\label{Vtau+1}
\frac{\tilde{V}_{\tau_1,+}(\Phi)}{\tilde{V}_{\tau_1,+}(\beta)}
>\frac{\tilde{V}_{\tau_2,+}(\Phi)}{\tilde{V}_{\tau_2,+}(\beta)}.
\end{gather}
If this is true, we conclude from \eqref{VTB1} and \eqref{ATB1} that $\tilde{V}_{\tau_1,\beta}(\Phi)
>\tilde{V}_{\tau_2,\beta}(\Phi)$ which gives \eqref{tW1}.

In order to obtain \eqref{Vtau+1}, it is sufficent to prove that 
\begin{gather*}
\frac{d}{d\tau}\Bigl(
\frac{\tilde{V}_{\tau,+}(\Phi)}{\tilde{V}_{\tau,+}(\beta)}\Bigr)
=\frac{
\frac{d}{d\tau}\tilde{V}_{\tau,+}(\Phi)\cdot\tilde{V}_{\tau,+}(\beta)
-\tilde{V}_{\tau,+}(\Phi)\cdot\frac{d}{d\tau}\tilde{V}_{\tau,+}(\beta)}
{\{\tilde{V}_{\tau,+}(\beta)\}^2}
<0 \quad
\text{for $\tau\in(0,\tau_\star]$}.
\end{gather*}
This inequality is equivalent to the following:
\begin{gather*}
\frac{\frac{d}{d\tau}\tilde{V}_{\tau,+}(\Phi)}
{\tilde{V}_{\tau,+}(\Phi)}
<\frac{\frac{d}{d\tau}\tilde{V}_{\tau,+}(\beta)}
{\tilde{V}_{\tau,+}(\beta)} 
\quad \Leftrightarrow \quad
\frac{(\frac{\Phi}{\tau}+1)^{\frac{3}{2}}-(\frac{\Phi}{\tau})^{\frac{3}{2}}-1}
{(\frac{\Phi}{\tau}+1)^{\frac{1}{2}}-1}
>\frac
{(\frac{\beta}{\tau}+1)^{\frac{3}{2}}-(\frac{\beta}{\tau})^{\frac{3}{2}}-1}
{(\frac{\beta}{\tau}+1)^{\frac{1}{2}}-1}.
\end{gather*}
The last inequality holds if the following function $f$ is strictly decreasing on $(0,+\infty)$:
\begin{gather*}
f(x):=\frac{(x+1)^{\frac{3}{2}}-x^{\frac{3}{2}}-1}{(x+1)^{\frac{1}{2}}-1}=(x+1)^{\frac{1}{2}}-x^{\frac{1}{2}}+x-x^{\frac{1}{2}}(x+1)^{\frac{1}{2}}+2.
\end{gather*}
Indeed, it is easy to see that
\begin{align*}
f'(x)
=\frac{-\{(x+1)^{\frac{1}{2}}-x^{\frac{1}{2}}\}
-\{(x+1)^{\frac{1}{2}}-x^{\frac{1}{2}}\}^2}{2x^{\frac{1}{2}}(x+1)^{\frac{1}{2}}}
<0.
\end{align*}
Thus \eqref{Vtau+1} holds. The proof is complete.
\end{proof}

\begin{lem}\label{LemB}
For each fixed $\beta \in (0,\beta_{\star}]$, $\tilde{\gamma}_{\tau,\beta}$ is a continue function with respect to $\tau$ on the interval $(0,\tau_{*}]$. Furthermore, if $\tau_{1},\tau_{2} \in (0,\tau_{\star})$ and $\tau_1<\tau_2$, 
then $\tilde{\gamma}_{\tau_1,\beta}<\tilde{\gamma}_{\tau_2,\beta}$ holds.
\end{lem}
\begin{proof}
For any fixed $\beta\in(0,\beta_\star]$, we see from Lemma \ref{LemBa} that $\tilde{W}_{\tau,\beta}(\Psi)\leq\tilde{W}_{\tau_\star,\beta}(\Psi)$ for any $\tau\in(0,\tau_\star]$ and $\Psi\in(0,1)$.
It is also seen from \eqref{gamma00} that $\tilde{W}_{\tau_\star,\beta}(\Psi)$ is integrable over $(0,1)$.
Hence, Lemma \ref{LemBa} and the dominated convergence theorem give the continuity  of $\tilde{\gamma}_{\tau,\beta}$ defined in \eqref{GTB1}.
The inequality $\tilde{\gamma}_{\tau_1,\beta}<\tilde{\gamma}_{\tau_2,\beta}$ immediately follows from \eqref{tW1}.
\end{proof}

\begin{lem}\label{LemC}
For each fixed $\tau \in (0,\tau_{\star}]$, $\tilde{\gamma}_{\tau,\beta}$ is a continuous function with respect to $\beta$ on the interval $(0,\beta_{*}]$, and also $\lim_{\beta \to +0}\tilde{\gamma}_{\tau,\beta}=0$ holds.
\end{lem}
\begin{proof}
For any $(\tau,\beta)\in(0,\tau_\star] \times (0,\beta_\star]$, we define the helper function $\hat{V}_{\tau,\beta}^*(\Phi)$ by
\begin{gather*}
\hat{V}_{\tau,\beta}^*(\Phi)
:=\Biggl\{
\begin{array}{ll}
2\beta^{-1}\tilde{V}_{\tau,\beta}(\frac{1}{2}\beta)\Phi 
& \hbox{if} \quad \Phi\in[0,\frac{1}{2}\beta], \\
2\beta^{-1}\tilde{V}_{\tau,\beta}(\frac{1}{2}\beta)(\beta-\Phi) 
& \hbox{if} \quad \Phi\in[\frac{1}{2}\beta,\beta],
\end{array}
\end{gather*}
whose graph is a polygonal line connecting the three points $(0,0)$, $(\frac{1}{2}\beta, \tilde{V}_{\tau,\beta}(\frac{1}{2}\beta))$, and $(\beta,0)$ on the graph of $\tilde{V}_{\tau,\beta}(\Phi)$.
We can have the magnitude relationship between $\hat{V}_{\tau,\beta}^*(\Phi)$ and $\tilde{V}_{\tau,\beta}(\Phi)$ by studying the convexity of $\tilde{V}_{\tau,\beta}$ as follows.
It is seen that
\begin{gather*}
\frac{d^2}{d\Phi^2}\tilde{V}_{\tau,\beta}(\Phi) 
=\frac{h(\Phi)}{\tilde{V}_{\tau,+}(\beta)}, 
\end{gather*}
where
\begin{gather*}
h(\Phi):=\frac{3}{4}(1-e^{-\kappa\beta})
\{(\Phi+\tau)^{-\frac{1}{2}}-\Phi^{-\frac{1}{2}}\}
+\kappa^2 
\{(\beta+\tau)^{\frac{3}{2}}-\beta^{\frac{3}{2}}-\tau^{\frac{3}{2}}\}
e^{-\kappa\Phi}.
\end{gather*}
We estimate $h(\Phi)$ by using $\beta \leq \beta_{\star}=\frac{1}{10\kappa}$ as follows:
\begin{align*}
h(\Phi) 
&\leq \frac{3}{4}(1-e^{-1}) \{(\Phi+\tau)^{-\frac{1}{2}}-\Phi^{-\frac{1}{2}}\}
+\kappa^2 \{(\beta+\tau)^{\frac{3}{2}}-\beta^{\frac{3}{2}}-\tau^{\frac{3}{2}}\}
\\
& \leq \frac{3}{4}(1-e^{-1}) \{(\Phi+\tau)^{-\frac{1}{2}}-\Phi^{-\frac{1}{2}}\}
+7\kappa^2 \beta_\star^\frac{1}{2}\tau <0 \quad \text{for $\Phi \in (0,\beta_{\star}]$},
\end{align*}
where we have also used \eqref{Astar1} and \eqref{Astar2} in deriving the second and last inequalities, respectively. 
Therefore, for each $(\tau,\beta)\in(0,\tau_\star] \times (0,\beta_\star]$, we obtain the magnitude relationship that
$\hat{V}_{\tau,\beta}^*(\Phi) \leq\tilde{V}_{\tau,\beta}(\Phi)$ for $\Phi\in(0,\beta)$.

We define another helper function $\hat{W}_{\tau,\beta}^*(\Psi)$ by
\begin{gather*}
\hat{W}_{\tau,\beta}^*(\Psi) :=\frac{\beta}{\sqrt{\hat{V}_{\tau,\beta}^*(\beta\Psi)}} \quad \text{for $\Psi\in(0,1)$},
\end{gather*}
which is well-defined thanks to $\hat{V}_{\tau,\beta}^*(\beta\Psi)>0$.
It is easy to see that
\begin{gather*}
\hat{W}_{\tau,\beta}^*(\Psi)
= \hat{a}_{\tau}^{*}(\beta) \hat{U}^*(\Psi), \quad
\hat{a}_{\tau}^{*}(\beta):=\frac{\beta}{\sqrt{2\tilde{V}_{\tau,\beta}(\frac{1}{2}\beta)}}, \quad
\hat{U}^*(\Psi) :=\Biggl\{
\begin{array}{ll}
\Psi^{-\frac{1}{2}}
& \hbox{if $\Phi\in[0,\frac{1}{2}]$}, \\
(1-\Psi)^{-\frac{1}{2}} 
& \hbox{if $\Phi\in[\frac{1}{2},1]$}. 
\end{array}
\end{gather*}
We observe by applying the L'H\^opital's theorem that
\begin{align*}
\lim_{\beta\to+0} \{\hat{a}_{\tau}^{*}(\beta)\}^{-2}
&=2\lim_{\beta\to+0}\frac{\tilde{V}_{\tau,+}(\frac{\beta}{2})\tilde{V}_-(\beta)
-\tilde{V}_{\tau,+}(\beta)\tilde{V}_-(\frac{\beta}{2})}
{\beta^2\tilde{V}_{\tau,+}(\beta)}
\\
&=2\lim_{\beta\to+0}\frac{\left\{\tilde{V}_{\tau,+}(\frac{\beta}{2})\tilde{V}_-(\beta)
-\tilde{V}_{\tau,+}(\beta)\tilde{V}_-(\frac{\beta}{2})\right\}'''}
{\left\{\beta^2\tilde{V}_{\tau,+}(\beta)\right\}'''}=+\infty.
\end{align*}
Hence, $\lim_{\beta\to +0}\hat{a}_{\tau}^{*}(\beta)=0$ holds.

From now on we complete the proof by using the helper functions.
We first show that for each fixed $\tau \in (0,\tau_{\star}]$, $\tilde{\gamma}_{\tau,\beta}$ is a continuous function with respect to $\beta$. It is easy to see that $\hat{a}_{\tau}^{*}$ is continuous on $(0,\beta_{\star}]$. From this fact and $ \lim_{\beta\to +0}\hat{a}_{\tau}^{*}(\beta)=0$, we can find a constant $M_{\tau}>0$ so that $\hat{a}_{\tau}^{*}(\beta) \leq M_{\tau}$ for $\beta\in(0,\beta_{\star}]$. 
On the other hand, the magnitude relationship $\hat{V}_{\tau,\beta}^*(\Phi) \leq\tilde{V}_{\tau,\beta}(\Phi)$ shown above means that $\tilde{W}_{\tau,\beta}(\Phi) \leq \hat{W}_{\tau,\beta}^*(\Phi)$. From these two inequalities, we arrive at
\begin{gather}\label{Mtau1}
\tilde{W}_{\tau,\beta}(\Phi) \leq M_{\tau} \hat{U}^*(\Psi) \quad   \text{for $\beta\in(0,\beta_{\star}]$, \ $\Psi\in(0,1)$.}
\end{gather}
It is straightforward to see that for each $\Psi\in(0,1)$, $\tilde{W}_{\tau,\beta}$ is a continuous function with respect to $\beta$ on $(0,\beta_{*}]$. 
Furthermore, $\hat{U}^*$ is integrable over $(0,1)$.
Now recalling \eqref{GTB1} and applying the dominated convergence theorem with \eqref{Mtau1}, we deduce the continuity of $\tilde{\gamma}_{\tau,\beta}$.
Moreover, it is observed that
\begin{gather*}
0\leq\tilde{\gamma}_{\tau,\beta}
\leq \int_0^1\hat{W}_{\tau,\beta}^*(\Psi)\,d\Psi
=\hat{a}_{\tau}^{*}(\beta)\int_0^1 \hat{U}^*(\Psi)\,d\Psi 
\to 0 \quad \text{as $\beta \to +0$}. 
\end{gather*}
Thus $\lim_{\beta \to +0}\tilde{\gamma}_{\tau,\beta}=0$ holds. The proof is complete.
\end{proof}

We are now in a position to prove Corollary \ref{tnonunique3}.

\begin{proof}[Proof of Corollary \ref{tnonunique3}]
Lemma \ref{LemA} ensures that the problem \eqref{tVP3} has a solution $(F_{+},\Phi)$ of the problem \eqref{tVP3} for each pair $(\tau,\beta) \in (0,\tau_{\star}] \times (0,\beta_{\star}]$. 
Owing to $\Phi(0)=\beta$ and the forms \eqref{tfform+}--\eqref{tfform-} in Theorem \ref{texistence1}, 
if $(\tau,\beta)$ are different, then so are the solutions.
Therefore, as mentioned above, to complete the proof of Corollary \ref{tnonunique3},
it is sufficient to find infinite many pairs $(\tau,\beta) \in(0,\tau_\star]\times(0,\beta_\star]$ with \eqref{GTB1} for each $\tilde{\gamma}\in(0,\tilde{\gamma}_\star)$.

Let us fix $\tilde{\gamma}\in(0,\tilde{\gamma}_\star)$.
Then Lemma \ref{LemB} ensures that
\begin{gather*}
\lim_{\tau\to\tau_\star-0}\tilde{\gamma}_{\tau,\beta_\star}
=\tilde{\gamma}_{\tau_\star,\beta_\star}=\tilde{\gamma}_\star.
\end{gather*}
Therefore, there exists a constant $\delta>0$ such that 
if $\tau\in(\tau_\star-\delta,\tau_\star]$, then $\tilde{\gamma}<\tilde{\gamma}_{\tau,\beta_\star}\leq\tilde{\gamma}_\star$ holds.
For each $\tau\in(\tau_\star-\delta,\tau_\star]$, we can find a constant $\beta\in(0,\beta_\star]$ so that 
$\tilde{\gamma}_{\tau,\beta}=\tilde{\gamma}$ by applying the intermediate value theorem with $\lim_{\beta \to +0}\tilde{\gamma}_{\tau,\beta}=0$ in Lemma \ref{LemC}.
We note that there are infinite many choices of $\tau \in(\tau_\star-\delta,\tau_\star]$.
Thus we have infinite many pairs $(\tau,\beta)$ with \eqref{GTB1}.
The proof is complete.
\end{proof}

\medskip

\noindent
\textbf{Acknowledgements. } M. S. was supported by JSPS KAKENHI Grant Numbers 18K03364 and 21K03308. The authors would like to thank Professor Walter A. Strauss at Brown University for the helpful discussions.


\end{document}